\documentclass[12pt]{article}
\DeclareUnicodeCharacter{FFFD}{}
\usepackage{amsmath,amssymb,amsthm}
\usepackage{graphicx}
\usepackage{enumitem}
\usepackage{appendix}
\usepackage{tikz}
\usepackage{makecell}
\usepackage[all]{xy}
\usepackage{natbib}
\usepackage[a4paper,top=1in,bottom=1in,left=1.25in,right=1.25in]{geometry}
\usepackage{pifont}

\setlist[enumerate,1]{leftmargin=*}
\setlist[itemize]{leftmargin=1em}

\usepackage{hyperref}
\usepackage[capitalize,nameinlink]{cleveref}
\usepackage{aliascnt}

\theoremstyle{plain}
\newtheorem{theorem}{Theorem}[section]

\newaliascnt{lemma}{theorem}
\newtheorem{lemma}[lemma]{Lemma}
\aliascntresetthe{lemma}

\newaliascnt{proposition}{theorem}
\newtheorem{proposition}[proposition]{Proposition}
\aliascntresetthe{proposition}

\newaliascnt{corollary}{theorem}
\newtheorem{corollary}[corollary]{Corollary}
\aliascntresetthe{corollary}

\theoremstyle{definition}
\newaliascnt{definition}{theorem}
\newtheorem{definition}[definition]{Definition}
\aliascntresetthe{definition}

\newaliascnt{example}{theorem}

\aliascntresetthe{example}

\newaliascnt{conjecture}{theorem}

\aliascntresetthe{conjecture}

\theoremstyle{remark}
\newaliascnt{remark}{theorem}
\newtheorem{remark}[remark]{Remark}
\aliascntresetthe{remark}

\newaliascnt{question}{theorem}

\aliascntresetthe{question}

\crefname{theorem}{Theorem}{Theorems}       \Crefname{theorem}{Theorem}{Theorems}
\crefname{lemma}{Lemma}{Lemmas}             \Crefname{lemma}{Lemma}{Lemmas}
\crefname{proposition}{Proposition}{Propositions} \Crefname{proposition}{Proposition}{Propositions}
\crefname{corollary}{Corollary}{Corollaries} \Crefname{corollary}{Corollary}{Corollaries}
\crefname{definition}{Definition}{Definitions} \Crefname{definition}{Definition}{Definitions}
\crefname{example}{Example}{Examples}       \Crefname{example}{Example}{Examples}
\crefname{remark}{Remark}{Remarks}          \Crefname{remark}{Remark}{Remarks}
\crefname{question}{Question}{Questions}    \Crefname{question}{Question}{Questions}
\crefname{conjecture}{Conjecture}{Conjectures} \Crefname{conjecture}{Conjecture}{Conjectures}

\crefname{equation}{Equation}{Equations}    \Crefname{equation}{Equation}{Equations}
\crefname{figure}{Figure}{Figures}          \Crefname{figure}{Figure}{Figures}
\crefname{table}{Table}{Tables}             \Crefname{table}{Table}{Tables}

\renewenvironment{proof}{\noindent\textbf{Proof.}\ }{\hfill$\square$\par}

\usepackage{mathrsfs}

\title{\textbf{A Proof of the Jung–Kim Conjecture on the Centers of Walled Brauer Algebras}
}
\author{\textbf{Shunsuke Hirota}}
\date{\textit{\today}}

\begin{document}

\maketitle
\begin{abstract}
We give a unified account of the five families of Khovanov
arc algebras $K_n^m$, $K_n^{-\infty}$, $K_n^{+\infty}$,
$K_n^\infty$ and $K_\infty^\infty$ in the framework of
level-two tensor product categorifications, making explicit
a perspective familiar to specialists.
Our main application is a proof of the Jung--Kim conjecture:
over an algebraically closed field of characteristic zero,
the centre of every walled Brauer algebra is generated by
supersymmetric polynomials in its Jucys--Murphy elements.
We also give a homogeneous basis of the centre in terms of
arc diagrams, for the grading induced by $K_\infty^\infty$.
The proof combines a calculation of the arc-algebra centre
with an explicit comparison of central actions under the
Brundan--Stroppel Morita equivalence and Brundan-Kleshchev isomorphism.
Along the way, we explain the relation between
$K_n^{-\infty}$ and classical super duality, and identify
$K_n^{-\infty}$ and $K_n^{+\infty}$ as Ringel duals.
We also revisit mixed Schur--Weyl duality for general linear
Lie superalgebras within the same framework.
\end{abstract}
\tableofcontents

\section{Introduction}
\label{sec:introduction}
Khovanov arc algebras were introduced in connection with
Khovanov's categorification of the Jones polynomial and its
extension to tangles
\cite{khovanov2000categorification,khovanov2002functor}.
These algebras and their Koszul variants,
$K_n^m$, $K_n^{-\infty}$, $K_n^{+\infty}$, $K_n^\infty$
and $K_\infty^\infty$, have been studied extensively;
see, for example,
\cite{brundanstroppelI,brundanstroppelII,BrundanStroppelIII,
brundanstroppelIV,BrundanStroppelV,chenkhovanov2006tangle,
stroppel2009parabolic,bowman2024quiver,bowman2026faithful,
barmeierwang2026deformations}.

Brundan and Stroppel realized integral blocks of maximally
parabolic category $\mathcal O$ for $\mathfrak{gl}_{m+n}$,
with Levi $\mathfrak{gl}_n\oplus\mathfrak{gl}_m$, as module
categories over finite arc algebras $K_{n-t}^{m-t}$
\cite{BrundanStroppelIII}.

In parabolic category $\mathcal O$, tensor Casimir operators
give a degenerate affine Hecke action, as in the
Arakawa--Suzuki construction \cite{arakawa1998duality}.
On the arc-algebra side, geometric bimodule functors carry
a KLR action \cite{BrundanStroppelIII}.
The Brundan--Kleshchev isomorphism identifies the relevant
level-two cyclotomic degenerate affine Hecke and KLR
algebras \cite{brundan2009blocks}.
Comparing endomorphism algebras of projective-injective
objects and using the double centralizer property then
gives an equivalence of highest weight categories
\cite{BrundanStroppelIII,brundan2012orthogonal}.
This equivalence sends translation functors to geometric
bimodule functors, preserves their adjunction maps, and
intertwines the degenerate affine Hecke and KLR actions
under the Brundan--Kleshchev isomorphism.
An equivalence with these compatibilities is called a
\emph{strongly equivariant equivalence}.
This additional structure is essential to our proof of
the centre conjecture.

Brundan and Stroppel also identified blocks of
finite-dimensional integral $\mathfrak{gl}(m\mid n)$-modules
of atypicality $a$ with module categories over $K_a^\infty$
\cite{brundanstroppelIV}.
They showed that suitable idempotent corners of
finite arc algebras are graded Morita equivalent to blocks of
level-two cyclotomic KLR algebras
\cite{BrundanStroppelIII}.
For walled Brauer algebras, they obtained Morita
equivalences with suitable idempotent truncations of the
$K_\infty^\infty$-blocks, together with a walled Brauer
analogue of the Brundan--Kleshchev isomorphism
\cite{BrundanStroppelV}. Brundan--Stroppel used the arc-algebra grading to show
that $B_{r,s}(\delta)$ admits a Koszul grading whenever
$\Lambda_{r,s}=\dot\Lambda_{r,s}$, in particular whenever
$\delta\neq0$; see \cite[Corollary~1.2]{BrundanStroppelV}.
This highlights the importance of retaining the graded
structure when studying walled Brauer algebras.
Brundan, Savage and Webster extended this picture by
proving an isomorphism theorem for generalized cyclotomic
quotients of Heisenberg and Kac--Moody categories
\cite{brundansavagewebster2020heisenberg}. For higher-level analogues of walled Brauer algebras, see
\cite{wan2025parabolic,GaoRuiSong2023,Brundan2018Heisenberg}.

Losev and Webster proved a uniqueness theorem for tensor
product categorifications \cite{losev2015uniqueness},
extending Rouquier's uniqueness theorem for minimal
categorifications \cite{rouquier20082,rouquier2012quiver}.
As explained in Brundan's survey \cite{brundan2014schur},
this gives a unified way to recover the results of Brundan
and Stroppel. See also the survey \cite{brundan2012orthogonal}.
Brundan, Losev and Webster applied this uniqueness theory
to the super Kazhdan--Lusztig conjecture for category
$\mathcal O$ of general linear Lie superalgebras
\cite{brundan2017tensor}; see also the survey
\cite{brundan2014representations}.
Within the same framework, Leonard gave a new proof of
classical super duality \cite{leonard2020graded}, the
equivalence established by Cheng and Lam between appropriate
parabolic categories $\mathcal O$ for infinite-rank general
linear Lie algebras and Lie superalgebras
\cite{cheng2010irreducible}. Earlier work and further
developments include \cite{cheng2008super,cheng2012superduality}.

The purpose of this paper is to organize these connections
in the framework of level-two tensor product categorification,
making explicit a perspective familiar to specialists
\cite{brundan2014schur}. We treat the five families listed
above uniformly.
Our main application is a proof of the centre conjecture
for walled Brauer algebras
formulated by Sartori--Stroppel and Jung--Kim
\cite{sartoristroppel2015walled,JungKim2020}.
We also obtain a homogeneous basis of the centre in terms
of arc diagrams, for the grading induced by
$K_\infty^\infty$.

Walled Brauer algebras, also called oriented Brauer algebras,
generalize symmetric group algebras and describe centralizers
on mixed tensor spaces. They arose in the work of Turaev
and Koike \cite{turaev1989operator,koike1989decomposition}.

For symmetric groups, one can study the algebras
$\Bbbk\mathfrak S_d$ together by taking their direct sum.
For walled Brauer algebras, the direct sum of the
$B_{r,s}(\delta)$ omits natural morphisms between different
mixed tensor powers, such as evaluation
$\uparrow\downarrow\to1$ and coevaluation
$1\to\uparrow\downarrow$, where $1$ denotes the tensor unit.
The oriented Brauer category $OB_\delta$ retains these
morphisms, with
$\operatorname{End}_{OB_\delta}(\uparrow^r\downarrow^s)
\cong B_{r,s}(\delta)$.
It is the $\Bbbk$-linear rigid symmetric monoidal category
freely generated by an object of categorical dimension
$\delta$ and its dual.
We study its representations by regarding it as a locally
unital algebra. Its additive Karoubi envelope is Deligne's
category
$\mathcal D_\delta=\underline{\operatorname{Rep}}(GL_\delta)$;
see \cite{comes2012deligne}.

For comparison, Brundan proved that the centre of a
cyclotomic quotient of a degenerate affine Hecke algebra
consists of symmetric polynomials in its commuting
polynomial generators
\cite{brundan2008centers}.
For walled Brauer algebras the relevant condition is
supersymmetry. Related centre questions appear already in
\cite[Remark~2.6]{BrundanStroppelV}; the distinction between
symmetric and supersymmetric polynomials is explained in
\cite[Remark~7.3]{sartoristroppel2015walled}.
Sartori computed the centre in the degenerate affine
setting \cite{sartori2014degenerate}; Sartori and
Stroppel constructed central elements in the finite setting
and conjectured that the resulting elements generate the
centre \cite[Conjecture~7.4]{sartoristroppel2015walled}.
Jung and Kim reformulated the conjecture using their
normalization of the Jucys--Murphy elements and proved it
in the semisimple case \cite{JungKim2020}.
More recently, Chavli, De Visscher, Parker, Salmon and
Wilson proved it for $B_{r,1}(\delta)$
for every complex parameter $\delta$
\cite{chavli2024center}.

Our main result is the following theorem, proved as
Theorem~\ref{thm:jk-centre}.

\begin{theorem}\label{thm:intro-centre}
The conjecture of \cite[Conjecture~5.4]{JungKim2020} holds.
More precisely, let $\Bbbk$ be an algebraically closed field of
characteristic zero, let $r,s\geq0$, and let $\delta\in\Bbbk$.
For the Jucys--Murphy elements $L_1,\ldots,L_{r+s}$ of
Definition~\ref{def:jk-jucys-murphy}, in the normalization of
Jung and Kim, we have
\begin{equation}\label{eq:intro-jung-kim}
 Z(B_{r,s}(\delta))
 =\Bbbk\left[
 \sum_{a=1}^{r}L_a^k
 +(-1)^{k+1}\sum_{b=1}^{s}L_{r+b}^k
 \ \middle|\ k\geq1\right].
\end{equation}
Equivalently, the centre is the image of the algebra of
supersymmetric polynomials under evaluation at the $L_a$.
\end{theorem}

Our convention for supersymmetry is separate symmetry in
$x_1,\ldots,x_r$ and $y_1,\ldots,y_s$, together with the
condition that substitution $x_r=t$, $y_1=-t$ gives a
polynomial independent of $t$. The latter condition is
omitted when $rs=0$. In characteristic zero this algebra is
generated by the power sums used in
\eqref{eq:intro-jung-kim} \cite{Stembridge1985}.
Empty sums are zero, and the generated algebra is unital.

For integral $\delta$, we first compute the centre of
the finite arc-algebra truncation corresponding to
$B_{r,s}(\delta)$. Arc diagrams give a homogeneous basis
and hence its graded dimension.
The Morita equivalence of \cite{BrundanStroppelV} identifies
this centre with $Z(B_{r,s}(\delta))$.
We obtain that Morita equivalence by restricting a strongly
equivariant equivalence of the full categories. Its
compatibility with translation functors and their Hecke
actions makes the induced isomorphism of centres explicit.

On the arc side, the circle-orientation maps defining the
KLR action are related to the degree-two central elements
$X_j$. We form particular linear combinations $H_k$ of
these elements. On the oriented Brauer side, the
corresponding degenerate affine Hecke action is given by
the Jucys--Murphy elements, with the prescribed signs on
the dual tensor factors. The strongly equivariant
equivalence and the Brundan--Kleshchev isomorphism
therefore determine the images of the
supersymmetric power sums: each image is $H_k$ plus an
explicit scalar on each block
(Proposition~\ref{prop:jk-explicit-formula}).
These images separate the blocks and generate the
arc-algebra centre, proving the conjecture.
The nonintegral case follows from the semisimple result
of Jung and Kim \cite[Theorem~3.5]{JungKim2020}.

Table~\ref{tab:arc-algebra-families} summarizes the connection
with type A Lie theory. The tensor factors come from four
kinds of minuscule $\mathfrak{sl}(\mathbb Z)$-modules:
finite exterior powers of the natural module $V^0$ and its
restricted dual $V^1$, and the highest and lowest weight
semi-infinite wedges $L^+(\boldsymbol\varpi_p)$ and
$L^-(-\boldsymbol\varpi_p)$.
The table includes two realizations of finite arc blocks
and, in the last row, the abelian envelope of Deligne's
category.

The third column lists full categories from type A Lie
theory. The categorical action is defined on the sum of
blocks. In the first row,
$t$ is the number of $\times$ vertices, with
$0\leq t\leq\min(m,n)$. In the second row,
$a,b\geq0$ and $a-b=p-q$.
For the two one-sided infinite families, $0\leq a\leq n$;
for $\mathcal F(m\mid n)$, $a$ is the atypicality of a
block and $0\leq a\leq\min(m,n)$.
The last column refers to the block category.
Self-injectivity follows from the symmetric structure in
\cite{brundanstroppelI}.

\clearpage
\begin{table}[htbp]
\centering
\refstepcounter{table}\label{tab:arc-algebra-families}
\small
\setlength{\tabcolsep}{3pt}
\renewcommand{\arraystretch}{1.35}
\begin{tabular}{@{}
 >{\raggedright\arraybackslash}p{0.14\textwidth}
 >{\raggedright\arraybackslash}p{0.29\textwidth}
 >{\raggedright\arraybackslash}p{0.32\textwidth}
 >{\raggedright\arraybackslash}p{0.19\textwidth}@{}}
\hline
Arc algebra & $\mathfrak{sl}(\mathbb Z)$-TPC& Type A Lie theory & Highest weight type\\
\hline
$K_{n-t}^{m-t}$
& $\bigwedge^n V^0\otimes\bigwedge^m V^0$
& Integral parabolic category $\mathcal O$ for
  $\mathfrak{gl}_{m+n}$ with Levi
  $\mathfrak{gl}_n\oplus\mathfrak{gl}_m$
& Finite\\
$K_a^b$
& $L^+(\boldsymbol\varpi_p)\otimes L^+(\boldsymbol\varpi_q)$
&
& Finite\\
$K_a^{-\infty}$
& $\bigwedge^n V^0\otimes L^+(\boldsymbol\varpi_0)$
& Infinite-rank parabolic categories $\mathcal O^{++}$ for
  $\mathfrak{gl}(n+\infty)$ and $\mathfrak{gl}(n\mid\infty)$,
  related by classical super duality
& Essentially finite; upper finite weight poset\\
$K_a^{+\infty}$
& $L^+(\boldsymbol\varpi_0)\otimes\bigwedge^n V^0$
&
& Essentially finite and lower finite; self-injective\\
$K_a^\infty$
& $\bigwedge^m V^0\otimes\bigwedge^n V^1$
& Category $\mathcal F(m\mid n)$ of finite-dimensional
  integral $\mathfrak{gl}(m\mid n)$-modules
& Essentially finite; self-injective\\
$K_\infty^\infty$
& $L^-(-\boldsymbol\varpi_{-\delta})\otimes
   L^+(\boldsymbol\varpi_0)$
& Locally finite-dimensional modules over
  $OB_\delta$, $\delta\in\mathbb Z$
& Upper finite\\
& $L^+(\boldsymbol\varpi_0)\otimes
   L^-(-\boldsymbol\varpi_{-\delta})$
& Abelian envelope $\mathcal V_\delta$ of
  $\mathcal D_\delta=\underline{\operatorname{Rep}}(GL_\delta)$,
  $\delta\in\mathbb Z$
& Lower finite\\
\hline
\end{tabular}
\end{table}

The last row uses the abelian envelope constructed by
Entova-Aizenbud, Hinich and Serganova
\cite{entovahinichserganova2020deligne}, with its tensor
product categorification described by Entova-Aizenbud
\cite{entova2018categorical}. For a general construction of
abelian envelopes covering this example, see Coulembier
\cite{coulembier2021envelopes}.
For related Ringel dualities between Brauer or walled
Brauer algebras and representations of supergroups, see
Coulembier \cite{coulembier2021ringel}.
The last two rows differ by reversing the tensor factors.
The two categories are related by semi-infinite Ringel
duality \cite[Section~6.6]{brundanstroppelsemiinfinite}.
In that section, Brundan and Stroppel also announce
a description of the blocks of the abelian envelope using
Khovanov's arc coalgebra. We therefore leave the arc algebra
entry in this row blank.

We now describe the contents of the paper.
Sections~2--4 are primarily expository: their purpose is
to organize established constructions in a common
framework. Some intermediate formulations and arguments
are given here in a form not explicitly recorded in the
existing literature. Section~3.4, Section~4.2 and Appendix~B are not needed
for the proof of the main theorem.

Section~\ref{sec:uniqueness-level-two-tpc} recalls tensor
product categorification at level two and the
Brundan--Kleshchev isomorphism. For uniqueness, our only input
is the theorem for finite intervals
\cite{losev2015uniqueness}.
Lemma~\ref{lem:finite-equivariant-comparison} supplies
compatible natural isomorphisms between the strongly
equivariant equivalences on those intervals. This lets us
pass to infinite intervals without using the stable module
categories of \cite{brundan2017tensor}, and treats the
semi-infinite cases together with the other level-two
tensor products. For other proofs of uniqueness, see
\cite{entova2018categorical} in the Deligne case and
\cite{wan2025parabolic} for more general Fock space tensor
products. We work at level two, although the same method
applies to the higher-level tensor products considered in
\cite{brundan2017tensor,leonard2020graded}.
It does not cover the more general Fock space tensor
products of \cite{wan2025parabolic}.

Section~\ref{sec:Khovanov-arc-algebras} recalls the arc
algebras in a common notation. In
Subsection~\ref{subsec:arc-tensor-product-categorifications},
we construct their level-two tensor product
categorifications.
We use the KLR action on geometric bimodule functors from
\cite{BrundanStroppelIII}, with the normalization needed
later to identify central elements.
Reversing the two tensor factors identifies
$K_n^{-\infty}$ and $K_n^{+\infty}$ as Ringel duals in
the essentially finite setting. Combining our arc
realization with Leonard's tensor product categorification
\cite[Theorem~4.23]{leonard2020graded} also identifies
the relevant blocks of the infinite-rank parabolic
categories $\mathcal O^{++}$ with module categories over
$K_a^{-\infty}$. These are the categories appearing in
classical super duality
\cite{cheng2008super,cheng2010irreducible,leonard2020graded}.

Subsection~3.4 constructs linear tilting coresolutions of
standard modules and, dually, linear tilting resolutions
of costandard modules for blocks of finite defect.
In particular, this includes $K_n^{-\infty}$ and
$K_n^{+\infty}$. These constructions are motivated by
Mazorchuk's theory of balanced gradings
\cite{mazorchuk2010balanced} and suggest a semi-infinite
extension of that theory. We expect the Koszul duals of
$K_n^{-\infty}$ and $K_n^{+\infty}$ to be related by a
nontrivial semi-infinite Ringel duality.
The defect-one examples in
Appendix~\ref{sec:defect-one-examples} illustrate this
expectation; the general question is left for future work.
For Koszulity in the semi-infinite setting, see
\cite{heidersdorfnehmestroppel2025koszulity}.

Subsection~\ref{subsec:representations-oriented-Brauer}
introduces the oriented Brauer category $OB_\delta$,
constructs its tensor product categorification, and
recovers the relevant results of \cite{BrundanStroppelV}
from uniqueness. This includes the oriented Brauer
analogue of the Brundan--Kleshchev isomorphism and its
compatibility with the categorical actions.

Subsection~\ref{subsec:mixed-tensors} begins with Deligne's
category and its relation to mixed tensors for general
linear Lie superalgebras
\cite{BrundanStroppelV,comes2012deligne}.
Yoneda, diagram reflection and the arc equivalence identify
$\mathcal D_\delta$ with the category of finitely generated
projective left $K(\delta)$-modules.
For $\delta=m-n$, evaluation then gives a full strongly
equivariant functor from this projective category to
$\mathcal F(m\mid n)$
(Theorem~\ref{thm:strongly-equivariant-evaluation}).
We use this functor to recover classification of indecomposable mixed tensor
summands.  This result of Brundan--Stroppel, which we have revisited
here, is fundamental to the study of
$\mathfrak{gl}(m\mid n)$; see
\cite{heidersdorf2017mixed,heidersdorfweissauer2021pieri,
heidersdorfweissauer2023classical}.

Section~5 proves Theorem~\ref{thm:intro-centre} by the
method outlined above.
Appendix~\ref{sec:semi-infinite-highest-weight-categories}
recalls the parts of Brundan--Stroppel's theory of
semi-infinite highest weight categories needed in the
paper \cite{brundanstroppelsemiinfinite}.
Appendix~\ref{sec:defect-one-examples} treats
$K_1^m$, $K_1^\infty$, $K_1^{-\infty}$ and
$K_1^{+\infty}$ explicitly. The last three have 
Koszul duals  which is infinite version of typeA preprojective algebras:
one on the doubly infinite line and two on the
one-sided line, with different boundary relations.
The latter two give a lower finite and an upper finite
highest weight category which are Ringel dual to each
other and are not essentially finite.

\medskip
\noindent\textbf{Acknowledgements.}
The author thanks Shun-Jen Cheng and Linliang Song for
discussions that influenced the formulation of the arc
algebra framework used in this paper.
This work was supported by JSPS KAKENHI Grant Number
JP25KJ1664.

\medskip
\noindent\textbf{Use of artificial intelligence.}
The author used ChatGPT 6 Astra to assist with English
translation, the organization of the paper, and the
development of some proofs. In particular, the idea behind
Lemma~\ref{lem:finite-equivariant-comparison} and the argument
avoiding stable module categories originated in suggestions
from Astra. Astra also proposed the initial common notation
for the five families of arc algebras in Section~3.
For Section~5, the author asked whether the Jung--Kim
conjecture could be proved using the Morita equivalence
between $B_{r,s}(\delta)$ and an arc-algebra truncation.
After checking examples, the author developed the argument
through discussions with Astra. Astra proposed the formula
for the arc-algebra centre and the use of an explicit
isomorphism between the two centres, and contributed the
initial calculations of that isomorphism.
The author determined the overall structure of the paper
and the arrangement of its theorems, paying particular
attention to the order in which the arguments are presented.
The author takes full responsibility for the statements
and proofs in this paper.

\section{Uniqueness of level-two tensor product categorifications}
\label{sec:uniqueness-level-two-tpc}

In this section, we follow Brundan--Losev--Webster
\cite{brundan2017tensor} and Leonard \cite{leonard2020graded}.

\subsection{Semi-infinite wedge modules}
\label{subsec:semi-infinite-wedges}

Let $\mathfrak{gl}(\mathbb Z)$ be the Lie algebra of finitary matrices
with rows and columns indexed by $\mathbb Z$, with standard matrix units
$e_{i,j}$. Let
\[
\mathfrak t
:=
\bigoplus_{i\in\mathbb Z}\Bbbk e_{i,i},
\qquad
\mathfrak n^\pm
:=
\bigoplus_{\substack{i,j\in\mathbb Z\\ \pm(j-i)>0}}
\Bbbk e_{i,j},
\qquad
\mathfrak b^\pm
:=
\mathfrak t\oplus\mathfrak n^\pm.
\]

For $i\in\mathbb Z$, define
$\boldsymbol\varepsilon_i\in\mathfrak t^*$ by
\[
\boldsymbol\varepsilon_i(e_{j,j})
=
\delta_{i,j},
\]
and set
\[
\boldsymbol\alpha_i
:=
\boldsymbol\varepsilon_i-\boldsymbol\varepsilon_{i+1}.
\]

Let $V=V^0$ be the natural $\mathfrak{gl}(\mathbb Z)$-module with basis
$\{v_i^0\mid i\in\mathbb Z\}$ and action
\[
e_{i,j}v_k^0
=
\delta_{j,k}v_i^0.
\]
Let
\[
V^1:=V^\vee
\]
be the restricted dual of $V^0$, with basis
$\{v_i^1\mid i\in\mathbb Z\}$ defined by
\[
v_i^1(v_j^0)
=
(-1)^i\delta_{i,j},
\]
and action
\[
e_{i,j}v_k^1
=
(-1)^{i+j+1}\delta_{i,k}v_j^1.
\]
Thus
\[
\operatorname{wt}(v_i^c)
=
(-1)^c\boldsymbol\varepsilon_i
\qquad
(c\in\{0,1\}).
\]
Let $\mathscr P$ be the set of partitions, and let $\nu'$ denote the
conjugate partition of $\nu\in\mathscr P$. For $c\in\{0,1\}$, set
\[
\nu^{(0)}:=\nu,
\qquad
\nu^{(1)}:=\nu'.
\]
For a box $b$ in row $a$ and column $d$, set
$\operatorname{ct}(b):=d-a$.

For $s\in\mathbb Z$, $c\in\{0,1\}$, and $\nu\in\mathscr P$, define
\[
v_\nu^{c,+,s}
:=
v_{s+\nu^{(c)}_1}^c
\wedge
v_{s+\nu^{(c)}_2-1}^c
\wedge
v_{s+\nu^{(c)}_3-2}^c
\wedge\cdots,
\]
and
\[
v_\nu^{c,-,s}
:=
v_{s+1-\nu^{(c)}_1}^c
\wedge
v_{s+2-\nu^{(c)}_2}^c
\wedge
v_{s+3-\nu^{(c)}_3}^c
\wedge\cdots.
\]
Set
\[
\bigwedge^{+\infty,s}V^c
:=
\bigoplus_{\nu\in\mathscr P}
\Bbbk v_\nu^{c,+,s},
\qquad
\bigwedge^{-\infty,s}V^c
:=
\bigoplus_{\nu\in\mathscr P}
\Bbbk v_\nu^{c,-,s}.
\]
The usual action on exterior powers extends to these semi-infinite
wedges since every element of $\mathfrak{gl}(\mathbb Z)$ is finitary.

Set
\[
\boldsymbol\lambda_s^+
:=
\sum_{j\leq s}\boldsymbol\varepsilon_j,
\qquad
\boldsymbol\lambda_s^-
:=
\sum_{j>s}\boldsymbol\varepsilon_j.
\]
These infinite sums define elements of $\mathfrak t^*$ since every
element of $\mathfrak t$ has finite support. For $\nu\in\mathscr P$,
we have
\[
\operatorname{wt}(v_\nu^{c,\pm,s})
=
(-1)^c\boldsymbol\lambda_s^\pm
\mp
(-1)^c
\sum_{b\in\nu}
\boldsymbol\alpha_{s\pm(-1)^c\operatorname{ct}(b)}.
\]

For $i\in\mathbb Z$, put
\[
E_i:=e_{i,i+1},
\qquad
F_i:=e_{i+1,i},
\qquad
H_i:=[E_i,F_i].
\]

For an addable or removable box $\Box$ of $\nu$, write
$\nu+\Box$ or $\nu-\Box$ for the resulting partition. For
$c\in\{0,1\}$, use the convention
\[
\nu\pm(-1)^c\Box
:=
\begin{cases}
\nu\pm\Box,&c=0,\\
\nu\mp\Box,&c=1.
\end{cases}
\]
Expressions involving a partition which does not exist are understood
to be zero.

The action on $\bigwedge^{\pm\infty,s}V^c$ is given by
\[
F_i v_\nu^{c,\pm,s}
=
\begin{cases}
v_{\nu\pm(-1)^c\Box}^{c,\pm,s},
&
i=s\pm(-1)^c\operatorname{ct}(\Box)
\text{ and }
\nu\pm(-1)^c\Box\in\mathscr P,\\
0,&\text{otherwise},
\end{cases}
\]
and
\[
E_i v_\nu^{c,\pm,s}
=
\begin{cases}
v_{\nu\mp(-1)^c\Box}^{c,\pm,s},
&
i=s\pm(-1)^c\operatorname{ct}(\Box)
\text{ and }
\nu\mp(-1)^c\Box\in\mathscr P,\\
0,&\text{otherwise}.
\end{cases}
\]

Set
\[
\mathfrak{sl}(\mathbb Z)
:=
\ker\bigl(
\operatorname{tr}:\mathfrak{gl}(\mathbb Z)\to\Bbbk
\bigr)
=
\langle E_i,F_i\mid i\in\mathbb Z\rangle,
\qquad
\mathfrak h
:=
\bigoplus_{i\in\mathbb Z}\Bbbk H_i.
\]
From now on, every $\mathfrak{gl}(\mathbb Z)$-module will also be
regarded as an $\mathfrak{sl}(\mathbb Z)$-module by restriction.

We use the same notation $\boldsymbol\varepsilon_i$ and
$\boldsymbol\alpha_i$ for their restrictions to $\mathfrak h$.
For $i\in\mathbb Z$, let $\boldsymbol\varpi_i\in\mathfrak h^*$ be
defined by
\[
\boldsymbol\varpi_i(H_j)=\delta_{i,j}.
\]
Then
\[
\boldsymbol\varepsilon_i
=
\boldsymbol\varpi_i-\boldsymbol\varpi_{i-1},
\qquad
\boldsymbol\alpha_i
=
2\boldsymbol\varpi_i
-\boldsymbol\varpi_{i-1}
-\boldsymbol\varpi_{i+1}.
\]
In particular,
\[
\boldsymbol\lambda_s^\pm|_{\mathfrak h}
=
\pm\boldsymbol\varpi_s.
\]
Thus $\boldsymbol\varpi_s$ is the weight of the shifted vacuum
\[
v_s^0\wedge v_{s-1}^0\wedge v_{s-2}^0\wedge\cdots,
\]
while $-\boldsymbol\varpi_s$ is the weight of
\[
v_{s+1}^0\wedge v_{s+2}^0\wedge v_{s+3}^0\wedge\cdots.
\]

A $\mathfrak{gl}(\mathbb Z)$-module $M$ is called a weight module if
\[
M=\bigoplus_{\lambda\in\mathfrak t^*}M_\lambda,
\qquad
M_\lambda
=
\{v\in M\mid hv=\lambda(h)v
\text{ for all }h\in\mathfrak t\}.
\]
A non-zero weight vector is called a highest weight vector if it is
annihilated by $\mathfrak n^+$, and a lowest weight vector if it is
annihilated by $\mathfrak n^-$.

For $\lambda\in\mathfrak t^*$, let $\Bbbk_\lambda^\pm$ be the
one-dimensional $\mathfrak b^\pm$-module of weight $\lambda$ on which
$\mathfrak n^\pm$ acts trivially. Define
\[
M^\pm(\lambda)
:=
U(\mathfrak{gl}(\mathbb Z))
\otimes_{U(\mathfrak b^\pm)}
\Bbbk_\lambda^\pm.
\]
Their unique irreducible quotients are denoted by $L^\pm(\lambda)$.

\begin{proposition}\cite{cheng2008super}
For every $s\in\mathbb Z$, as $\mathfrak{gl}(\mathbb Z)$-modules,
\[
\bigwedge^{\pm\infty,s}V^0
\cong
L^\pm(\boldsymbol\lambda_s^\pm)
\cong
\bigwedge^{\mp\infty,s}V^1
\otimes
\Bbbk_{\operatorname{tr}},
\]
where $\Bbbk_{\operatorname{tr}}$ denotes the one-dimensional
$\mathfrak{gl}(\mathbb Z)$-module afforded by the trace character.
The isomorphism between the first and third modules is given by
\[
v_\nu^{0,\pm,s}
\longmapsto
v_\nu^{1,\mp,s}\otimes1.
\]

Hence, as $\mathfrak{sl}(\mathbb Z)$-modules,
\[
\bigwedge^{\pm\infty,s}V^0
\cong
L^\pm(\pm\boldsymbol\varpi_s)
\cong
\bigwedge^{\mp\infty,s}V^1,
\qquad
v_\nu^{0,\pm,s}
\longmapsto
v_\nu^{1,\mp,s}.
\]
\end{proposition}

For $\lambda,\mu\in\mathfrak h^*$, write
\[
\lambda\unrhd\mu
\quad\Longleftrightarrow\quad
\lambda-\mu
\in
\bigoplus_{i\in\mathbb Z}
\mathbb Z_{\geq0}\boldsymbol\alpha_i.
\]

For $c\in\{0,1\}$ and $m\in\mathbb N$, let
\[
\Xi_m
:=
\left\{
(\lambda_1,\ldots,\lambda_m)\in\mathbb Z^m
\ \middle|\
\lambda_1>\cdots>\lambda_m
\right\}.
\]
We identify $\Xi_m$ with the monomial basis of $\bigwedge^mV^c$ by
\[
\lambda
\longleftrightarrow
v_\lambda^c
:=
v_{\lambda_1}^c\wedge\cdots\wedge v_{\lambda_m}^c.
\]
We also identify $\mathscr P$ with the monomial basis of
$L^\pm(\pm\boldsymbol\varpi_s)$ by
\[
\nu
\longleftrightarrow
v_\nu^{0,\pm,s}.
\]

We equip these indexing sets with the orders induced by their weights.
Thus, for $\lambda,\mu\in\Xi_m$,
\[
\lambda\unrhd_{\wedge^mV^c}\mu
\quad\Longleftrightarrow\quad
\operatorname{wt}(v_\lambda^c)
\unrhd
\operatorname{wt}(v_\mu^c),
\]
and, for $\nu,\mu\in\mathscr P$,
\[
\nu\unrhd_{L^\pm(\pm\boldsymbol\varpi_s)}\mu
\quad\Longleftrightarrow\quad
\operatorname{wt}(v_\nu^{0,\pm,s})
\unrhd
\operatorname{wt}(v_\mu^{0,\pm,s}).
\]
Explicitly,
\[
\lambda\unrhd_{\wedge^mV^c}\mu
\quad\Longleftrightarrow\quad
(-1)^c\lambda_a\leq(-1)^c\mu_a
\quad
\text{for all }a,
\]
while
\[
\nu\unrhd_{L^\pm(\pm\boldsymbol\varpi_s)}\mu
\quad\Longleftrightarrow\quad
\begin{cases}
\nu\subseteq\mu,&\text{for the sign $+$},\\
\nu\supseteq\mu,&\text{for the sign $-$}.
\end{cases}
\]

Let $X$ and $Y$ each be of the form
\[
\bigwedge^mV^c
\qquad\text{or}\qquad
L^\pm(\pm\boldsymbol\varpi_s).
\]
We equip the product of their indexing sets with the
\emph{inverse dominance order} by declaring
\[
(\lambda,\mu)
\preceq_{X\otimes Y}
(\lambda',\mu')
\]
if
\[
\operatorname{wt}(v_\lambda)
+
\operatorname{wt}(v_\mu)
=
\operatorname{wt}(v_{\lambda'})
+
\operatorname{wt}(v_{\mu'})
\]
and
\[
\operatorname{wt}(v_\lambda)
\unrhd
\operatorname{wt}(v_{\lambda'}).
\]
Equivalently, under the equality of total weights, the second
condition may be written as
\[
\operatorname{wt}(v_\mu)
\unlhd
\operatorname{wt}(v_{\mu'}).
\]

\subsection{Degenerate affine Hecke algebras}
\label{subsec:degenerate-affine-Hecke}

Throughout this subsection, $\Bbbk$ is algebraically closed of
characteristic zero. We fix the Hecke-algebra conventions used
for categorical actions.

\begin{definition}\label{def:degenerate-affine-hecke-algebra}
\cite{Lusztig1989,Drinfeld1986}
For $d\geq1$, the degenerate affine Hecke algebra $H_d$ over
$\Bbbk$ is the unital associative algebra generated by
$x_a$ for $1\leq a\leq d$ and $t_r$ for $1\leq r<d$,
with relations
\[
\begin{gathered}
 x_ax_b=x_bx_a,\qquad t_r^2=1,\qquad
 t_rt_s=t_st_r\quad(|r-s|>1),\\
 t_rt_{r+1}t_r=t_{r+1}t_rt_{r+1},\qquad
 t_rx_a=x_at_r\quad(a\neq r,r+1),\\
 t_rx_{r+1}-x_rt_r=1.
\end{gathered}
\]
Here $a,b$ range over $\{1,\ldots,d\}$ and $r,s$ over
$\{1,\ldots,d-1\}$; the braid relation is imposed for
$1\leq r\leq d-2$.
Set $H_0=\Bbbk$.
\end{definition}

Fix a nonempty interval
$I\subseteq\mathbb Z$, and let $\mathfrak{sl}_I$ have Chevalley
generators $e_i,f_i,h_i$ for $i\in I$.
We work in either of the following cases:
\begin{enumerate}
\item[(i)] $\mathcal C$ is an essentially small Hom-finite
$\Bbbk$-linear abelian category of finite-length objects, with
$\operatorname{End}(L)=\Bbbk$ for every simple object $L$.
We do not assume enough projectives or injectives.
\item[(ii)] $\mathcal C=\operatorname{mod}_{\mathrm{lfd}}\!\text{-}A$,
where $A=\bigoplus_{a,b}e_aAe_b$ is a pointed locally unital
algebra with finite-dimensional corners, as in
Appendix~\ref{sec:semi-infinite-highest-weight-categories}.
Thus $\dim Me_a<\infty$ for every $M\in\mathcal C$ and $a$.
\end{enumerate}
Write $\mathcal C^{\mathrm{proj}}$ for the projective objects
in case (i), and the finitely generated projective objects
in case (ii).

A family $(G_i)$ is \emph{locally finite} if only finitely many
$G_iM$ are nonzero for each $M$ in case (i), or only finitely
many $(G_iM)e_a$ are nonzero for each $M,a$ in case (ii).
An endomorphism $u$ of $M$ is \emph{locally nilpotent} if
$M=\bigcup_{N\geq1}\ker u^N$. This means nilpotence in case (i),
or nilpotence on each $Me_a$ in case (ii).
For a natural endomorphism, this is required on every object.

\begin{definition}\label{def:categorical-slI-action}
\cite[Definition 2.6]{brundan2017tensor}
A \emph{categorical $\mathfrak{sl}_I$-action} on $\mathcal C$
consists of locally finite families of $\Bbbk$-linear functors
$F_i,E_i$, fixed adjunctions $F_i\dashv E_i$, and natural
transformations $x\in\operatorname{End}(F)$ and
$t\in\operatorname{End}(F^2)$, where
$F=\bigoplus_{i\in I}F_i$ and $E=\bigoplus_{i\in I}E_i$,
satisfying:
\begin{enumerate}
\item[(CA1)] The transformation $x$ preserves each $F_i$, and
$(x-i)|_{F_i}$ is locally nilpotent.
\item[(CA2)] For every $d\geq0$, there is a unital
$\Bbbk$-algebra homomorphism
\[
 H_d\longrightarrow\operatorname{End}(F^d),\qquad
 \begin{aligned}
 x_j&\longmapsto F^{d-j}xF^{j-1} &&(1\leq j\leq d),\\
 t_j&\longmapsto F^{d-j-1}tF^{j-1} &&(1\leq j<d).
 \end{aligned}
\]
Here $\operatorname{End}(F^d)$ is the algebra of natural
endomorphisms of $F^d$, and $F^0=\operatorname{Id}_{\mathcal C}$.
\item[(CA3)] Each $F_i$ is also a right adjoint of $E_i$.
\item[(CA4)] The functors $E_i,F_i$ induce an integrable
$\mathfrak{sl}_I$-module structure on
\[
 \begin{cases}
 \Bbbk\otimes_{\mathbb Z}K_0(\mathcal C),
     &\text{in case (i)},\\
 \Bbbk\otimes_{\mathbb Z}
 K_0^{\mathrm{split}}(\mathcal C^{\mathrm{proj}}),
     &\text{in case (ii)},
 \end{cases}
\]
with $e_i=[E_i]$ and $f_i=[F_i]$.
The classes of simple objects in case (i), and of indecomposable
projective objects in case (ii), are weight vectors.
\end{enumerate}
\end{definition}

By biadjunction, all $E_i,F_i$ are exact and preserve projectives.
In case (ii), their adjoints preserve directed unions, so
adjunction shows that $E_i(e_aA)$ and $F_i(e_aA)$ are finitely
generated. Thus they preserve $\mathcal C^{\mathrm{proj}}$.
Locally finite sums are also products, so the chosen adjunctions
give $F\dashv E$.

\begin{definition}\label{def:strongly-equivariant-slI}
\cite[Definition 2.7]{brundan2017tensor}
Let $\mathcal C,\mathcal C'$ be as above, with categorical
$\mathfrak{sl}_I$-actions $(F,E,x,t)$ and $(F',E',x',t')$.
Write $\eta,\varepsilon$ and $\eta',\varepsilon'$ for the units
and counits of the fixed adjunctions $F\dashv E$ and $F'\dashv E'$.
A $\Bbbk$-linear functor $G:\mathcal C\to\mathcal C'$
preserving the locally finite sums is \emph{strongly equivariant}
if it is equipped with an isomorphism
$\zeta:F'G\xrightarrow{\sim}GF$ satisfying:
\begin{enumerate}
\item[(E1)] The adjunction mate
\[
 E'G\varepsilon\circ E'\zeta E\circ\eta'GE:
 GE\longrightarrow E'G
\]
is an isomorphism.
\item[(E2)] $\zeta\circ x'G=Gx\circ\zeta$.
\item[(E3)]
$\zeta F\circ F'\zeta\circ t'G
=Gt\circ\zeta F\circ F'\zeta$.
\end{enumerate}
If $G$ is an equivalence, then {\rm (E1)} is automatic.
We then call $G$ a \emph{strongly equivariant equivalence}.
\end{definition}

\subsection{Finite intervals}
\label{subsec:finite-intervals}
% Replacement for 7.tex. Finite uniqueness is an explicit external input.

Let $\Bbbk$ be algebraically closed of characteristic zero.
Fix a finite nonempty interval $I=\{\ell,\ldots,u\}\subset\mathbb Z$,
put $I^+=I\cup(I+1)$ and $N=|I^+|$, and let $V_I$ have basis
$\{v_a\mid a\in I^+\}$, with
$E_iv_a=\delta_{a,i+1}v_i$ and $F_iv_a=\delta_{a,i}v_{i+1}$.
For $0\leq d\leq N$, put
\[
 \Xi_{I,d}=\{(\lambda_1,\ldots,\lambda_d)\in(I^+)^d
                   \mid\lambda_1>\cdots>\lambda_d\},
 \qquad v_\lambda=v_{\lambda_1}\wedge\cdots\wedge v_{\lambda_d}.
\]
For $d=0$ use the empty tuple and $v_\varnothing=1$.
Fix $0\leq m,n\leq N$ and equip $\Xi_{I,m}\times\Xi_{I,n}$
with the inverse dominance order defined above, restricted to $I$.

\begin{definition}\label{def:finite-level-two-tpc}
An $\mathfrak{sl}_I$-tensor product categorification of
$\bigwedge^mV_I\otimes\bigwedge^nV_I$ is a finite highest
weight category $\mathcal C$ with this weight poset and a
categorical $\mathfrak{sl}_I$-action as in
Definition~\ref{def:categorical-slI-action}, such that:
\begin{enumerate}
\item[(TPC1)] Every $E_i,F_i$ preserves $\mathcal C^\Delta$.
\item[(TPC2)] The isomorphism
\[
 \Bbbk\otimes_{\mathbb Z}K_0(\mathcal C^\Delta)
 \xrightarrow{\sim}\bigwedge^mV_I\otimes\bigwedge^nV_I,
 \qquad [\Delta(\lambda,\mu)]\longmapsto v_\lambda\otimes v_\mu,
\]
intertwines $[E_i],[F_i]$ with the Chevalley generators.
\end{enumerate}
Here $K_0(\mathcal C^\Delta)$ uses the inherited exact structure;
see \cite[Definition 2.10]{brundan2017tensor}.
\end{definition}

Let $\mathcal C$ be such a categorification.
Let $L$ be the simple object labelled by the highest tensor monomial
$(v_{\ell+m-1}\wedge\cdots\wedge v_\ell)
 \otimes(v_{\ell+n-1}\wedge\cdots\wedge v_\ell)$,
with empty factors interpreted as $1$.
For $i\in I$, put
$a_i=\delta_{i,\ell+m-1}+\delta_{i,\ell+n-1}$.
For $d\geq1$, define
\[
 H_d^{m,n}:=H_d\Big/\Big\langle
       \prod_{i\in I}(x_1-i)^{a_i}\Big\rangle.
\]
In this finite-dimensional cyclotomic quotient, let
$1_{\mathbf i}$ be the simultaneous generalized eigenspace
idempotent for $x_1,\ldots,x_d$ with eigenvalues
$\mathbf i=(i_1,\ldots,i_d)$. Put
$e_{I,d}=\sum_{\mathbf i\in I^d}1_{\mathbf i}$ and
$H_{I,d}^{m,n}=e_{I,d}H_d^{m,n}e_{I,d}$, and set
$H_{I,0}^{m,n}=\Bbbk$.
Thus $H_{I,d}^{m,n}$ retains precisely the blocks whose
residues lie in $I$; see \cite[\S2.3]{brundan2017tensor}.

\begin{theorem}\label{thm:finite-hecke-endomorphisms}
\cite{BrundanKleshchev2008}
\cite[Theorem 2.14]{brundan2017tensor}
For every $d\geq0$, the object $F^dL$ is projective-injective,
and the categorical Hecke action induces an algebra isomorphism
\begin{equation}\label{eq:finite-hecke-endomorphisms}
 H_{I,d}^{m,n}\xrightarrow{\sim}
 \operatorname{End}_{\mathcal C}(F^dL).
\end{equation}
In particular, $H_d\to\operatorname{End}_{\mathcal C}(F^dL)$
is surjective.
\end{theorem}

We use the following theorem as an external input.

\begin{theorem}\label{thm:finite-level-two-tpc-uniqueness}
\cite[Theorem 2.12, finite-interval case]{brundan2017tensor}
\cite[Theorem 6.1]{losev2015uniqueness}
Such categorifications exist. Any two are related by a strongly
equivariant equivalence $G:\mathcal C\xrightarrow{\sim}\mathcal C'$
with $GL(\lambda,\mu)\cong L'(\lambda,\mu)$ for every label.
Consequently $G$ also preserves the labelled standards,
costandards, projective covers and injective hulls.
\end{theorem}

For passage to infinite intervals we also need to compare the
equivalences, not only the categories.

\begin{lemma}\label{lem:finite-equivariant-comparison}
Let $(G_1,\zeta_1)$ and $(G_2,\zeta_2)$ be label-preserving
strongly equivariant equivalences between the same finite tensor
product categorifications, where $\zeta_a:F'G_a\xrightarrow{\sim}G_aF$.
There is a natural isomorphism $\alpha:G_1\xrightarrow{\sim}G_2$
satisfying
\begin{equation}\label{eq:finite-comparison-equivariance}
 (\alpha F)\circ\zeta_1=\zeta_2\circ(F'\alpha).
\end{equation}
Thus $\alpha$ intertwines the specified strong-equivariance
structures on $G_1$ and $G_2$.
\end{lemma}

\begin{proof}
\emph{Idea.}
We first choose $\alpha$ on $L$, transport it to all $F^dL$ along
$\zeta_1,\zeta_2$, and then extend it step by step: first to
projectives, then to arbitrary objects.
At each stage the compatibility
\eqref{eq:finite-comparison-equivariance} is checked on the
generating objects and propagated by exactness.

\medskip
\noindent\textbf{Step 0: the inputs.}
Put $\mathcal J=\operatorname{add}\{F^dL\mid d\geq0\}$.
We use the following three facts.
\begin{enumerate}
\item[(a)] Every object of $\mathcal J$ is projective-injective, and
$H_d\to\operatorname{End}(F^dL)$ is surjective
(Theorem~\ref{thm:finite-hecke-endomorphisms}).
Here surjectivity holds because $e_{I,d}$ is central in $H_d^{m,n}$:
the set $I^d$ is stable under $\mathfrak S_d$, so $e_{I,d}$ is a sum
of full block idempotents. Hence $e_{I,d}H_d^{m,n}e_{I,d}$ is a
quotient of $H_d$.
\item[(b)] $\operatorname{Hom}(F^dL,F^eL)=0$ for $d\neq e$.
Indeed, $F^dL$ is the direct sum of the objects
$F_{i_d}\cdots F_{i_1}L$ over $(i_1,\ldots,i_d)\in I^d$.
By {\rm (CA4)}, every composition factor of a nonzero such object
has weight
$\operatorname{wt}(L)-\boldsymbol\alpha_{i_1}-\cdots-
\boldsymbol\alpha_{i_d}$.
The simple roots are linearly independent, so no weight occurring
in $F^dL$ occurs in $F^eL$ when $d\ne e$.
Thus these objects have no common composition factor, and every
morphism between them is zero.
\item[(c)] Every projective $P$ has an exact sequence
\begin{equation}\label{eq:finite-comparison-coresolution}
 0\longrightarrow P\xrightarrow{\;j\;}J^0\xrightarrow{\;\partial\;}J^1,
 \qquad J^0,J^1\in\mathcal J,
\end{equation}
by \cite[Lemma 5.3]{losev2015uniqueness}; see also
\cite[proof of Lemma 2.16]{brundan2017tensor}.
\end{enumerate}
None of these follows from the mere existence of an equivalence.

\medskip
\noindent\textbf{Step 1: $\alpha$ on $\mathcal J$.}
Since both $G_a$ preserve labels, $G_1L\cong G_2L$; fix an
isomorphism $\alpha_0:G_1L\to G_2L$.
Define the iterated isomorphisms
$\zeta_a^{(d)}:F'^dG_a\to G_aF^d$ by
$\zeta_a^{(0)}=\operatorname{id}$ and
$\zeta_a^{(d+1)}=(\zeta_aF^d)\circ(F'\zeta_a^{(d)})$,
and put
\[
 \alpha_d:=\zeta_{2,L}^{(d)}\circ F'^d(\alpha_0)\circ
 \bigl(\zeta_{1,L}^{(d)}\bigr)^{-1}
 \colon G_1F^dL\xrightarrow{\ \sim\ }G_2F^dL .
\]
In words, $\alpha_d$ is $\alpha_0$ carried through $d$ copies of $F$.

\emph{$\alpha_d$ commutes with Hecke generators.}
For a generator $h\in\{x_j,t_j\}$ of $H_d$ write $h'$ for the
corresponding generator for $F'$. Then
\begin{equation}\label{eq:finite-comparison-hecke-intertwining}
 \zeta_a^{(d)}\circ(h'G_a)=(G_ah)\circ\zeta_a^{(d)} .
\end{equation}
For $d=1$, $h=x$ this is (E2); for $d=2$, $h=t$ it is (E3), because
$\zeta_a^{(2)}=(\zeta_aF)\circ(F'\zeta_a)$.
In general, the isomorphism $\zeta_a^{(d)}$ moves the $d$ copies of
$F'$ across $G_a$ one at a time, and $h$ only involves one or two
adjacent copies. The other steps commute with it by naturality of
$\zeta_a,x',t'$. Now compute, using
\eqref{eq:finite-comparison-hecke-intertwining} for $a=1,2$ and
the naturality of $h'$ with respect to $\alpha_0$:
\[
\begin{aligned}
 \alpha_d\,G_1(h_L)
 &=\zeta_{2,L}^{(d)}\,F'^d(\alpha_0)\,h'_{G_1L}\,(\zeta_{1,L}^{(d)})^{-1}\\
 &=\zeta_{2,L}^{(d)}\,h'_{G_2L}\,F'^d(\alpha_0)\,(\zeta_{1,L}^{(d)})^{-1}\\
 &=G_2(h_L)\,\alpha_d .
\end{aligned}
\]

\emph{$\alpha$ is natural on $\mathcal J$.}
By (a), every endomorphism of $F^dL$ is of the form $h_L$ with
$h\in H_d$, so $\alpha_d$ commutes with all of $\operatorname{End}(F^dL)$.
By (b), a morphism between finite sums of the $F^dL$ is a matrix
whose entries lie in the $\operatorname{End}(F^dL)$.
So the $\alpha_d$ assemble to a natural isomorphism on finite sums of
the $F^dL$, and, being compatible with idempotents, on $\mathcal J$.

\emph{Compatibility with $\zeta$ on $\mathcal J$.}
Unwinding the recursion
$\zeta^{(d+1)}=(\zeta F^d)\circ(F'\zeta^{(d)})$ gives
\[
 \alpha_{d+1}\circ\zeta_{1,F^dL}=\zeta_{2,F^dL}\circ F'(\alpha_d).
\]
Since $F\mathcal J\subseteq\mathcal J$ and everything is natural,
\eqref{eq:finite-comparison-equivariance} holds on $\mathcal J$.

\medskip
\noindent\textbf{Step 2: $\alpha$ on projectives.}
Let $P$ be projective and choose a sequence
\eqref{eq:finite-comparison-coresolution}. Since $G_1,G_2$ are
exact, $G_aP=\ker\bigl(G_a\partial\bigr)$. Naturality of $\alpha$ on
$\mathcal J$ gives
$\alpha_{J^1}G_1(\partial)=G_2(\partial)\alpha_{J^0}$, so
$\alpha_{J^0}$ restricts to an isomorphism
$\alpha_P:G_1P\to G_2P$.

\emph{Independence and naturality.}
Let $f:P\to Q$ be a morphism of projectives with chosen sequences
$(j_P,\partial_P)$ and $(j_Q,\partial_Q)$.
Because $J_Q^0$ is injective and $j_P$ is a monomorphism, the map
$j_Qf$ extends to $g^0:J_P^0\to J_Q^0$ with $g^0j_P=j_Qf$.
Naturality on $g^0$ and the definition on kernels give
\[
\begin{aligned}
G_2(j_Q)\alpha_QG_1(f)
 &=\alpha_{J_Q^0}G_1(g^0)G_1(j_P)\\
 &=G_2(g^0)\alpha_{J_P^0}G_1(j_P)
  =G_2(j_Q)G_2(f)\alpha_P.
\end{aligned}
\]
Since $G_2(j_Q)$ is a monomorphism, this proves
$\alpha_QG_1(f)=G_2(f)\alpha_P$.
Taking $f=\operatorname{id}_P$ with two different sequences shows
that $\alpha_P$ does not depend on the choice.

\emph{Compatibility with $\zeta$ on projectives.}
Let $P\xrightarrow{j}J^0$ be as above. By biadjunction, $F$ is
exact and preserves projectives. Hence $Fj$ is a monomorphism
between projectives, and $G_2(Fj)$ is a monomorphism.
It therefore suffices to check
\eqref{eq:finite-comparison-equivariance} after composing with
$G_2(Fj)$. There
\[
\begin{aligned}
 G_2(Fj)\,(\alpha F)_P\,\zeta_{1,P}
 &=(\alpha F)_{J^0}\,G_1(Fj)\,\zeta_{1,P}
\\
 &=(\alpha F)_{J^0}\,\zeta_{1,J^0}\,F'G_1(j)
\\
 &=\zeta_{2,J^0}\,F'(\alpha_{J^0})\,F'G_1(j)
\\
 &=\zeta_{2,J^0}\,F'G_2(j)\,F'(\alpha_P)\\
 &=G_2(Fj)\,\zeta_{2,P}\,F'(\alpha_P),
\end{aligned}
\]
The first two equalities use naturality, the third uses Step~1,
and the last two use naturality on $j$. Cancelling the
monomorphism gives the claim on $P$.

\medskip
\noindent\textbf{Step 3: arbitrary objects.}
Let $M$ be an object with a projective presentation
$P_1\to P_0\xrightarrow{\pi}M\to0$.
Exactness of $G_1,G_2$ makes $G_aM$ the cokernel of $G_aP_1\to G_aP_0$,
and naturality of $\alpha$ on projectives induces an isomorphism
$\alpha_M:G_1M\to G_2M$.
Every morphism lifts to a map of presentations, so $\alpha$ is
natural; comparing two presentations of $M$ shows $\alpha_M$ is
independent of the choice.

For \eqref{eq:finite-comparison-equivariance} on $M$,
note that $F'G_1\pi$ is an epimorphism, and compute
\[
\begin{aligned}
 (\alpha F)_M\,\zeta_{1,M}\,F'G_1(\pi)
 &=G_2(F\pi)\,(\alpha F)_{P_0}\,\zeta_{1,P_0}\\
 &=G_2(F\pi)\,\zeta_{2,P_0}\,F'(\alpha_{P_0})\\
 &=\zeta_{2,M}\,F'(\alpha_M)\,F'G_1(\pi).
\end{aligned}
\]
Cancelling $F'G_1(\pi)$ gives the equality on $M$.
The statement for $E$ is the same identity for adjunction mates, which
are determined by $\zeta_a$. Hence $\alpha$ intertwines the strong
equivariance structures.
\end{proof}

\subsection{The full interval}
\label{subsec:stable-module-categories}
% Replacement for 8.tex. Uses 7.tex and the highest weight conventions
% and finite-subquotient lemma in semi_infinite_highest_weight.tex.

Let each of $X,Y$ be one of
\[
 \bigwedge^mV^0,\qquad \bigwedge^mV^1,\qquad
 L^+(\boldsymbol\varpi_s),\qquad L^-(-\boldsymbol\varpi_s),
\]
where $m\geq0$ and $s\in\mathbb Z$ may differ between the factors.
Write $\Xi_X,\Xi_Y$ for their monomial indexing sets and put
$\Xi=\Xi_X\times\Xi_Y$, with the inverse dominance order.
We use the essentially finite, lower finite and upper finite
highest weight settings of
Appendix~\ref{sec:semi-infinite-highest-weight-categories}.

\begin{definition}\label{def:level-two-tpc}
A \emph{level-two $\mathfrak{sl}(\mathbb Z)$-tensor product
categorification of type $(X,Y)$} is a highest weight category
$\mathcal C$ in one of these three settings, with weight poset
$\Xi$ and a categorical $\mathfrak{sl}(\mathbb Z)$-action
satisfying:
\begin{enumerate}
\item[(TPC1)] Every $E_i,F_i$ preserves $\mathcal C^\Delta$.
\item[(TPC2)] The isomorphism
\[
 \Bbbk\otimes_{\mathbb Z}K_0(\mathcal C^\Delta)
 \xrightarrow{\sim}X\otimes Y,\qquad
 [\Delta(\lambda,\mu)]\longmapsto
 v_\lambda^X\otimes v_\mu^Y,
\]
intertwines $[E_i],[F_i]$ with the Chevalley generators.
\end{enumerate}
\end{definition}

The choice of highest weight setting is part of the definition.
The lower finite case uses finite highest weight ideal subcategories;
the upper finite case uses the full locally finite-dimensional
module category. As in \cite[Definition 2.10]{brundan2017tensor},
{\rm (TPC1)--(TPC2)} determine standard multiplicities.
For our factors these are at most one, and the highest weight
order determines the following exact sequences.

For a monomial label $\lambda$, write $e_i\lambda$ and
$f_i\lambda$ for the labels obtained by applying the corresponding
Chevalley generator, and interpret a standard or costandard with
a nonexistent label as zero.

\begin{lemma}\label{lem:tpc-standard-costandard-sequences}
In a level-two tensor product categorification there are exact
sequences
\[
\begin{aligned}
0&\longrightarrow\Delta(f_i\lambda,\mu)
 \longrightarrow F_i\Delta(\lambda,\mu)
 \longrightarrow\Delta(\lambda,f_i\mu)\longrightarrow0,\\
0&\longrightarrow\Delta(\lambda,e_i\mu)
 \longrightarrow E_i\Delta(\lambda,\mu)
 \longrightarrow\Delta(e_i\lambda,\mu)\longrightarrow0.
\end{aligned}
\]
For costandards the corresponding sequences are
\[
\begin{aligned}
0&\longrightarrow\nabla(\lambda,f_i\mu)
 \longrightarrow F_i\nabla(\lambda,\mu)
 \longrightarrow\nabla(f_i\lambda,\mu)\longrightarrow0,\\
0&\longrightarrow\nabla(e_i\lambda,\mu)
 \longrightarrow E_i\nabla(\lambda,\mu)
 \longrightarrow\nabla(\lambda,e_i\mu)\longrightarrow0.
\end{aligned}
\]
Thus $E_i,F_i$ preserve finite costandard filtrations. They also
preserve tilting objects, using ascending standard and descending
costandard flags in the upper finite setting. The sequences need
not split.
\end{lemma}

\begin{proof}
By {\rm (TPC1)--(TPC2)}, the standard sections are exactly the
displayed terms, each once.
When both terms occur, inverse dominance gives
$(\lambda,f_i\mu)\prec(f_i\lambda,\mu)$ and
$(e_i\lambda,\mu)\prec(\lambda,e_i\mu)$.
The vanishing $\operatorname{Ext}^1(\Delta(b),\Delta(a))=0$
for $a\prec b$ puts the larger section at the bottom.

For costandards, exact biadjunction gives
\[
 \operatorname{Ext}^1(\Delta(b),F_i\nabla(a))
 \cong\operatorname{Ext}^1(E_i\Delta(b),\nabla(a))=0.
\]
The costandard-flag criteria in
\cite[Theorems 3.11, 3.39 and 3.56]{brundanstroppelsemiinfinite}
apply. In the upper finite case the flag is finite because
$F_i\nabla(a)$ is finitely cogenerated. Indeed, biadjunction and
the preservation of finitely generated projectives imply that
$F_i$ preserves finitely cogenerated injectives, and
$\nabla(a)$ embeds in one. The same holds for $E_i$. Biadjunction and orthogonality give
$(F_i\nabla(a):\nabla(b))=(E_i\Delta(b):\Delta(a))$.
Since $e_i$ and $f_i$ have transpose matrices on monomials,
these are the required multiplicities; dual highest weight
vanishing orders the sections as displayed. In the upper finite
case, exactness and the two adjunctions also preserve the directed
unions and inverse limits defining ascending standard and
descending costandard flags. Hence tiltings are preserved.
\end{proof}

For $X=L^+(\boldsymbol\varpi_s)$ and
$Y=L^-(-\boldsymbol\varpi_t)$ in the lower finite setting,
the sets $|\lambda|+|\mu|\leq d$ are finite ideals exhausting
the labels. Their Serre subcategories inherit the standards,
costandards and tiltings. Every object is a subquotient of a
finite sum of tiltings: work in a finite ideal subcategory,
take a projective cover there, and embed this standard-filtered
projective into a tilting object.
After shifting residues by $-s$, the four sequences above are
those in Entova-Aizenbud's definition with parameter $s-t$
\cite[Definitions 10.2.1 and 10.2.3]{entova2018categorical}.
Her uniqueness theorem also assumes invertibility of the mixed
crossing $EF\to FE$ \cite[Theorem 10.2.7]{entova2018categorical}.
We do not use that theorem or assume this extra condition in
the proof below; the finite subquotients are constructed explicitly.

Put $I_r=\{1-r,\ldots,r-1\}$, $I_r^+=\{1-r,\ldots,r\}$
and $V_r=V_{I_r}$. Take $r\geq1$ with $2r\geq m$ for each
finite exterior power and $r\geq|s|$ for each shifted factor.
Define $\Xi_r\subseteq\Xi$ by imposing the following condition
on each of its two labels. Here $s$ is the shift of the indicated
factor, $\nu$ is its partition label, and missing parts are zero.
\[
\begin{array}{c|c|c}
 Z&\text{condition on its label}&n_Z(r)\\ \hline
 \bigwedge^mV^0&\lambda_a\in I_r^+\ (1\leq a\leq m)&m\\
 \bigwedge^mV^1&\lambda_a\in I_r^+\ (1\leq a\leq m)&2r-m\\
 L^+(\boldsymbol\varpi_s)
   &\nu_1\leq r-s,\quad\nu_{r+s+1}=0&r+s\\
 L^-(-\boldsymbol\varpi_s)
   &\nu_1\leq r+s,\quad\nu_{r-s+1}=0&r-s
\end{array}
\]
Then $\Xi_r\subseteq\Xi_{r+1}$ and $\bigcup_r\Xi_r=\Xi$.
Restricting the wedge indices to $I_r^+$, and taking their
complement there for a dual finite exterior power, identifies
the monomial span of $\Xi_r$ with
\begin{equation}\label{eq:finite-truncation-module}
 \bigwedge^{n_X(r)}V_r\otimes\bigwedge^{n_Y(r)}V_r.
\end{equation}
This identification respects inverse dominance and the
$\mathfrak{sl}_{I_r}$-action.
The numbers $n_Z(r)$ record only the exterior degrees.
For $L^+(\boldsymbol\varpi_s)$ the empty partition occupies
$1-r,\ldots,s$, whereas for $L^-(-\boldsymbol\varpi_s)$ it
occupies $s+1,\ldots,r$. Thus, even when these degrees agree,
the two vacuum labels give different finite wedge monomials.
The ordered factors and their monomial identifications are retained
throughout the truncation.

Let $\Xi_{\leq r}$ consist of the labels $(\lambda,\mu)$ for
which there is $(\lambda',\mu')\in\Xi_r$ with
$\lambda\unrhd_X\lambda'$ and
\[
 \operatorname{wt}(v_\lambda^X\otimes v_\mu^Y)
 -\operatorname{wt}(v_{\lambda'}^X\otimes v_{\mu'}^Y)
 \in\sum_{i\in I_r}\mathbb Z\boldsymbol\alpha_i.
\]
Set $\Xi_{<r}:=\Xi_{\leq r}\setminus\Xi_r$.

\begin{lemma}\label{lem:tpc-finite-truncation}
The subsets $\Xi_{\leq r}$ and $\Xi_{<r}$ are ideals whose
highest weight subcategories are preserved by $E_i,F_i$
for $i\in I_r$. Their Serre subquotient
\[
 \mathcal C_r=\mathcal C_{\leq r}/\mathcal C_{<r}
\]
is a finite tensor product categorification of
\eqref{eq:finite-truncation-module}.
For $s\geq r$, truncating $\mathcal C_s$ to $\Xi_r$ gives
$\mathcal C_r$, compatibly with standards, adjunctions, dots
and crossings for residues in $I_r$.
\end{lemma}

\begin{proof}
We first make the defining conditions explicit. Represent a label
by its occupied positions: use the wedge indices for $V^0$ and
for the two semi-infinite wedges, and use the complement of the
wedge indices for $V^1$. Each factor has fixed occupations outside
a finite set. The condition that the difference of total weights
lies in $\sum_{i\in I_r}\mathbb Z\boldsymbol\alpha_i$ says exactly
that, outside $I_r^+$, the sum of the two occupations agrees with
that of labels in $\Xi_r$.
Indeed, the difference of two such total occupation sequences
has finite support and total sum zero; the roots in $I_r$ span
precisely the sequences of this kind supported on $I_r^+$.

Every finite wedge has a lowest monomial, with its occupied
positions as far to the right as possible. Consequently the
existence of a witnessing first label in $\Xi_r$ is equivalent
to the following condition on the first factor:
\begin{itemize}
\item for $\bigwedge^mV^0$, all wedge indices are at most $r$;
\item for $\bigwedge^mV^1$, all wedge indices are at least $1-r$;
\item for $L^+(\boldsymbol\varpi_s)$, the partition lies in
the rectangle in the table;
\item for $L^-(-\boldsymbol\varpi_s)$, there is no further
condition, since the empty partition is the lowest monomial.
\end{itemize}
This also proves that $\Xi_{\leq r}$ is an ideal: on passing
to a smaller tensor label, total weight is unchanged and the
first-factor weight increases.

The subset $\Xi_{<r}$ is an ideal as well. To see this, suppose
$(\rho,\sigma)\preceq(\lambda,\mu)\in\Xi_{\leq r}$ and
$(\rho,\sigma)\in\Xi_r$. Choose a witnessing
$(\lambda',\mu')\in\Xi_r$. Then
$\rho\unrhd_X\lambda\unrhd_X\lambda'$.
The finite-window labels in each individual factor are convex
for its weight order: this follows from the coordinate bounds
for finite wedges and the rectangle bounds for partitions.
Thus $\lambda$ is a finite-window label. Since
$(\rho,\sigma)$ and $(\lambda,\mu)$ have the same total weight,
their total occupations outside $I_r^+$ agree. The first
occupations there are now fixed, so those of $\mu$ are fixed
as well. Hence $(\lambda,\mu)\in\Xi_r$, as required.

A Chevalley move at $i\in I_r$ changes only positions $i,i+1$,
both of which belong to $I_r^+$. It preserves the exterior
occupation condition and each of the four first-factor
conditions above. It also preserves whether both labels belong
to $\Xi_r$. Therefore every standard section of $E_i\Delta(a)$
or $F_i\Delta(a)$ remains in the same ideal when
$a\in\Xi_{\leq r}$ or $a\in\Xi_{<r}$.
In the finite-length settings, every simple in an ideal is a
quotient of its standard, so exactness proves stability of the
whole ideal subcategory. In the upper finite setting, use the
projective covers in that ideal subcategory, which have finite
standard flags, and then quotients and directed unions of
finitely generated subobjects. The functors preserve these
unions by adjunction.

Both adjoints preserve both subcategories, so their adjunctions,
as well as $x$ and $t$, descend to the Serre subquotient.
Lemma~\ref{lem:finite-highest-weight-subquotients} makes it
finite highest weight. The surviving standard flags are exactly
the finite wedge tensor product rules, proving the finite TPC
axioms; compare \cite[Lemma 2.18 and Theorem 2.19]
{brundan2017tensor}.
Finally, the same exterior occupation conditions describe
truncation from $\Xi_s$ to $\Xi_r$. Transitivity of Serre
subquotients identifies it with $\mathcal C_r$, including the
induced adjunctions and natural transformations.
\end{proof}

\begin{theorem}
\cite{brundan2017tensor,entova2018categorical,wan2025parabolic}
\label{thm:level-two-tpc-uniqueness}
Let $\mathcal C,\mathcal C'$ be level-two tensor product
categorifications of the same ordered type $(X,Y)$, in the
same one of the essentially finite, lower finite or upper
finite settings. There is a strongly equivariant equivalence
\[
 G:\mathcal C\xrightarrow{\sim}\mathcal C',\qquad
 GL(\lambda,\mu)\cong L'(\lambda,\mu)
 \quad((\lambda,\mu)\in\Xi).
\]
\end{theorem}

\begin{proof}
We first treat the essentially finite case. The lower finite and
upper finite cases are explained in
Remark~\ref{rem:tpc-uniqueness-semi-infinite}.

\emph{Idea.}
For any fixed finite collection of projectives, truncation
eventually preserves both the projectives and all morphisms
between them. We will use the finite equivalences to identify
these Hom spaces. The comparison lemma makes the identifications
compatible as the finite interval increases. We then obtain an
equivalence of the original projective categories and extend it
to the abelian categories.

\medskip
\noindent\textbf{Step 1: finite equivalences and their comparison.}
Write $F_r=\bigoplus_{i\in I_r}F_i$ on $\mathcal C_r$, and
use a prime for the corresponding functor on $\mathcal C'_r$.
By Lemma~\ref{lem:tpc-finite-truncation} and
Theorem~\ref{thm:finite-level-two-tpc-uniqueness}, choose
label-preserving strongly equivariant equivalences
$G_r:\mathcal C_r\xrightarrow{\sim}\mathcal C'_r$, with
$\zeta_r:F'_rG_r\xrightarrow{\sim}G_rF_r$.

For $s\geq r$, denote the relevant subquotient functors by
$q_{r,s}$ and $q'_{r,s}$. These are defined on the ideal
subcategories used in Lemma~\ref{lem:tpc-finite-truncation};
they are not inclusions of $\mathcal C_r$ into $\mathcal C_s$.
Since $G_s$ preserves labels, it induces an equivalence on these
subquotients, denoted by $G_s|_r$.
The comparison lemma gives a natural isomorphism
\[
 \beta_r:G_{r+1}|_r\xrightarrow{\sim}G_r,\qquad
 (\beta_rF_r)\circ(\zeta_{r+1}|_r)
   =\zeta_r\circ(F'_r\beta_r).
\]
Here $\zeta_{r+1}|_r$ means its restriction to residues in $I_r$
and its induced transformation on the subquotient.
Composing these isomorphisms along the chain gives
$\Phi_{r,s}:G_s|_r\xrightarrow{\sim}G_r$ with
$\Phi_{r,t}=\Phi_{r,s}\circ(\Phi_{s,t}|_r)$.
We use the canonical transitivity identifications of the
subquotients. Each $\Phi_{r,s}$ respects the given strong
equivariance structures for residues in $I_r$.

\medskip
\noindent\textbf{Step 2: stabilization on projectives.}
Fix projective covers $P(a)$ and $P'(a)$ in the original
categories, for every $a\in\Xi$.
For any finite set of these projectives, choose $r$ so large
that all labels in their finite standard flags belong to
$\Xi_r$. This uses finiteness of each standard flag, not upper
finiteness of the whole weight poset.
Such projectives belong to the respective ideal subcategories
$\mathcal C_{\leq r}$ and $\mathcal C'_{\leq r}$.
Their images are the projective covers $P_r(a)$ and $P'_r(a)$,
and truncation gives isomorphisms
\begin{equation}\label{eq:projective-hom-stabilization-new}
 \operatorname{Hom}_{\mathcal C}(P(a),P(b))
 \xrightarrow{\sim}
 \operatorname{Hom}_{\mathcal C_r}(P_r(a),P_r(b)),
\end{equation}
and likewise for $\mathcal C'$.
This is the projective assertion of
Lemma~\ref{lem:finite-highest-weight-subquotients}.
Concretely, the finite quotient is given by an idempotent
corner. The projectives with surviving heads become its
principal projectives, and their Hom spaces are the same
algebra corners as before truncation.
The same assertion applies between finite stages once the
standard labels of the projectives lie in $\Xi_r$.

\medskip
\noindent\textbf{Step 3: an equivalence of the original projective categories.}
For each $a$, choose $R_a$ so that Step~2 applies to $P(a)$ and
$P'(a)$ for every $r\geq R_a$. We choose isomorphisms
\[
 \theta_{r,a}:G_rP_r(a)\xrightarrow{\sim}P'_r(a)
 \qquad(r\geq R_a)
\]
compatible with $\beta_r$, in the sense that
\begin{equation}\label{eq:tpc-compatible-projective-identifications}
 q'_{r,r+1}(\theta_{r+1,a})
   =\theta_{r,a}\circ\beta_{r,P_r(a)}.
\end{equation}
To construct them, choose $\theta_{R_a,a}$ arbitrarily.
Once $\theta_{r,a}$ has been chosen, the right hand side of
\eqref{eq:tpc-compatible-projective-identifications} is an
isomorphism in $\mathcal C'_r$. Both projectives upstairs,
$G_{r+1}P_{r+1}(a)$ and $P'_{r+1}(a)$, have all their standard
labels in $\Xi_r$: the first follows because $G_{r+1}$ preserves
the labelled standards, and the second follows from the choice
of $R_a$.
Step~2 therefore lifts this isomorphism uniquely to
$\theta_{r+1,a}$. Its inverse lifts as well, so the lift is an
isomorphism.

Now set $G(P(a))=P'(a)$, an object of the original category
$\mathcal C'$. For $f:P(a)\to P(b)$, choose
$r\geq R_a,R_b$ and let $f_r$ be its image at the finite stage.
Define $G(f)$ to be the unique morphism whose image in
$\mathcal C'_r$ is
\begin{equation}\label{eq:tpc-glued-projective-morphism}
 \theta_{r,b}\,G_r(f_r)\,\theta_{r,a}^{-1}.
\end{equation}
Existence and uniqueness follow from
\eqref{eq:projective-hom-stabilization-new} for $\mathcal C'$.
Naturality of $\beta_r$ and
\eqref{eq:tpc-compatible-projective-identifications} show that
the formula at stage $r+1$ reduces to the formula at stage $r$.
The injectivity of the Hom comparison then proves independence
of $r$.

Identities and compositions can be checked at one sufficiently
large finite stage. Formula~\eqref{eq:tpc-glued-projective-morphism}
is an isomorphism on every Hom space, since $G_r$ and both
truncation maps are fully faithful there. Extending to finite
direct sums gives an equivalence between the projective
categories. Taking cokernels of projective presentations gives
an exact equivalence $G:\mathcal C\xrightarrow{\sim}\mathcal C'$.
It sends the simple head of $P(a)$ to that of $P'(a)$, hence
preserves all simple labels.

\medskip
\noindent\textbf{Step 4: strong equivariance.}
Fix $i\in\mathbb Z$ and a projective $P$. By biadjunction,
$F_iP$ and $E_iP$ are projective.
Choose a finite stage containing $i$ and all standard labels of
the projectives involved, on both sides. Use the identifications
$\theta$ to transport the $i$-component of $\zeta_r$ to an
isomorphism between the finite images of $F'_iG(P)$ and
$G(F_iP)$. Step~2 lifts it uniquely to
$\zeta_{i,P}:F'_iG(P)\xrightarrow{\sim}G(F_iP)$.
The compatibility of $\beta_r$ with $\zeta_r$ makes this lift
independent of the stage. Naturality follows by putting a
morphism of projectives and its images into one finite stage.

For each projective $P$, only finitely many $F_iP$ are nonzero.
Indeed, $P$ has a finite standard flag and each monomial has
only finitely many nonzero Chevalley moves. Thus the
$\zeta_{i,P}$ sum to an isomorphism
$\zeta_P:F'G(P)\xrightarrow{\sim}G(FP)$.
The identities involving $x$ and $t$ can be checked at a finite
stage containing all projectives and residues occurring in
$P,FP,F^2P$ and their counterparts in $\mathcal C'$.
They hold there by strong equivariance of $G_r$, and the
faithfulness in Step~2 gives them in the original categories.

Finally, all the functors involved are exact. A projective
presentation of $M$ extends $\zeta$ to $M$ by taking cokernels.
Naturality and the identities involving $x,t$ descend along
the resulting epimorphisms from projectives.
The adjunction-mate condition is automatic for an equivalence,
so $G$ is strongly equivariant.
\end{proof}

\begin{remark}\label{rem:tpc-uniqueness-semi-infinite}
\label{rem:tpc-uniqueness-semi-infinite-new}
We give the changes needed in the two semi-infinite settings.

\emph{Lower finite case.}
Use finite ideal subcategories instead of global projectives.
For a finite ideal $\Omega\subseteq\Xi$, take $r$ with
$\Omega\subseteq\Xi_r$. No label of $\Omega$ or below it is
removed by truncation. Hence
Lemma~\ref{lem:finite-highest-weight-subquotients} identifies
$\mathcal C_\Omega$ and $\mathcal C'_\Omega$ with the respective
Serre subcategories of $\mathcal C_r$ and $\mathcal C'_r$.
The restrictions of $G_r$ are therefore actual equivalences
$\mathcal C_\Omega\to\mathcal C'_\Omega$.
The isomorphisms $\Phi_{r,s}$ compare these restrictions,
compatibly with enlarging both the finite stage and the ideal.
Choose an increasing exhaustion of $\Xi$ by finite ideals and
identify the equivalences on successive subcategories using
these isomorphisms. This gives a functor on their union:
each object, morphism and finite diagram belongs to one such
subcategory, and compatibility makes its image independent of
the choice. The same construction for the inverses proves that
the functor is an equivalence.
For a fixed object, its images under $F,E,F^2$ and the finitely
many residues involved belong to a sufficiently large finite
ideal subcategory. The finite equivariance isomorphisms and
their identities then give strong equivariance on the union.

\emph{Upper finite case.}
Steps~1--3 apply to finitely generated projectives, since these
have finite standard flags. The equivalence of these projective
categories gives an isomorphism of their locally unital
endomorphism algebras preserving the labelled idempotents.
It therefore induces an equivalence of the categories of
locally finite-dimensional right modules.

For strong equivariance, Step~4 first gives the required
isomorphisms on finitely generated projectives, for each $F_i$.
Taking adjunction mates within the projective categories also
gives $E'_iG\cong GE_i$ there.
For any $M\in\mathcal C$ and finitely generated projective $P$,
biadjunction now gives natural isomorphisms
\[
\begin{aligned}
 \operatorname{Hom}(GP,F'_iGM)
 &\cong\operatorname{Hom}(E'_iGP,GM)\\
 &\cong\operatorname{Hom}(GE_iP,GM)
  \cong\operatorname{Hom}(E_iP,M)\\
 &\cong\operatorname{Hom}(P,F_iM)
  \cong\operatorname{Hom}(GP,GF_iM).
\end{aligned}
\]
Modules are determined by these Hom spaces and their
precomposition maps. Hence these isomorphisms define
$\zeta_{i,M}:F'_iGM\xrightarrow{\sim}GF_iM$, extending the
isomorphism already constructed on projectives.
Their locally finite sum gives $\zeta:F'G\xrightarrow{\sim}GF$.

The identities involving $x,t$ hold on finitely generated
projectives. They then hold on finitely generated objects,
by an epimorphism from a finite sum of such projectives.
Every locally finite-dimensional module is a directed union of
its finitely generated submodules; the functors preserve these
unions by exactness and adjunction. Thus the identities hold
on all objects. This argument does not require arbitrary
direct sums to belong to the locally finite-dimensional category.
\end{remark}

\subsection{Ringel duality}
% Replacement for 10.tex; the subsection heading remains in main.tex.

\begin{proposition}\label{prop:ringel-duality-tpc}
Let $\mathcal C$ be a lower finite, upper finite, or tilting-bounded
essentially finite tensor product categorification of type $(X,Y)$.
Its Ringel dual $\mathcal C^\vee$ is, respectively, upper finite,
lower finite, or tilting-bounded essentially finite, and has type
$(Y,X)$.
\end{proposition}

\begin{proof}
The highest weight assertion follows from
Theorem~\ref{thm:semi-infinite-ringel-duality}.
Lemma~\ref{lem:tpc-standard-costandard-sequences} makes
$E_i,F_i$ preserve tiltings and both finite flag categories.
Transport through \eqref{eq:ringel-finite-flag-equivalences} extends
by projective presentations, or in finite ideal subcategories when
$\mathcal C^\vee$ is lower finite, preserving adjunctions and
Hecke transformations. Local finiteness, nilpotence and
integrability are checked on finite flags.
Compare \cite[Remark 5.31]{brundan2017tensor} and
\cite[Lemma 2.32]{wan2025parabolic}.

Biadjunction gives
$(F_i\nabla(a):\nabla(b))=(E_i\Delta(b):\Delta(a))$,
so the transported standard multiplicities are the tensor product
multiplicities. Exchanging $(\lambda,\mu)$ with $(\mu,\lambda)$
reverses inverse dominance at fixed total weight, giving
{\rm (TPC1)--(TPC2)} for $(Y,X)$.
The functors $E_i,F_i$ themselves are not interchanged.
\end{proof}

\subsection{Type A KLR algebras}
% KLR conventions and the comparison of categorical actions.

We use the KLR conventions of
\cite[\S\S2.3--2.4]{brundan2017tensor}.
The field $\Bbbk$ is algebraically closed of characteristic zero.

\begin{definition}\label{def:klr-from-hecke}
For $d\geq1$, the type $A_\infty$ KLR algebra $R_d$ is the
locally unital $\Bbbk$-algebra generated by
$1_{\mathbf i}$, $\xi_a1_{\mathbf i}$ and
$\tau_r1_{\mathbf i}$, where
$\mathbf i=(i_1,\ldots,i_d)\in\mathbb Z^d$,
$1\leq a\leq d$ and $1\leq r<d$, with the relations below.
The entries $i_a$ are residues, while $a,r$ are positions;
$s_r\mathbf i$ exchanges the entries in positions $r,r+1$.
The symbols $\xi_a,\tau_r$ denote multipliers, so every
relation is understood after multiplication by $1_{\mathbf i}$.
The local units are finite sums of the $1_{\mathbf i}$.
\[
\begin{gathered}
 1_{\mathbf i}1_{\mathbf j}
   =\delta_{\mathbf i,\mathbf j}1_{\mathbf i},\qquad
 \xi_a1_{\mathbf i}=1_{\mathbf i}\xi_a,\qquad
 \xi_a\xi_b=\xi_b\xi_a,\\
 \tau_r1_{\mathbf i}=1_{s_r\mathbf i}\tau_r,\qquad
 \tau_r\tau_t=\tau_t\tau_r\quad(|r-t|>1),\\
 (\tau_r\xi_a-\xi_{s_r(a)}\tau_r)1_{\mathbf i}
 =\delta_{i_r,i_{r+1}}
       (\delta_{a,r+1}-\delta_{a,r})1_{\mathbf i}.
\end{gathered}
\]
\[
 \tau_r^21_{\mathbf i}=
 \begin{cases}
 0,&i_r=i_{r+1},\\
 (\xi_r-\xi_{r+1})1_{\mathbf i},&i_{r+1}=i_r+1,\\
 (\xi_{r+1}-\xi_r)1_{\mathbf i},&i_{r+1}=i_r-1,\\
 1_{\mathbf i},&\text{otherwise},
 \end{cases}
\]
\[
 (\tau_{r+1}\tau_r\tau_{r+1}
       -\tau_r\tau_{r+1}\tau_r)1_{\mathbf i}
 =\begin{cases}
 1_{\mathbf i},&i_r=i_{r+2}=i_{r+1}-1,\\
 -1_{\mathbf i},&i_r=i_{r+2}=i_{r+1}+1,\\
 0,&\text{otherwise}.
 \end{cases}
\]
Here $\mathbf i,\mathbf j\in\mathbb Z^d$,
$1\leq a,b\leq d$ and $1\leq r,t<d$;
the last relation is imposed for $1\leq r\leq d-2$.
Set $R_0=\Bbbk$.
The grading is given by $\deg1_{\mathbf i}=0$,
$\deg(\xi_a1_{\mathbf i})=2$, and
$\deg(\tau_r1_{\mathbf i})=-2,1,0$ according as
$i_r=i_{r+1}$, $|i_r-i_{r+1}|=1$, or
$|i_r-i_{r+1}|>1$.
\end{definition}

Fix a nonempty interval $I\subseteq\mathbb Z$ and an ambient
category $\mathcal C$ as in
Definition~\ref{def:categorical-slI-action}.
Use the same conventions for local finiteness and local
nilpotence.

\begin{definition}\label{def:type-a-klr-action}
A \emph{categorical $\mathfrak{sl}_I$-action in KLR form}
consists of locally finite families of exact biadjoint
$\Bbbk$-linear functors $F_i,E_i$, with fixed adjunctions
$F_i\dashv E_i$, and
natural transformations
$\xi_i\in\operatorname{End}(F_i)$ and
$\tau_{ij}:F_jF_i\to F_iF_j$ for $i,j\in I$.
Each $\xi_i$ is locally nilpotent, and condition
{\rm (CA4)} of Definition~\ref{def:categorical-slI-action}
is required.

Put $F=\bigoplus_{i\in I}F_i$ and
$F_{\mathbf i}=F_{i_d}\cdots F_{i_1}$ for
$\mathbf i\in I^d$.
For every $d\geq0$, the following assignments must define
an algebra homomorphism
$R_d\to\operatorname{End}(F^d)$:
$1_{\mathbf i}$ projects onto $F_{\mathbf i}$;
$\xi_a1_{\mathbf i}$ inserts $\xi_{i_a}$ in the factor
$F_{i_a}$; and $\tau_r1_{\mathbf i}$ inserts
$\tau_{i_r,i_{r+1}}$ in $F_{i_{r+1}}F_{i_r}$, with
target $F_{s_r\mathbf i}$.
All these maps vanish on the other summands, and
$1_{\mathbf i}$ acts as zero for $\mathbf i\notin I^d$.
For $d>0$ the projections sum locally to
$\operatorname{id}_{F^d}$; for $d=0$ use the scalar action.
\end{definition}

We use the following standard comparison; see
\cite[\S3.2.6]{rouquier20082} and
\cite[Theorem 2.5 and \S2.4]{brundan2017tensor}.

\begin{theorem}\label{thm:hecke-klr-action-equivalence}
A fixed choice of the Brundan--Kleshchev--Rouquier change of
generators identifies the categorical actions in
Definitions~\ref{def:categorical-slI-action} and
\ref{def:type-a-klr-action}, preserving the functors and
their fixed adjunctions. The natural endomorphisms satisfy
\begin{equation}\label{eq:hecke-klr-dot}
 x|_{F_i}=i\operatorname{id}_{F_i}+\xi_i.
\end{equation}
Strong equivariance in Hecke form is equivalent to strong
equivariance in KLR form, where one intertwines the
$F_i$, $\xi_i$ and $\tau_{ij}$ and uses the adjunction mates
on the $E_i$.
In particular, the mate of $x|_{F_i}$ is
$i\operatorname{id}_{E_i}+\xi_i^*$, where $\xi_i^*$ is
the mate of $\xi_i$.
\end{theorem}

Local nilpotence makes the power series in this comparison
well defined on each finite-length object, or on each
finite-dimensional idempotent component in case {\rm (ii)}.
Thus the comparison applies in both ambient categories
used here.

\section{Khovanov arc algebras}
\label{sec:Khovanov-arc-algebras}

\subsection{Weight diagrams and blocks}
\label{subsec:weight-diagrams}
% Pure diagram conventions used in 32.tex, 30.tex and 35.tex.

We use the diagram conventions of
\cite[\S2]{brundanstroppelI} and \cite[\S2]{brundanstroppelII}.

A \emph{number line} has vertices indexed by an interval
$I\subseteq\mathbb Z$. A \emph{weight diagram} labels each vertex
by $\vee$ or $\wedge$. Outside a finite set of vertices,
no $\vee$ may lie to the left of a $\wedge$.
Two weights belong to the same \emph{block} if one is obtained
from the other by a finite permutation of its labels.
The \emph{Bruhat order} is generated by exchanging a $\vee$
and a $\wedge$ so that the $\vee$ moves to the right.
Weights in different blocks are incomparable.

For nonnegative integers $m,n$, let $\Lambda_n^m$ consist of
weights on $\{1,\ldots,m+n\}$ with $n$ labels $\vee$ and
$m$ labels $\wedge$.
Let $\Lambda_n^{+\infty}$, $\Lambda_n^{-\infty}$ and
$\Lambda_n^\infty$ consist of weights with exactly $n$ labels
$\vee$, on $\mathbb Z_{>0}$, $\mathbb Z_{\leq0}$ and
$\mathbb Z$, respectively; every other label is $\wedge$.
Finally, $\Lambda_\infty^\infty$ is the block of weights on
$\mathbb Z$ with
$\lambda_i=\wedge$ for $i\ll0$ and $\lambda_i=\vee$ for $i\gg0$,
subject to
\[
 \#\{i\leq0\mid\lambda_i=\vee\}
 =
 \#\{i>0\mid\lambda_i=\wedge\}.
\]
This fixes the charge; other fixed charges give translated
copies of the same poset.
The poset $\Lambda_\infty^\infty$ is upper finite:
a move upwards from $\lambda$ changes vertices only between
its leftmost $\vee$ and its rightmost $\wedge$.

A \emph{cup diagram} consists of nonintersecting cups below a
number line and downward rays at the remaining vertices.
It has only finitely many cups. A \emph{cap diagram} is its
reflection in a horizontal axis; its rays point upwards.
For a weight $\lambda$, construct $\underline{\lambda}$ by
repeatedly joining a $\vee$ to a $\wedge$ to its right whenever
all intervening vertices have already been paired.
Attach rays to the remaining vertices.
Put $\overline{\lambda}=(\underline{\lambda})^*$, where
$*$ denotes reflection in a horizontal axis.

A labelled cup diagram $a\lambda$ is \emph{oriented} if each
cup joins opposite labels and all $\wedge$-labelled rays lie
to the left of all $\vee$-labelled rays.
The same rule defines an oriented cap diagram $\lambda b$.
A cup or cap is \emph{clockwise} if its left endpoint is
$\wedge$ and its right endpoint is $\vee$, and is
\emph{anti-clockwise} otherwise.

For weights $\alpha,\lambda$ in the same block, write
$\alpha\subset\lambda$ if $\underline{\alpha}\lambda$ is
oriented. Equivalently, $\lambda\overline{\alpha}$ is oriented,
and we also write $\lambda\supset\alpha$.
The notation $(\alpha\lambda\beta)$ denotes the circle diagram
$\underline{\alpha}\lambda\overline{\beta}$.
It is oriented precisely when $\alpha\subset\lambda\supset\beta$.
One has $\alpha\subset\lambda\Rightarrow\alpha\leq\lambda$;
see \cite[Lemma 2.3]{brundanstroppelI}.

Set $\operatorname{def}(\lambda)$ equal to the number of cups
in $\underline{\lambda}$, and
$\operatorname{def}(\Lambda)=\sup_{\lambda\in\Lambda}
\operatorname{def}(\lambda)$. Thus
\[
 \operatorname{def}(\Lambda_n^m)=\min(m,n),\qquad
 \operatorname{def}(\Lambda_n^{\pm\infty})
 =\operatorname{def}(\Lambda_n^\infty)=n,\qquad
 \operatorname{def}(\Lambda_\infty^\infty)=\infty.
\]
For fixed $\lambda$, there are $2^{\operatorname{def}(\lambda)}$
weights $\mu$ with $\lambda\subset\mu$.
For fixed $\mu$, the set of $\lambda$ with $\lambda\subset\mu$
is finite exactly when the block has finite defect
\cite[Lemma 2.4]{brundanstroppelI}.

Let $\Lambda,\Gamma$ be blocks.
A \emph{crossingless $\Lambda\Gamma$-matching} $t$ joins their
number lines, with $\Lambda$ below $\Gamma$.
It consists of finitely many cups and caps, together with
nonintersecting line segments between the two lines.
A cap joins two bottom vertices; a cup joins two top vertices.
Write $\operatorname{caps}(t)$ and $\operatorname{cups}(t)$
for their numbers, and $t^*$ for the reflected matching.

The diagram $\lambda t\mu$ is oriented if every cup and cap
joins opposite labels, and every line segment joins equal
labels. The matching is \emph{proper} if such labels exist.
For a cup diagram $a$ below $t$ and a cap diagram $b$ above it,
$a\lambda t\mu b$ is oriented when $a\lambda$, $\lambda t\mu$
and $\mu b$ are all oriented.
Its degree counts every clockwise cup and cap, including
those in $t$:
\[
 \deg(a\lambda t\mu b)
 =\deg(a\lambda)+\deg(\lambda t\mu)+\deg(\mu b).
\]
In particular, $\deg(\lambda t\mu)$ counts the clockwise cups
and caps of the labelled matching itself.
A circle is clockwise if its rightmost vertex is labelled
$\vee$, and anti-clockwise if it is labelled $\wedge$.

In $tb$, an \emph{upper circle} or \emph{upper line} is a
connected component disjoint from the bottom number line.
The \emph{upper reduction} of $tb$ is the cap diagram obtained
by deleting these components and the top number line, then
straightening the remaining curves.
The lower reduction of $at$ is defined by reflection.

For a $\Lambda\Gamma$-matching $t$ and a
$\Gamma\Upsilon$-matching $u$, place $u$ above $t$ and identify
vertices with the same index on their $\Gamma$-number lines.
Write $tu$ for the resulting diagram, retaining its middle
number line. A \emph{stack of matchings} is a diagram obtained
by repeating this construction.
It is \emph{proper} if its number lines can be labelled by weights
of their respective blocks so that each matching is oriented.
Its \emph{reduction} is obtained by deleting all circles
disjoint from the two outside number lines, deleting the
intermediate number lines, and straightening the remaining
curves. The deleted circles are called \emph{internal circles}.

\subsection{Highest weight structures}

% Highest weight structure of K_\Lambda; no results about H_\Lambda.

Let $\Lambda$ be one of the weight posets above.
The graded algebra $K_\Lambda$ has homogeneous basis
\[
 \{(\alpha\lambda\beta)\mid
   \alpha,\lambda,\beta\in\Lambda,\ 
   \alpha\subset\lambda\supset\beta\}.
\]
Its multiplication is the one defined in
\cite[\S4, and the surgery description in \S6]{brundanstroppelI}.

A surgery cuts a cup and a cap and reconnects their endpoints
vertically. Its effect on the labels is given by the following
local rules. Write $1$ for an anti-clockwise circle and $z$
for a clockwise circle. Merging
two circles and splitting one circle are given by
\begin{equation}\label{eq:arc-circle-surgery}
\begin{gathered}
 m(1\otimes1)=1,\qquad
 m(1\otimes z)=m(z\otimes1)=z,\qquad m(z\otimes z)=0,\\
 d(1)=1\otimes z+z\otimes1,\qquad d(z)=z\otimes z.
\end{gathered}
\end{equation}
Thus $m$ is multiplication in $\Bbbk[z]/(z^2)$.
For a line, denoted by $\mathsf l$, the extra rules are
\[
 1\otimes\mathsf l\longmapsto\mathsf l,\qquad
 z\otimes\mathsf l\longmapsto0,\qquad
 \mathsf l\longmapsto z\otimes\mathsf l,
\]
also allowing interchange of the two components.
A surgery from two lines to two lines is nonzero only when
one input line has both rays labelled $\wedge$ and the other
has both rays labelled $\vee$. In that case reconnect the
lines and keep the labels on the rays. In all other cases it
is zero. These are the rules of
\cite[(3.4)--(3.5), (6.1)--(6.2)]{brundanstroppelI}.

Set $e_\lambda=(\lambda\lambda\lambda)$.
These are mutually orthogonal idempotents, and
$K_\Lambda=\bigoplus_{\alpha,\beta}e_\alpha K_\Lambda e_\beta$.
Thus $K_\Lambda$ is locally unital, and is unital when $\Lambda$
is finite.
Each corner is finite-dimensional.
The grading is nonnegative, with
$(K_\Lambda)_0=\bigoplus_\lambda\Bbbk e_\lambda$;
see \cite[(4.13) and (5.2)]{brundanstroppelI}.

We write $K_m^n=K_{\Lambda_m^n}$,
$K_n^{\pm\infty}=K_{\Lambda_n^{\pm\infty}}$,
$K_n^\infty=K_{\Lambda_n^\infty}$ and
$K_\infty^\infty=K_{\Lambda_\infty^\infty}$.

Reverse the order of the vertices and interchange $\vee,\wedge$.
Write $\lambda^\sharp$ for the resulting weight and
$\Lambda^\sharp$ for its block.
Reflect the curves along with the vertices.

\begin{proposition}\label{prop:arc-diagram-symmetries}
Horizontal reflection defines a graded anti-involution
$(\alpha\lambda\beta)^*=(\beta\lambda\alpha)$ of $K_\Lambda$.
The map
\[
 K_\Lambda\xrightarrow{\sim}K_{\Lambda^\sharp},\qquad
 (\alpha\lambda\beta)\longmapsto
 (\alpha^\sharp\lambda^\sharp\beta^\sharp)
\]
is a graded algebra isomorphism. It is an involutive
automorphism when $\Lambda^\sharp=\Lambda$ and the chosen
reversal squares to the identity.
In particular, $K_m^n\cong K_n^m$; on $K_\infty^\infty$,
reversing vertices by $i\mapsto1-i$ gives an involutive
graded automorphism.
\end{proposition}

\begin{proof}
Both operations preserve degrees and the surgery rules.
Horizontal reflection reverses the order of multiplication,
whereas the second operation preserves it; equivalently, the
second is the composite of the two anti-isomorphisms in
\cite[(4.14)--(4.15)]{brundanstroppelI}.
For $\Lambda_\infty^\infty$, the reversal $i\mapsto1-i$
together with the label exchange preserves the charge condition.
\end{proof}

Modules in this section are locally unital left modules.
Write $\operatorname{Rep}(K_\Lambda)$ for finite-dimensional
graded modules and $\operatorname{rep}(K_\Lambda)$ for
finite-dimensional ungraded modules.
Write $K_\Lambda\text{-mod}_{\mathrm{lfd}}$ for the full category
of modules $M$ with $\dim e_\lambda M<\infty$ for every
$\lambda$, and $K_\Lambda\text{-gmod}_{\mathrm{lfd}}$ for its
graded version. This is the left-module convention for
Appendix~\ref{sec:semi-infinite-highest-weight-categories}.
Grading shifts raise degrees:
$(M\langle j\rangle)_a=M_{a-j}$.

Put $P(\lambda)=K_\Lambda e_\lambda$.
Let $L(\lambda)$ be the one-dimensional module in degree zero
on which $e_\lambda$ acts as $1$ and all other basis elements
act as $0$.
Let $\Delta(\lambda)$ be the module defined in
\cite[\S5]{brundanstroppelI}, with homogeneous basis
\[
 \{(\alpha\lambda\mid)\mid\alpha\subset\lambda\},
 \qquad
 \deg(\alpha\lambda\mid)=\deg(\underline{\alpha}\lambda).
\]
We use the action defined there.

For a module $M$, let $M^{\circledast}$ be its graded restricted
dual, so that $(e_\lambda M^{\circledast})_j=
\operatorname{Hom}_\Bbbk(e_\lambda M_{-j},\Bbbk)$, with the action
$(xf)(v)=f(x^*v)$.
Put $\nabla(\lambda)=\Delta(\lambda)^{\circledast}$.

\begin{theorem}\label{thm:arc-highest-weight-structure}
\cite[Theorems 5.1--5.3 and Corollary 5.4]{brundanstroppelI}
The module $P(\lambda)$ is the projective cover of $L(\lambda)$,
and $\Delta(\lambda)$ is its largest quotient supported on
labels $\mu\leq\lambda$.
It has head $L(\lambda)$.
There is a finite standard filtration of $P(\lambda)$, with
$\Delta(\lambda)$ as its top quotient, and
\[
 \bigl(P(\lambda):\Delta(\mu)\langle j\rangle\bigr)
 =\dim\bigl(e_\lambda\Delta(\mu)\bigr)_j
 =
 \begin{cases}
 1,&\lambda\subset\mu,\quad
      j=\deg(\underline{\lambda}\mu),\\
 0,&\text{otherwise}.
 \end{cases}
\]
If $\operatorname{def}(\Lambda)<\infty$, then
$\operatorname{rep}(K_\Lambda)$ is an essentially finite
highest weight category with weight poset $\Lambda$.
Its graded version is positively graded.
The costandards and injective hulls are
$\nabla(\lambda)$ and $P(\lambda)^{\circledast}$.
In this case, the displayed dimension also equals
$[\Delta(\mu):L(\lambda)\langle j\rangle]$.
If $\Lambda$ is finite, $K_\Lambda$ is quasi-hereditary.
\end{theorem}

\begin{theorem}\label{thm:arc-upper-finite}
The category $K_\infty^\infty\text{-mod}_{\mathrm{lfd}}$
is an upper finite highest
weight category with weight poset $\Lambda_\infty^\infty$.
Its standard and costandard objects are $\Delta(\lambda)$ and
$\nabla(\lambda)$, and the projective cover of $L(\lambda)$ is
$P(\lambda)=K_\infty^\infty e_\lambda$, with gradings forgotten.
\end{theorem}

\begin{proof}
Since $\Lambda_\infty^\infty$ is upper finite, the projective
standard filtrations in Theorem~\ref{thm:arc-highest-weight-structure}
give the axioms of Definition~\ref{def:upper-finite-highest-weight},
and restricted duality gives the costandards.
\end{proof}

\subsection{Geometric bimodules}
\label{subsec:extended-Khovanov-algebras}
% Geometric bimodules and their functors.
% All diagram terminology is defined in 21.tex.

Let $t$ be a proper $\Lambda\Gamma$-matching.
The geometric bimodule $K_{\Lambda\Gamma}^t$ has homogeneous basis
\[
 \{(a\lambda t\mu b)\mid
   \lambda\in\Lambda,\ \mu\in\Gamma,\
   a\lambda t\mu b\text{ is oriented}\}.
\]
Its left $K_\Lambda$-action and right $K_\Gamma$-action are
those of \cite[\S3]{brundanstroppelII}.
The corner $e_\alpha K_{\Lambda\Gamma}^t e_\beta$ consists of
the basis diagrams with $a=\underline{\alpha}$ and
$b=\overline{\beta}$.
For the identity matching,
$K_{\Lambda\Lambda}^{\mathrm{id}}\cong K_\Lambda$.

Define
\[
 G_{\Lambda\Gamma}^t
 :=K_{\Lambda\Gamma}^t\langle-\operatorname{caps}(t)\rangle
       \otimes_{K_\Gamma}-.
\]
It acts from $K_\Gamma\text{-gmod}_{\mathrm{lfd}}$ to
$K_\Lambda\text{-gmod}_{\mathrm{lfd}}$.

\begin{proposition}\label{prop:geometric-functor-adjunctions}
\cite[Theorems 3.5--3.6, Corollaries 4.3--4.4 and 4.9, Theorem 4.10,
and Corollary 4.12]{brundanstroppelII}
The bimodule $K_{\Lambda\Gamma}^t$ is projective on both sides.
The functor $G_{\Lambda\Gamma}^t$ is exact, preserves
finite-dimensional modules and finitely generated projectives,
and commutes with $\circledast$.
There are graded adjunctions
\[
 G_{\Gamma\Lambda}^{t^*}
   \langle\operatorname{cups}(t)-\operatorname{caps}(t)\rangle
 \dashv G_{\Lambda\Gamma}^t
 \dashv
 G_{\Gamma\Lambda}^{t^*}
   \langle\operatorname{caps}(t)-\operatorname{cups}(t)\rangle .
\]
Thus $G_{\Lambda\Gamma}^t$ and $G_{\Gamma\Lambda}^{t^*}$
are biadjoint after forgetting the grading.

Composition of geometric functors is computed by placing
their matchings one below another, in the order in which
the functors act, starting at the top.
Suppose the resulting diagram is proper.
Remove all circles meeting neither outside boundary,
erase the intermediate number lines, and straighten the
remaining curves. Let $u$ be the matching that remains
and let $c$ count the removed circles. Then, with gradings
forgotten,
\begin{equation}\label{eq:geometric-functor-composition}
 G^{t_d}\cdots G^{t_1}\cong(G^u)^{\oplus 2^c}.
\end{equation}
Each removed circle contributes two copies, corresponding
to its two orientations. If the diagram is not proper,
the composite is zero.
\end{proposition}

\begin{proposition}\label{prop:geometric-functor-projectives}
\cite[Theorem 4.2]{brundanstroppelII}
Let $\gamma\in\Gamma$.
The module $G_{\Lambda\Gamma}^tP(\gamma)$ is zero precisely
when some upper line of $t\gamma\overline{\gamma}$ has
equally labelled rays.
Otherwise, let $\overline{\lambda}$ be the upper reduction of
$t\overline{\gamma}$ and let $n$ count the removed circles.
Then
\[
 G_{\Lambda\Gamma}^tP(\gamma)
 \cong
 P(\lambda)\otimes_\Bbbk
 \bigl(\Bbbk\langle-1\rangle\oplus\Bbbk\langle1\rangle\bigr)^{\otimes n}
 \langle\operatorname{cups}(t)-\operatorname{caps}(t)\rangle .
\]
After forgetting the grading this is $P(\lambda)^{\oplus2^n}$.
\end{proposition}

\begin{proposition}\label{prop:geometric-functor-standards}
\cite[Theorem 4.5]{brundanstroppelII}
For $\gamma\in\Gamma$, the module
$G_{\Lambda\Gamma}^t\Delta(\gamma)$ has a finite standard
filtration with sections
\[
 \Delta(\mu)\langle\deg(\mu t\gamma)-\operatorname{caps}(t)\rangle,
 \qquad
 \mu\in\Lambda,\quad \mu t\gamma\text{ oriented},
\]
each occurring once.
The filtration can be ordered with larger labels first,
starting from zero.
The module is nonzero exactly when every cup of $t\gamma$
joins opposite labels.
When nonzero, it is indecomposable, with head
$L(\lambda)\langle\deg(\lambda t\gamma)-\operatorname{caps}(t)\rangle$,
where $\lambda t\gamma$ is oriented with all caps anti-clockwise.
\end{proposition}

\begin{proposition}\label{lem:geometric-functor-simple-head}
\cite[Theorem 4.11(ii)--(iii)]{brundanstroppelII}
Let $t$ be a proper $\Lambda\Gamma$-matching and
$\gamma\in\Gamma$. Write $G^t=G^t_{\Lambda\Gamma}$.
Then $G^tL(\gamma)$ is nonzero exactly when all
cups of $t\gamma$ are anti-clockwise.
In that case let $\beta$ be the target label for which
$\beta t\gamma$ is oriented and all caps are anti-clockwise.
The module is self-dual and indecomposable, with head
$L(\beta)\langle-\operatorname{caps}(t)\rangle$ and socle
$L(\beta)\langle\operatorname{caps}(t)\rangle$.
\end{proposition}

\subsection{Linear tilting coresolution}

% Insert with \input{33} under ``Linear tilting coresolution'', after 30.tex.
% Uses the geometric-bimodule results in 30.tex.
% Bibliography: bowman2026faithful and brundanstroppelII (already in refs.bib).

Throughout this subsection, assume $\operatorname{def}(\Lambda)<\infty$.
We keep the grading shifts and left-module convention above.
Tiltings are normalized by a degree-zero inclusion
$\Delta(\lambda)\hookrightarrow T(\lambda)$.

Suppose two neighbouring vertices of $\lambda$ are labelled
$\wedge,\vee$, in this order.
Let $\lambda^-$ be obtained by interchanging these two labels,
and let $\lambda'$ be obtained by deleting the two vertices.
Let $\Gamma$ be the block of $\lambda'$ and let $t$ have a cap
joining the deleted vertices on the bottom line, with all other
vertices joined to their counterparts on the top line.
Proposition~\ref{prop:geometric-functor-standards} gives
\begin{equation}\label{eq:arc-standard-cap-sequence}
 0\longrightarrow\Delta(\lambda)
 \longrightarrow G_{\Lambda\Gamma}^t\Delta(\lambda')
 \longrightarrow\Delta(\lambda^-)\langle-1\rangle
 \longrightarrow0.
\end{equation}
The two shifts are $0$ and $-1$: the cap is clockwise for
$\lambda$ and anti-clockwise for $\lambda^-$, and the definition
of $G_{\Lambda\Gamma}^t$ includes the shift $\langle-1\rangle$.
Also $\lambda^-<\lambda$.

\begin{proposition}
\cite{bowman2026faithful}
\label{prop:arc-tiltings-from-bimodules}
The indecomposable graded tilting modules are constructed as follows.
If $\lambda$ has no neighbouring $\wedge,\vee$, put
$T(\lambda)=\Delta(\lambda)=L(\lambda)$.
Otherwise use the matching above and put
$T(\lambda)=G_{\Lambda\Gamma}^tT(\lambda')$.
The resulting module is independent, up to graded isomorphism,
of the choices. It is self-dual and has a finite standard
filtration with bottom section $\Delta(\lambda)$; every other
section is $\Delta(\mu)\langle s\rangle$ with $\mu<\lambda$
and $s<0$.
\end{proposition}

\begin{proof}
This is the graded form of the construction in
\cite[Theorem 5.1]{bowman2026faithful}.
Deleting a pair reduces the finite number of occurrences of
one of the two symbols, so the construction terminates.
At the last step the weight is minimal and its standard is simple.
The functors $G^t$ are exact, preserve standard filtrations,
and commute with $\circledast$.
Thus the modules constructed are self-dual and standard-filtered,
and hence tilting.
Sequence~\eqref{eq:arc-standard-cap-sequence}, applied successively,
gives the stated inclusion and filtration shifts.
Indeed, inserting $\wedge,\vee$ preserves the order on the
remaining vertices; the alternative insertion $\vee,\wedge$
has a smaller label and contributes the additional shift $-1$.

For indecomposability, forget the grading and inductively show
that $T(\lambda)$ has a simple socle.
Adjunction and Proposition~\ref{prop:geometric-functor-projectives}
show that $G^{t^*}L(\nu)$ is zero unless the two vertices of
$\nu$ carry $\vee,\wedge$; in that case it is $L(\nu')$,
where $\nu'$ deletes these vertices.
Indeed, its component at $e_{\mu'}$ has dimension
$\dim\operatorname{Hom}(G^tP(\mu'),L(\nu))$, which is one
precisely when $\nu$ inserts $\vee,\wedge$ into $\mu'$, and
zero otherwise.
The other adjunction gives
$\operatorname{Hom}(L(\nu),G^tT(\lambda'))
\cong\operatorname{Hom}(G^{t^*}L(\nu),T(\lambda'))$.
By induction this is one-dimensional for exactly one $\nu$.
Hence $G^tT(\lambda')$ has a simple socle and is indecomposable.
The characterization of $T(\lambda)$ by its highest standard
section proves independence of the choices.
The inclusion fixes the grading shift.
\end{proof}

\begin{lemma}\label{lem:arc-positive-tilting-morphisms}
For the tiltings in Proposition~\ref{prop:arc-tiltings-from-bimodules},
there are no homogeneous morphisms of negative degree, and
the degree-zero morphisms between $T(\lambda)$ and $T(\mu)$
form $\Bbbk$ if $\lambda=\mu$ and are zero otherwise.
\end{lemma}

\begin{proof}
The standard sections of $T(\lambda)$ have shifts at most zero,
with $\Delta(\lambda)$ the only section of shift zero.
By self-duality, the costandard sections of $T(\mu)$ have
shifts at least zero, with only $\nabla(\mu)$ in shift zero.
Use these filtrations and the highest weight identities
$\operatorname{Ext}^1_{\mathrm{gr}}(\Delta(a),\nabla(b)\langle s\rangle)=0$
and
$\operatorname{Hom}_{\mathrm{gr}}(\Delta(a),\nabla(b)\langle s\rangle)
=\delta_{a,b}\delta_{s,0}\Bbbk$.
A map from $\Delta(a)\langle s\rangle$ to
$\nabla(b)\langle t\rangle$ can have only degree $t-s\geq0$.
Degree zero requires the two distinguished sections.
\end{proof}

We call a tilting coresolution linear if its term in cohomological
degree $j$ is a finite direct sum of modules
$T(\mu)\langle-j\rangle$ and its differentials have degree zero.

\begin{theorem}\label{thm:arc-linear-tilting-finite-defect}
Every standard has a linear tilting coresolution
\begin{equation}\label{eq:arc-linear-tilting-coresolution}
 0\longrightarrow\Delta(\lambda)
 \longrightarrow T^0(\lambda)
 \longrightarrow T^1(\lambda)\longrightarrow\cdots,
 \qquad
 T^j(\lambda)\in\operatorname{add}
 \{T(\mu)\langle-j\rangle\mid\mu\leq\lambda\}.
\end{equation}
Here $T^0(\lambda)=T(\lambda)$, and all successive cokernels
have finite standard filtrations.
The coresolution is bounded when $\Lambda$ is finite or lower finite.
Applying $\circledast$ gives a linear tilting resolution of
$\nabla(\lambda)$.
\end{theorem}

\begin{proof}
We adapt the mapping-cone construction of linear projective
resolutions in \cite[Theorem 5.3]{brundanstroppelII}.
If $\Delta(\lambda)=T(\lambda)$, there is nothing to do.
Otherwise use \eqref{eq:arc-standard-cap-sequence}.
A coresolution of $\Delta(\lambda')$, followed by $G^t$, gives
a linear tilting coresolution $A^\bullet$ of
$G^t\Delta(\lambda')$: the preceding proposition gives
$G^tT(\nu')=T(\nu)$, where $\nu$ inserts $\wedge,\vee$.
A coresolution of $\Delta(\lambda^-)$, shifted by $-1$, gives
a coresolution $B^\bullet$ of $\Delta(\lambda^-)\langle-1\rangle$.

Lift the surjection in \eqref{eq:arc-standard-cap-sequence}
to a cochain map $A^\bullet\to B^\bullet$.
Such lifts exist successively because
$\operatorname{Ext}^1_{\mathrm{gr}}(M,T(\nu)\langle s\rangle)=0$
for every standard-filtered $M$.
Taking the mapping cone, shifted by $-1$, gives a coresolution
of the kernel $\Delta(\lambda)$ with terms
\[
 A^0,\qquad A^j\oplus B^{j-1}\quad(j\geq1).
\]
Both summands in degree $j$ are sums of tiltings shifted by $-j$.
The successive cokernels are extensions of those in the two
input coresolutions, so remain standard-filtered.
All labels are at most $\lambda$.

For completeness, this recursion does not require a minimal
weight in an infinite block.
Induct first on the finite number of occurrences of one of
the symbols, and within each step construct the complexes
successively in degrees $0,1,\ldots$.
To reach degree $j$, the complex for the smaller block is
already available, while the complex for $\lambda^-$ is needed
only through degree $j-1$.
The comparison maps extend the ones already chosen.
This constructs an exact complex in every degree, with finite sums.
For a finite or lower finite block one may instead induct on
the finite lower set of $\lambda$, together with the number
of symbols. The construction then terminates.
Restricted duality proves the last assertion.
\end{proof}

\subsection{Tensor product categorifications}
\label{subsec:arc-tensor-product-categorifications}
% Arc tensor product categorifications and their KLR circle maps.
% Uses the conventions and results in 21.tex, 30.tex, 32.tex and 13.tex.

Let $\Bbbk$ be algebraically closed of characteristic zero.
We give one construction for the following ordered pairs, where
$m,n\geq0$ and $s,t\in\mathbb Z$:
\begin{equation}\label{eq:arc-tpc-types}
\begin{gathered}
 (\bigwedge^mV^0,\bigwedge^nV^0),\qquad
 (L^+(\boldsymbol\varpi_s),L^+(\boldsymbol\varpi_t)),\qquad
 (\bigwedge^mV^0,\bigwedge^nV^1),\\
 (\bigwedge^mV^0,L^+(\boldsymbol\varpi_s)),\qquad
 (L^+(\boldsymbol\varpi_s),\bigwedge^mV^0),\qquad
 (L^-(-\boldsymbol\varpi_s),L^+(\boldsymbol\varpi_t)).
\end{gathered}
\end{equation}

Associate to each monomial of a factor $Z$ a subset of $\mathbb Z$:
\begin{equation}\label{eq:arc-subset-dictionary}
\begin{array}{c|c|c}
 Z&\text{monomial label}&\text{subset}\\ \hline
 \bigwedge^mV^0&\lambda_1>\cdots>\lambda_m
     &\{\lambda_1,\ldots,\lambda_m\}\\
 \bigwedge^mV^1&\lambda_1>\cdots>\lambda_m
     &\mathbb Z\setminus\{\lambda_1,\ldots,\lambda_m\}\\
 L^+(\boldsymbol\varpi_s)&\nu\in\mathscr P
     &\{s+1-a+\nu_a\mid a\geq1\}\\
 L^-(-\boldsymbol\varpi_s)&\nu\in\mathscr P
     &\{s+a-\nu_a\mid a\geq1\}
\end{array}
\end{equation}
Missing parts of a partition are zero.
In every row, $F_i$ replaces $i$ by $i+1$ when $i$ belongs
to the subset and $i+1$ does not; otherwise it is zero.
The action of $E_i$ reverses this replacement.
Every nonzero coefficient is one. For $V^1$ this uses the
normalization $v_i^1(v_j^0)=(-1)^i\delta_{ij}$ fixed earlier.

For a monomial $v_\lambda^X\otimes v_\mu^Y$, denote its two
subsets by $S,T$ and label vertex $j$ by
\[
 \vee\ \text{if }j\in S\setminus T,\qquad
 \wedge\ \text{if }j\in T\setminus S,\qquad
 \times\ \text{if }j\in S\cap T,\qquad
 \circ\ \text{if }j\notin S\cup T.
\]
Write $\Lambda(X,Y)$ for the resulting diagrams.
Two such diagrams belong to the same block if they differ by
a finite permutation of the $\vee,\wedge$ labels, with the
$\circ,\times$ vertices fixed. Write $\operatorname{Bl}(X,Y)$
for these blocks. For each block $\Gamma$, delete the fixed
vertices and index the remaining vertices in order. Write
$K_\Gamma$ for the algebra of the resulting block, using
the definitions above.
For a diagram $\lambda$, write $v_\lambda$ for its tensor
monomial. Define
\[
 K(X,Y):=\bigoplus_{\Gamma\in\operatorname{Bl}(X,Y)}K_\Gamma.
\]

\begin{lemma}\label{lem:arc-tpc-blocks}
Blocks in $\Lambda(X,Y)$ are exactly the sets of labels with
the same total weight. Their Bruhat order is the inverse
dominance order on $\Xi_X\times\Xi_Y$.
Their algebras have the following forms:
\begin{equation}\label{eq:arc-tpc-block-table}
\begin{array}{c|c|l}
 (X,Y)&K_\Gamma&\text{parameters}\\ \hline
 (\bigwedge^mV^0,\bigwedge^nV^0)
    &K_a^b&a=m-r,\ b=n-r,\ r=\#\times\\
 (L^+(\boldsymbol\varpi_s),L^+(\boldsymbol\varpi_t))
    &K_a^b&a=\#\vee,\ b=\#\wedge,\ a-b=s-t\\
 (\bigwedge^mV^0,\bigwedge^nV^1)
    &K_a^\infty&a=\#\vee\leq\min(m,n)\\
 (\bigwedge^mV^0,L^+(\boldsymbol\varpi_s))
    &K_a^{-\infty}&a=\#\vee\leq m\\
 (L^+(\boldsymbol\varpi_s),\bigwedge^mV^0)
    &K_b^{+\infty}&b=\#\wedge\leq m\\
 (L^-(-\boldsymbol\varpi_s),L^+(\boldsymbol\varpi_t))
    &K_\infty^\infty&
\end{array}
\end{equation}
All algebra identifications preserve the grading.
\end{lemma}

\begin{proof}
The total weight records the sum of the two subset indicators,
modulo constants; the fixed tails remove this ambiguity.
It therefore fixes the $\circ,\times$ vertices, while the fixed
cardinalities or charges make the remaining labels differ by
a finite permutation. This identifies the blocks.

On a finite interval containing all differences, dominance of
the first subsets is determined by their cumulative counts.
Thus increasing inverse dominance moves $\vee$ to the right
past $\wedge$, exactly as in Bruhat order.

After deleting fixed vertices, the first two rows have finitely
many vertices; the next two have finitely many $\vee$ labels
on a two-sided or left-infinite line, respectively.
Their symbol counts give the displayed parameters.
For the fifth row, reverse the left-infinite line and interchange
$\vee,\wedge$ to obtain $\Lambda_b^{+\infty}$.
For the last, reindex and translate the charge to obtain
$\Lambda_\infty^\infty$. These operations preserve orientations,
degrees and surgery, hence give graded algebra isomorphisms.
\end{proof}

For the first five pairs, let
$\mathcal C(X,Y)=\operatorname{rep}(K(X,Y))$, the category
of finite-dimensional ungraded left modules.
For the last pair, let $\mathcal C(X,Y)$ be the full category
of unital left $K(X,Y)$-modules with $\dim e_\lambda M<\infty$
for every label $\lambda$.
Theorems~\ref{thm:arc-highest-weight-structure} and
\ref{thm:arc-upper-finite} make these categories respectively
essentially finite and upper finite highest weight categories.
In the fifth case the weight poset is also lower finite.
We use the left-module version of the ambient category
conventions in Definition~\ref{def:categorical-slI-action}.

To define $F_i$, use the following elementary matchings;
see \cite[(3.11)]{BrundanStroppelIII}.
The diagram of a block replaces each $\vee,\wedge$ label
by $\bullet$ and retains its fixed labels. At positions $i,i+1$,
the top and bottom block diagrams and the nonvertical part
of the matching are
\begin{equation}\label{eq:elementary-arc-matchings}
\begin{array}{c|c|l}
 \text{top}&\text{bottom}&\text{matching}\\ \hline
 \bullet\ \bullet&\circ\ \times&\text{a cup at the top}\\
 \times\ \circ&\bullet\ \bullet&\text{a cap at the bottom}\\
 \bullet\ \circ&\circ\ \bullet&\text{a segment from top }i
                                      \text{ to bottom }i+1\\
 \times\ \bullet&\bullet\ \times&\text{a segment from top }i+1
                                      \text{ to bottom }i
\end{array}
\end{equation}
All other $\bullet$ vertices are joined vertically.
No curve meets a $\circ$ or $\times$ vertex.
Deleting those fixed vertices gives a crossingless matching
in the sense defined above.
We use a matching only when it is proper. In particular, the
first row requires both a $\vee$ and a $\wedge$ in the top block.
Write $t_i(\Gamma)$ for the matching with top block $\Gamma$,
when it exists.

Identify a block $\Gamma$ with its total weight.
When $\Gamma-\boldsymbol\alpha_i$ is a block, use the
matching $t_i(\Gamma)$ from
\eqref{eq:elementary-arc-matchings} and set
\[
 F_i\big|_{K_\Gamma\text{-mod}_{\mathrm{lfd}}}
   :=G^{t_i(\Gamma)}_{\Gamma-\boldsymbol\alpha_i,\Gamma},
 \qquad
 E_i\big|_{K_{\Gamma-\boldsymbol\alpha_i}\text{-mod}_{\mathrm{lfd}}}
   :=G^{t_i(\Gamma)^*}_{\Gamma,\Gamma-\boldsymbol\alpha_i}.
\]
Use zero when the required block does not exist.
These formulas first define functors on locally
finite-dimensional modules, with gradings forgotten.
By Proposition~\ref{prop:geometric-functor-adjunctions}, they
are exact, biadjoint and preserve finite-dimensional modules
and finitely generated projectives. Thus they act on the
chosen categories $\mathcal C(X,Y)$.

\begin{lemma}\label{lem:arc-standard-action}
For a standard module $\Delta(\lambda)$, the standard sections
of $F_i\Delta(\lambda)$ are precisely the monomials obtained
by replacing $i$ by $i+1$ in one of its two subsets, each once.
For $E_i$ the replacement is reversed.
In particular, these filtrations have at most two sections.
The families $(E_i)$ and $(F_i)$ are locally finite in the
sense of Definition~\ref{def:categorical-slI-action}.
\end{lemma}

\begin{proof}
By Proposition~\ref{prop:geometric-functor-standards}, a
standard section is specified by an orientation of the
elementary matching with the given top label.
A segment moves its label to the other endpoint; a cup
requires opposite labels; and a cap allows the two choices
of opposite labels. These are exactly the stated replacements.
Reflection proves the assertion for $E_i$.

Each monomial allows only finitely many replacements, since
each subset is constant sufficiently far in either direction.
Each $P(\lambda)$ has a finite standard filtration, so
$E_iP(\lambda)$ and $F_iP(\lambda)$ vanish for all but
finitely many $i$.
In the finite-dimensional categories, every object is a
quotient of a finite sum of projectives, which proves local
finiteness there.
In the full locally finite-dimensional category, use
\[
 e_\lambda F_iM\cong
 \operatorname{Hom}(P(\lambda),F_iM)
 \cong\operatorname{Hom}(E_iP(\lambda),M).
\]
The last space is finite-dimensional because
$E_iP(\lambda)$ is a finite sum of indecomposable projectives.
It vanishes for all but finitely many $i$. The same argument
applies to $E_i$.
\end{proof}

We use the local bimodule maps of
\cite[\S5]{BrundanStroppelIII} to equip these functors with
the KLR action. We record explicitly the circle map and its
sign, which enter the centre calculation below.

For a source block $\Gamma$ and $i\in\mathbb Z$, choose an
even integer $b<i$ such that the block diagram is constant
at every vertex $j\leq b$, and set
\[
 \epsilon_i(\Gamma)=
 (-1)^{\#\{j\mid b<j\leq i,\ \Gamma_j=\bullet\}}.
\]
The choice is immaterial: changing the even cutoff adds
either no $\bullet$ vertices or an even number of them.
An elementary $j$-move gives
\begin{equation}\label{eq:arc-sign-change}
 \epsilon_i(\Gamma-\boldsymbol\alpha_j)
   =(-1)^{\delta_{ij}}\epsilon_i(\Gamma).
\end{equation}

Let
$d=\underline{\lambda}\,\alpha\,t_i(\Gamma)\,
\beta\,\overline{\mu}$
be a basis diagram of the geometric bimodule defining
$F_i|_\Gamma$.
The nonvertical part of $t_i(\Gamma)$ lies in a unique
connected component of this entire diagram.
Define $y_i^\Gamma(d)$ by changing that component from
an anti-clockwise circle to a clockwise circle.
If it is a line or an already clockwise circle, set
$y_i^\Gamma(d)=0$.
In the nonzero case, the underlying curves are unchanged:
only the orientation of the indicated circle is reversed,
and every other component retains its labels.

Equivalently, attach a small clockwise circle to the
indicated component and merge it by surgery.
In the notation of \eqref{eq:arc-circle-surgery}, this is
multiplication by $z$:
\[
 1\longmapsto z,\qquad
 z\longmapsto0,\qquad
 \mathsf l\longmapsto0.
\]
Thus $y_i^\Gamma$ has degree two and square zero.
The index $i$ specifies the elementary matching at
positions $i,i+1$; the component on which the map acts
is determined only after the matching is joined to
the surrounding cup and cap diagrams.
The sign needed for the KLR action is introduced
separately below.
For a module $M$ in block $\Gamma$ and $v\in M$, define
\begin{equation}\label{eq:arc-klr-maps}
 \xi_i(d\otimes v)
   =\epsilon_i(\Gamma)y_i^\Gamma(d)\otimes v.
\end{equation}
We also write $\xi_i$ for the signed bimodule map
$\epsilon_i(\Gamma)y_i^\Gamma$.
For an absent elementary matching the map is zero.

\begin{lemma}\label{lem:arc-klr-relations}
The maps $\xi_i$ are natural transformations of degree two
with $\xi_i^2=0$. There are natural transformations
$\tau_{ij}:F_jF_i\to F_iF_j$ such that their insertions,
together with the $\xi_i$ and the projections onto
$F_{i_d}\cdots F_{i_1}$, define the homomorphisms
$R_d\to\operatorname{End}(F^d)$ of
Definition~\ref{def:type-a-klr-action}.
These statements hold on every category $\mathcal C(X,Y)$
in \eqref{eq:arc-tpc-types}.
\end{lemma}

\begin{proof}
The map $y_i^\Gamma$ can be computed by attaching a clockwise
circle to the indicated component and applying surgery;
see \cite[(5.5)--(5.6) and the discussion following them]{BrundanStroppelIII}.
Independence of the order of surgery makes this a bimodule
homomorphism. It therefore induces a natural transformation
after tensoring with a module. Changing an anti-clockwise
circle to clockwise raises the degree by two. Applying the
map again gives zero, so $(y_i^\Gamma)^2=0$ and $\xi_i^2=0$.

For the remaining generators, let $\psi(ij)$ be the maps of
\cite[(5.7)--(5.8)]{BrundanStroppelIII}, using
$\epsilon_i(\Gamma)$ for the sign there.
Take $\tau_{ij}=-\psi(ij)$ when $j=i+1$, and
$\tau_{ij}=\psi(ij)$ otherwise.
For finite diagrams, \cite[Lemmas 5.4--5.5 and Theorem 5.6]{BrundanStroppelIII}
give the required relations. The indicated sign change makes
their adjacent-residue quadratic and braid relations agree
with Definition~\ref{def:klr-from-hecke}; the equal-residue
and distant-residue relations are unchanged.

For an infinite diagram, each calculation involves finitely
many elementary matchings and finitely many cups and caps.
Choose an even left cutoff and a right cutoff outside all
these pieces. Outside the resulting interval the diagrams
have only unchanged rays and vertical segments.
The calculation inside the interval has the same surgery
rules and the same signs $\epsilon_i(\Gamma)$ as above,
so the finite relations apply.
They are identities of bimodule maps before tensoring, and
hence hold on every module in the stated categories.
\end{proof}

\begin{theorem}\label{thm:arc-tensor-product-categorifications}
For every pair in \eqref{eq:arc-tpc-types}, the category
$\mathcal C(X,Y)$ with the functors and transformations above
is a level-two tensor product categorification of type $(X,Y)$.
We use its categorical action in the KLR form of
Definition~\ref{def:type-a-klr-action}.
Its standard-basis identification is
\[
 \Bbbk\otimes_{\mathbb Z}K_0(\mathcal C(X,Y)^\Delta)
 \xrightarrow{\sim}X\otimes Y,
 \qquad [\Delta(\lambda)]\longmapsto v_\lambda.
\]
\end{theorem}

\begin{proof}
The highest weight structure, exactness and biadjunction are
established above. Lemma~\ref{lem:arc-standard-action} gives
{\rm (TPC1)--(TPC2)} and local finiteness, while
Lemma~\ref{lem:arc-klr-relations} gives the KLR relations.
It remains to verify {\rm (CA4)}.

The finite, unitriangular standard expansions of projectives
embed their split Grothendieck group into
$K_0(\mathcal C(X,Y)^\Delta)$. Hence the Lie relations on the
standard basis hold on projectives, whose classes are weight
vectors. Moreover, $E_i^3$ and $F_i^3$ vanish on standards,
since each subset can cross $i,i+1$ at most once in either
fixed direction. They therefore vanish on projectives,
proving integrability and {\rm (CA4)} in the upper finite case.

For finite-dimensional modules, the Hom pairing makes
projective and simple classes dual. Biadjunction transposes
$E_i,F_i$, transferring the Lie relations to
$K_0(\mathcal C(X,Y))$ with the prescribed simple weights.
Quotients of projectives inherit the same third-power vanishing,
so this action is integrable too.
\end{proof}

We forget gradings in the following three corollaries.

\begin{corollary}\label{cor:arc-pm-ringel-duality}
For every $n\geq0$, the categories
$\operatorname{rep}(K_n^{-\infty})$ and
$\operatorname{rep}(K_n^{+\infty})$, with their specified
highest weight structures, are mutually Ringel dual
in the essentially finite sense.
\end{corollary}

\begin{proof}
Put $X=\bigwedge^nV^0$ and $Y=L^+(\boldsymbol\varpi_0)$.
In every block of $\mathcal C(X,Y)$, the construction in
Proposition~\ref{prop:arc-tiltings-from-bimodules}
eventually removes all $\vee$ labels.
The remaining algebra is $\Bbbk$, so
Proposition~\ref{prop:geometric-functor-projectives}
shows that every indecomposable tilting is projective.
Since the left and right principal ideals of these arc
algebras are finite-dimensional, the tilting Hom matrix
has finite support in each row and column.
Thus $\mathcal C(X,Y)$ is tilting-bounded.

By Lemma~\ref{lem:tpc-standard-costandard-sequences},
the biadjoint functors $E_i,F_i$ preserve tiltings.
Their action and KLR transformations therefore induce
a categorical action on the Ringel dual.
Its standards correspond to the original costandards,
and exchanging the two tensor factors reverses inverse
dominance. Hence it is a tensor product categorification
of type $(Y,X)$.
Theorems~\ref{thm:level-two-tpc-uniqueness}
and~\ref{thm:arc-tensor-product-categorifications}
identify it with $\mathcal C(Y,X)$.
Restricting to corresponding blocks with $n$ free
$\vee$ labels, using Lemma~\ref{lem:arc-tpc-blocks},
proves the assertion.
\end{proof}

\begin{corollary}\label{cor:arc-plus-maximal-defect-corner}
Put
\[
 \Lambda_n^{-\infty,\circ}
 :=\{\lambda\in\Lambda_n^{-\infty}
       \mid \operatorname{def}(\lambda)=n\},
 \qquad
 H_n^{-\infty}
 :=\bigoplus_{\alpha,\beta\in\Lambda_n^{-\infty,\circ}}
      e_\alpha K_n^{-\infty}e_\beta.
\]
Then $H_n^{-\infty}\cong K_n^{+\infty}$ as locally
unital algebras.
\end{corollary}

\begin{proof}
The preceding proof shows that every indecomposable
tilting for $K_n^{-\infty}$ is projective.
Conversely, by \cite[Theorem~6.1]{brundanstroppelII},
the indecomposable projective-injectives are precisely
the $P(\lambda)$ with
$\lambda\in\Lambda_n^{-\infty,\circ}$.
They are self-dual after forgetting the grading and
have standard filtrations, so they are tilting.
Consequently, the opposite of the finite-entry
endomorphism algebra of a basic tilting generator
is $H_n^{-\infty}$.
Corollary~\ref{cor:arc-pm-ringel-duality} identifies its
module category with $\operatorname{rep}(K_n^{+\infty})$.
Both algebras are basic, so matching their indecomposable
projectives gives the claimed algebra isomorphism.
\end{proof}

\begin{corollary}\label{cor:arc-leonard-super-duality}
Assume $\Bbbk=\mathbb C$ and $n\geq1$.
For $\epsilon\in\{0,1\}$, let
$\mathcal O_\epsilon^{++}(n)$ be Leonard's category
$\mathcal O_\epsilon^{++}$ from
\cite[\S4.1]{leonard2020graded}, with
$l=1$, $n_1=n$ and $c_1=0$.
These are the specified infinite-rank parabolic categories
for $\mathfrak{gl}(n+\infty)$ and
$\mathfrak{gl}(n\mid\infty)$, respectively, with Levi
$\mathfrak{gl}_n\oplus\mathfrak{gl}_\infty$.
There are strongly equivariant highest weight equivalences
\[
 \mathcal O_0^{++}(n)
 \simeq
 \mathcal C\bigl(\bigwedge^nV^0,
                 L^+(\boldsymbol\varpi_0)\bigr)
 \simeq
 \mathcal O_1^{++}(n).
\]
Every block is therefore equivalent to
$\operatorname{rep}(K_a^{-\infty})$, where $0\leq a\leq n$
is the number of $\vee$ labels in its weight diagrams.
The resulting equivalence between the two parabolic
categories is classical super duality.
\end{corollary}

\begin{proof}
By \cite[Proposition~3.5 and Theorem~4.23]{leonard2020graded},
both categories have tensor product type
$(\bigwedge^nV^0,L^+(\boldsymbol\varpi_0))$ under the
monomial identification of the two semi-infinite wedges.
The equivalences follow from
Theorems~\ref{thm:level-two-tpc-uniqueness}
and~\ref{thm:arc-tensor-product-categorifications},
and the block description follows from
Lemma~\ref{lem:arc-tpc-blocks}.
Their label-preserving composite is the super duality
equivalence of \cite[Corollary~4.24]{leonard2020graded}.
\end{proof}

\section{Results of Brundan--Stroppel V }
\label{sec:Brundan-Stroppel-V}

\subsection{Representations of the oriented Brauer category}
\label{subsec:representations-oriented-Brauer}
Let $\Bbbk$ be algebraically closed of characteristic zero and
fix $\delta\in\mathbb Z$. We use left modules.

\begin{definition}\label{def:oriented-brauer-category}
\cite[\S III.1]{reynolds2015oriented}
The oriented Brauer category $\mathcal{OB}(\delta)$ has finite
words in $\uparrow,\downarrow$ as objects, including
$\varnothing$.
A basis of $\operatorname{Hom}(\mathbf a,\mathbf b)$ consists
of matchings of the bottom vertices $\mathbf a$ and top vertices
$\mathbf b$. A strand joins equal arrows on different boundaries,
or opposite arrows on the same boundary.
Crossings carry no extra data.
Composition is vertical concatenation, with each closed component
replaced by $\delta$; tensor product is horizontal concatenation.

We also write $\mathcal{OB}(\delta)$ for the locally unital algebra
$\bigoplus_{\mathbf a,\mathbf b}\operatorname{Hom}(\mathbf a,\mathbf b)$,
with multiplication given by composition when defined and zero
otherwise, and idempotents $1_{\mathbf a}$.
Its category $\mathcal{OB}(\delta)\text{-}\operatorname{mod}_{\mathrm{lfd}}$
consists of unital left modules
$M=\bigoplus_{\mathbf a}1_{\mathbf a}M$ with finite-dimensional
components.
A diagram $d:\mathbf a\to\mathbf b$ acts from
$1_{\mathbf a}M$ to $1_{\mathbf b}M$ by $v\mapsto dv$.
The walled Brauer algebra is
$B_{r,s}(\delta)=\operatorname{End}_{\mathcal{OB}(\delta)}
(\uparrow^r\downarrow^s)$.
\end{definition}

Let $\mathcal{OB}^0$ be the subalgebra spanned by diagrams
whose strands all join the two boundaries, and
$\mathcal{OB}^\sharp$ the subalgebra spanned by diagrams
with no pairs on the top boundary.
For words with $r$ upward and $s$ downward arrows,
the corresponding summand of $\mathcal{OB}^0$ is Morita
equivalent to $\Bbbk S_r\otimes\Bbbk S_s$;
see \cite[\S\S III.2--III.4]{reynolds2015oriented}.
For a bipartition $\lambda=(\lambda^L,\lambda^R)$, let
$S(\lambda)$ correspond to
$S^{\lambda^L}\otimes S^{\lambda^R}$, with $\lambda^L$
labelling upward arrows.
Inflate $S(\lambda)$ to $\mathcal{OB}^\sharp$ by letting
diagrams with a bottom pair act as zero, and put
\[
 \Delta(\lambda)=
 \mathcal{OB}(\delta)\otimes_{\mathcal{OB}^\sharp}S(\lambda).
\]

Equip the bipartitions with the inverse dominance order
$\preceq$ of Section~\ref{subsec:semi-infinite-wedges},
using the ordered pair $(\lambda^R,\lambda^L)$ to index
$L^-(-\boldsymbol\varpi_{-\delta})\otimes
L^+(\boldsymbol\varpi_0)$.

\begin{lemma}\label{lem:ob-standards}
\cite[Theorem V.2.5]{reynolds2015oriented}
The category
$\mathcal{OB}(\delta)\text{-}\operatorname{mod}_{\mathrm{lfd}}$
is an upper finite highest weight category with standards
$\Delta(\lambda)$ and order $\preceq$.
\end{lemma}

Write $L(\lambda)$ for the simple head of $\Delta(\lambda)$
and $P(\lambda)$ for its projective cover.

For a word $\mathbf a$, write $\mathbf a\uparrow$ and
$\mathbf a\downarrow$ for the words obtained by appending
one arrow on the right. Define $F,E$ by
\[
 1_{\mathbf a}(FM)=1_{\mathbf a\downarrow}M,
 \qquad
 1_{\mathbf a}(EM)=1_{\mathbf a\uparrow}M.
\]
A diagram $d:\mathbf a\to\mathbf b$ acts on these components
by $d\otimes1_{\downarrow}$ and $d\otimes1_{\uparrow}$,
respectively; module homomorphisms act componentwise.
We fix the biadjunctions given by the oriented cups and caps.
On representables, bending the appended strand gives
\[
 F\bigl(\mathcal{OB}(\delta)1_{\mathbf a}\bigr)
       \cong\mathcal{OB}(\delta)1_{\mathbf a\uparrow},
 \qquad
 E\bigl(\mathcal{OB}(\delta)1_{\mathbf a}\bigr)
       \cong\mathcal{OB}(\delta)1_{\mathbf a\downarrow};
\]
see \cite[Proposition IV.1.1]{reynolds2015oriented}.

For $\mathbf a=a_1\cdots a_q$, let $s_{j,q+1}$ interchange
the indicated equal arrows. For opposite arrows, let
$e_{j,q+1}$ join them on each boundary, with all other
strands vertical. Put
\begin{equation}\label{eq:ob-jucys-murphy}
\begin{aligned}
 J_{\mathbf a\uparrow}
   &=\sum_{a_j=\uparrow}s_{j,q+1}
       -\sum_{a_j=\downarrow}e_{j,q+1},\\
 J_{\mathbf a\downarrow}
   &=-\delta\,1_{\mathbf a\downarrow}
       -\sum_{a_j=\downarrow}s_{j,q+1}
       +\sum_{a_j=\uparrow}e_{j,q+1}.
\end{aligned}
\end{equation}
Empty sums are zero. This is the normalization with
$m=0$, $m'=-\delta$ in
\cite[\S\S III.2 and IV.2]{reynolds2015oriented}, with
strands appended on the right.
On $1_{\mathbf a}(FM)$ define
$x_M(v)=J_{\mathbf a\downarrow}v$.
Its adjunction mate on $1_{\mathbf a}(EM)$ is
$v\mapsto J_{\mathbf a\uparrow}v$.
On $1_{\mathbf a}(F^2M)=1_{\mathbf a\downarrow\downarrow}M$,
define $t_M(v)=s_{q+1,q+2}v$.
All these products use the original diagram action on $M$.

On $1_{\mathbf a}(F^dM)=1_{\mathbf a\downarrow^d}M$,
number the appended strands from right to left, as in
Definition~\ref{def:categorical-slI-action}.
Then $t_j$ interchanges appended strands $j,j+1$, and
\[
 x_jv=
 \left(J_{\mathbf a\downarrow^{d-j+1}}
               \otimes1_{\downarrow^{j-1}}\right)v
 \qquad(1\leq j\leq d).
\]

For $i\in\mathbb Z$, define $1_{\mathbf a}(F_iM)$ to be
the vectors $v\in1_{\mathbf a\downarrow}M$ satisfying
$(J_{\mathbf a\downarrow}-i1_{\mathbf a\downarrow})^Nv=0$
for some $N\geq1$.
Define $1_{\mathbf a}(E_iM)$ by the same condition with
$\uparrow$ in place of $\downarrow$.

For a box $\square$ in row $a$ and column $b$, put
$\operatorname{ct}(\square)=b-a$.

\begin{theorem}\label{thm:oriented-brauer-tpc}
\cite[Theorem V.3.2]{reynolds2015oriented}
The category
$\mathcal{OB}(\delta)\text{-}\operatorname{mod}_{\mathrm{lfd}}$,
with the functors and transformations above, is a tensor
product categorification of
\[
 L^-(-\boldsymbol\varpi_{-\delta})
       \otimes L^+(\boldsymbol\varpi_0),
 \qquad
 [\Delta(\lambda)]\longmapsto
 v_{\lambda^R}^-\otimes v_{\lambda^L}^+.
\]
For the ambient convention in
Definition~\ref{def:categorical-slI-action}{\rm (ii)},
use the opposite of a basic algebra Morita equivalent to
$\mathcal{OB}(\delta)$.
\end{theorem}

Put
\[
 K(\delta)=K\bigl(L^-(-\boldsymbol\varpi_{-\delta}),
                    L^+(\boldsymbol\varpi_0)\bigr).
\]
The bipartition $\lambda=(\lambda^L,\lambda^R)$ corresponds
to the weight diagram with subsets
\[
 I_\vee(\lambda)=\{a-\delta-\lambda^R_a\mid a\geq1\},
 \qquad
 I_\wedge(\lambda)=\{1-a+\lambda^L_a\mid a\geq1\}.
\]
This is the algebra $K(\delta)$ of
\cite[\S1]{BrundanStroppelV}, the sum of its
$K_\infty^\infty$-blocks.

\begin{corollary}\label{cor:ob-arc-equivalence}
There is a label-preserving strongly equivariant equivalence
\[
 \mathcal G:\mathcal{OB}(\delta)\text{-}\operatorname{mod}_{\mathrm{lfd}}
       \xrightarrow{\sim}K(\delta)\text{-}\operatorname{mod}_{\mathrm{lfd}}.
\]
\end{corollary}

Let $Q_\varnothing=K(\delta)e_{\varnothing,\varnothing}$.
For a word $\mathbf a$, define
$Q_{\mathbf a\uparrow}=FQ_{\mathbf a}$ and
$Q_{\mathbf a\downarrow}=EQ_{\mathbf a}$, using the total
arc functors $F=\bigoplus_iF_i$ and $E=\bigoplus_iE_i$.
The empty-word representable is
$\Delta(\varnothing,\varnothing)=P(\varnothing,\varnothing)$;
strong equivariance and the formulas on representables give
$\mathcal G(\mathcal{OB}(\delta)1_{\mathbf a})\cong Q_{\mathbf a}$.

For $r,s\geq0$, set
\[
 \Lambda_{r,s}=
 \{\lambda\mid |\lambda^L|=r-k,\ |\lambda^R|=s-k,\
                    0\leq k\leq\min(r,s)\}.
\]
Let $\dot\Lambda_{r,s}=\Lambda_{r,s}$, except that
$(\varnothing,\varnothing)$ is removed when
$\delta=0$ and $r=s>0$.
Put
\[
 e_{r,s}=\sum_{\lambda\in\dot\Lambda_{r,s}}e_\lambda,
 \qquad
 K_{r,s}(\delta)=e_{r,s}K(\delta)e_{r,s},
 \qquad
 Q_{r,s}=E^sF^rQ_\varnothing.
\]
Thus $Q_{r,s}=Q_{\uparrow^r\downarrow^s}$.

For $\lambda\in\dot\Lambda_{r,s}$, write
$D_{r,s}(\lambda):=1_{\uparrow^r\downarrow^s}L(\lambda)$ for
the corresponding simple walled Brauer module
\cite[Theorem VI.3.1]{reynolds2015oriented}.

\begin{corollary}\label{cor:walled-brauer-arc-truncation}
There are algebra isomorphisms
\[
 B_{r,s}(\delta)
 \cong \operatorname{End}_{K(\delta)}(Q_{r,s})^{\mathrm{op}}
 \cong \operatorname{End}_{K_{r,s}(\delta)}
                       (e_{r,s}Q_{r,s})^{\mathrm{op}}.
\]
The module $e_{r,s}Q_{r,s}$ is a projective generator.
Consequently $B_{r,s}(\delta)$ is Morita equivalent to
$K_{r,s}(\delta)$, with simple labels $\dot\Lambda_{r,s}$.
\end{corollary}

\begin{proof}
By \cite[Theorem VI.3.1]{reynolds2015oriented}, the labels with
$1_{\uparrow^r\downarrow^s}L(\lambda)\neq0$ are exactly
$\dot\Lambda_{r,s}$. Since these dimensions are the projective
multiplicities in the corresponding representable,
Corollary~\ref{cor:ob-arc-equivalence} gives
$\operatorname{add}(Q_{r,s})
 =\operatorname{add}(K(\delta)e_{r,s})$.
The functor $M\mapsto e_{r,s}M$ is fully faithful on this
additive category and sends $Q_{r,s}$ to a projective generator.
The result follows from that corollary and Morita theory.
\end{proof}

Replace each $\uparrow$ in $\mathbf a$ by the letter $E$ of
\cite{BrundanStroppelV}, and each $\downarrow$ by its letter $F$.
The resulting sequence $R$ satisfies
$Q_{\mathbf a}\cong{}_R^*P(\eta)$ after forgetting the grading:
both sides use the same matching bimodules, and
$\eta=(\varnothing,\varnothing)$.
Thus $\operatorname{End}_{K(\delta)}(Q_{\mathbf a})^{\mathrm{op}}$
is the algebra $B_R(\delta)$ defined in
\cite[(6.3)]{BrundanStroppelV}.
Thus the last two corollaries recover the ungraded isomorphism
of \cite[Corollary 8.2]{BrundanStroppelV} and the Morita
equivalence of \cite[Theorem 1.1]{BrundanStroppelV}.
The grading on the arc projective functors also grades this
endomorphism algebra.
In general $B_{r,s}(\delta)$ is not isomorphic to its basic
corner $K_{r,s}(\delta)$: the projective multiplicities in
$Q_{r,s}$ must be retained.

\subsection{Mixed super Schur--Weyl duality}
\label{subsec:mixed-tensors}
Let $m,n\geq0$, put $\delta=m-n$, and work over the
algebraically closed field $\Bbbk$ of characteristic zero.
We first identify Deligne's category with the projective
category of $K(\delta)$, and realize $\mathcal F(m|n)$ by
arc modules. Evaluation then gives a strongly equivariant
functor between these two arc categories. We use it to
describe mixed tensor powers and their indecomposable summands.

Now let $\mathfrak g=\mathfrak{gl}(m|n)$, write
$V=\Bbbk^{m|n}$ for its natural supermodule, and put $W=V^*$.
For a homogeneous vector $v$, write $|v|$ for its parity.
All homomorphisms below are even.
Let $\mathcal F(m|n)$ be the category of finite-dimensional
$\mathfrak g$-supermodules which are direct sums of integral
weight spaces, with the weight space of
$\nu=(\nu_1,\ldots,\nu_{m+n})$ concentrated in parity
$\nu_{m+1}+\cdots+\nu_{m+n}\pmod2$.
This chooses one of the two parity shifts of each simple module.
The category contains $\Bbbk,V,W$ and is closed under tensor
products and duals.

Let $e_{ab}$ be the matrix units, with $|a|=0$ for $a\leq m$
and $|a|=1$ otherwise. On $\mathcal F(m|n)$ put
\[
 F(M)=M\otimes V,\qquad E(M)=M\otimes W,\qquad
 \Omega=\sum_{a,b=1}^{m+n}(-1)^{|b|}e_{ab}\otimes e_{ba}.
\]
Take $x_M=\Omega_{M,V}$, where the subscripts indicate the
two modules on which $\Omega$ acts, and let $t_M$ interchange
the last two factors by the super symmetry.
Use the usual tensor adjunction $F\dashv E$; the mate of $x$
on $E$ is $-\delta-\Omega_{M,W}$.
Write $F_i,E_i$ for the corresponding generalized
$i$-eigenspace functors.

Since $\operatorname{sdim}V=\delta$, the universal property of
$\mathcal D_\delta$ gives a symmetric monoidal functor
$\operatorname{ev}_{m,n}:\mathcal D_\delta\to\mathcal F(m|n)$
sending $\uparrow$ to $V$ and $\downarrow$ to $W$
\cite[Proposition 3.5.1]{comes2012deligne}.
It evaluates crossings by the super symmetry and cups and
caps by coevaluation and evaluation.
For a word $\mathbf a$, set
$T_{\mathbf a}=\operatorname{ev}_{m,n}(\mathbf a)$, and write
$T_{r,s}=V^{\otimes r}\otimes W^{\otimes s}$.

We next recall the arc realization of $\mathcal F(m|n)$.
For the upper triangular Borel, the simple objects
$L_{\mathfrak g}(\nu)$ are indexed by
\[
 \Xi(m|n)=\{\nu\in\mathbb Z^{m+n}\mid
       \nu_1\geq\cdots\geq\nu_m,\quad
       \nu_{m+1}\geq\cdots\geq\nu_{m+n}\}.
\]
Let $\mathfrak p=\mathfrak g_{\overline0}\oplus\mathfrak g_1$,
where $\mathfrak g_1$ is the upper right odd block.
Extend the simple $\mathfrak g_{\overline0}$-module
$L_{\overline0}(\nu)$ to $\mathfrak p$ by making
$\mathfrak g_1$ act as zero, and define
$\Delta_{\mathfrak g}(\nu)
=U(\mathfrak g)\otimes_{U(\mathfrak p)}L_{\overline0}(\nu)$.
Its highest weight vector has the parity specified above.

Identify $\nu$ with the weight diagram given by
\begin{equation}\label{eq:glmn-weight-diagram}
 I_\vee(\nu)=\{\nu_a+1-a\mid1\leq a\leq m\},\qquad
 I_\wedge(\nu)=
   \mathbb Z\setminus\{j-m-\nu_{m+j}\mid1\leq j\leq n\}.
\end{equation}
These are the subsets for monomials in
$\bigwedge^mV^0$ and $\bigwedge^nV^1$, respectively, under
\eqref{eq:arc-subset-dictionary}. Denote their tensor monomial
by $v_\nu$.

\begin{theorem}\label{thm:glmn-tensor-product-categorification}
\cite[Theorems 3.8 and 3.10]{brundan2017tensor}
The category $\mathcal F(m|n)$ is essentially finite highest
weight, with standards $\Delta_{\mathfrak g}(\nu)$ and the
inverse dominance order on the monomials $v_\nu$.
The functors and transformations above make it a tensor
product categorification of type
$(\bigwedge^mV^0,\bigwedge^nV^1)$, with
\[
 \Bbbk\otimes_{\mathbb Z}K_0(\mathcal F(m|n)^\Delta)
 \xrightarrow{\sim}\bigwedge^mV^0\otimes\bigwedge^nV^1,
 \qquad [\Delta_{\mathfrak g}(\nu)]\longmapsto v_\nu.
\]
\end{theorem}

Write $K(m|n)$ for the arc algebra of type
$(\bigwedge^mV^0,\bigwedge^nV^1)$, and let
$\operatorname{rep}K(m|n)$ denote its finite-dimensional
ungraded left modules. Write $\zeta$ for the diagram of
$\nu=0$ in \eqref{eq:glmn-weight-diagram}.
Applying Theorem~\ref{thm:level-two-tpc-uniqueness} to the
preceding theorem and
Theorem~\ref{thm:arc-tensor-product-categorifications} gives
the following result.

\begin{corollary}\label{cor:glmn-arc-equivalence}
\textup{(cf.\ \cite[Theorems 5.11--5.12]{brundanstroppelIV})}
There is a strongly equivariant equivalence
\[
 \mathcal G_{m,n}:\mathcal F(m|n)
       \xrightarrow{\sim}\operatorname{rep}K(m|n)
\]
which sends $L_{\mathfrak g}(\nu)$,
$\Delta_{\mathfrak g}(\nu)$ and $P_{\mathfrak g}(\nu)$
to the simple, standard and projective arc modules with
diagram $\nu$, respectively.
In particular, $\mathcal G_{m,n}(\Bbbk)=L(\zeta)$.
Strong equivariance gives
$\mathcal G_{m,n}(T_{r,s})\cong E^sF^rL(\zeta)$.
\end{corollary}

Deligne's category is the additive Karoubi completion
\[
 \mathcal D_\delta:=\underline{\operatorname{Rep}}(GL_\delta)
   =\operatorname{Kar}\bigl(\operatorname{Add}
                           \mathcal{OB}(\delta)\bigr);
\]
see \cite[\S3.5]{comes2012deligne}.
Thus its objects are direct summands of finite sums of words
in $\uparrow,\downarrow$, and its tensor unit is the empty word.
On this category use $F(X)=X\otimes\uparrow$ and
$E(X)=X\otimes\downarrow$, with the adjunctions given by cups
and caps. On a word $\mathbf a$, take
$x_{\mathbf a}=J_{\mathbf a\uparrow}$ and let
$t_{\mathbf a}$ interchange the last two upward strands.
The mate of $x$ is $J_{\mathbf a\downarrow}$, with the
normalization in \eqref{eq:ob-jucys-murphy}.
These maps extend to sums and summands; their generalized
eigensummands define $F_i,E_i$. We use strong
equivariance for the resulting additive categories to mean
the compatibility conditions of
Definition~\ref{def:strongly-equivariant-slI}.

\begin{proposition}\label{prop:deligne-arc-projectives}
\textup{(cf.\ \cite[Proposition IV.1.1 and \S IV.2]
{reynolds2015oriented}; \cite[Theorem 4.6.2]{comes2012deligne})}
The equivalence in Corollary~\ref{cor:ob-arc-equivalence}
induces a strongly equivariant equivalence
\[
 \mathcal H_\delta:\mathcal D_\delta
       \xrightarrow{\sim}
       K(\delta)\text{-}\operatorname{proj}_{\mathrm{fg}},
 \qquad \mathbf a\longmapsto Q_{\mathbf a}.
\]
It sends the indecomposable object with bipartition label
$\lambda$ to $P(\lambda)$. In particular,
$\mathcal H_\delta(\mathbf1)=P(\eta)$, where
$\eta=(\varnothing,\varnothing)$.
\end{proposition}

\begin{proof}
For left modules, Yoneda sends a word $\mathbf a$ to
$\mathcal{OB}(\delta)1_{\mathbf a}
=\operatorname{Hom}_{\mathcal{OB}(\delta)}(\mathbf a,-)$
contravariantly. Let $\sigma$ be horizontal reflection of
diagrams, fixing objects and reversing morphisms.
Precomposing Yoneda with $\sigma$ gives a covariant equivalence
from $\mathcal D_\delta$ to the finitely generated projective
left $\mathcal{OB}(\delta)$-modules. Compose it with the
restriction of $\mathcal G$ from
Corollary~\ref{cor:ob-arc-equivalence} to obtain
$\mathcal H_\delta$.
Reflection preserves the cellular labels, so the bipartition
labels agree by the labeling of primitive
summands in Corollary~\ref{cor:walled-brauer-arc-truncation}.

It remains to check the transformations on representables.
For $f:\mathbf a\to\mathbf b\downarrow$, write
$f^\flat:\mathbf a\uparrow\to\mathbf b$ for its adjoint.
Bending the last strand gives
\[
 (J_{\mathbf b\downarrow}f)^\flat
       =f^\flat J_{\mathbf a\uparrow}.
\]
Thus the action of $x$ on the representable agrees with
right composition by $J_{\mathbf a\uparrow}$.
Reflection fixes this element and the crossing of the last
two strands. Bending also identifies the cup and cap maps
with the fixed adjunction maps on representables.
These identities prove compatibility with $x,t$ and the
adjunctions; see \cite[Proposition IV.1.1 and \S IV.2]
{reynolds2015oriented}.
They extend to sums and summands, and the strong equivariance
of $\mathcal G$ completes the proof.
\end{proof}

Choose a strongly equivariant quasi-inverse of
$\mathcal H_\delta$ and put
$\mathcal E_{m,n}:=\mathcal G_{m,n}\circ
\operatorname{ev}_{m,n}\circ\mathcal H_\delta^{-1}$.
Thus we have a functor
\begin{equation}\label{eq:mixed-evaluation-on-projectives}
 \mathcal E_{m,n}:
 K(\delta)\text{-}\operatorname{proj}_{\mathrm{fg}}
       \longrightarrow\operatorname{rep}K(m|n).
\end{equation}
The source blocks are of type $K_\infty^\infty$, and the
target blocks are of type $K_a^\infty$, with
$0\leq a\leq\min(m,n)$.
From now on, we identify $\mathcal F(m|n)$ with
$\operatorname{rep}K(m|n)$ using
Corollary~\ref{cor:glmn-arc-equivalence} and suppress the
equivalence from the notation. In particular, $\Bbbk$ is
identified with $L(\zeta)$, $T_{r,s}$ with $E^sF^rL(\zeta)$,
and $L_{\mathfrak g}(\nu)$ with $L(\nu)$.

\begin{theorem}\label{thm:strongly-equivariant-evaluation}
\textup{(cf.\ \cite[Theorem 4.7.1 and \S8.3]{comes2012deligne};
\cite[Lemma 7.4]{BrundanStroppelV})}
The functor $\mathcal E_{m,n}$ is full and strongly equivariant.
In particular, it commutes with $E_i,F_i$ and preserves
$x,t$ and the fixed adjunction maps. Moreover,
$\mathcal E_{m,n}P(\eta)\cong L(\zeta)$ and
$\mathcal E_{m,n}Q_{\mathbf a}\cong T_{\mathbf a}$.
\end{theorem}

\begin{proof}
The equivalences on either side of $\operatorname{ev}_{m,n}$
are strongly equivariant, so it suffices to check evaluation.
Its monoidal structure identifies evaluation after $F,E$
with tensoring by $V,W$, respectively, and carries the
adjunction maps to the tensor adjunction maps.
It also carries $t$ to the super symmetry.

For $\mathbf a=a_1\cdots a_q$, the coproduct in $U(\mathfrak g)$
gives $\Omega_{T_{\mathbf a},V}=\sum_{j=1}^q\Omega_{j,q+1}$.
If $a_j=\uparrow$, the summand is the super permutation of
the two indicated $V$-factors. If $a_j=\downarrow$, it is
minus the contraction diagram on the indicated $W,V$-factors.
This is the Casimir calculation of
\cite[Lemma 7.4]{BrundanStroppelV}. Hence
\begin{equation}\label{eq:evaluation-hecke-generators}
 \operatorname{ev}_{m,n}(J_{\mathbf a\uparrow})
   =\sum_{a_j=\uparrow}\operatorname{ev}_{m,n}(s_{j,q+1})
    -\sum_{a_j=\downarrow}\operatorname{ev}_{m,n}(e_{j,q+1})
   =\Omega_{T_{\mathbf a},V}.
\end{equation}
Since evaluation preserves the adjunction, it also sends
$J_{\mathbf a\downarrow}$ to
$-\delta-\Omega_{T_{\mathbf a},W}$.
Thus it intertwines the $H_d$-actions on $F^d$ for every $d$.
All these identities extend to sums and summands.
The generalized eigenspace projections are polynomials in
the corresponding endomorphisms, so they are preserved as well.
This proves strong equivariance.

For fullness, take words with $r,s$ and $p,q$ upward and
downward arrows. If $r-s\neq p-q$, there are no homomorphisms
between their evaluations: the central identity matrix of
$\mathfrak g$ acts by the distinct scalars $r-s$ and $p-q$.
If $r-s=p-q$, put $d=r+q=p+s$.
Bending strands identifies the map on these Hom spaces with
the pure tensor action
$\Bbbk S_d\to\operatorname{End}_{\mathfrak g}(V^{\otimes d})$.
This map is surjective by super Schur--Weyl duality
\cite{sergeev1985tensor,berele1987hook}.
Evaluation is therefore full on words, and taking matrices
and idempotent corners proves fullness on $\mathcal D_\delta$.
Transport through the fixed equivalences proves the theorem.
\end{proof}

Apply $\mathcal E_{m,n}$ to endomorphisms of
$Q_{r,s}=E^sF^rP(\eta)$. Under the isomorphism
$B_{r,s}(\delta)\cong
\operatorname{End}_{K(\delta)}(Q_{r,s})^{\mathrm{op}}$
of Corollary~\ref{cor:walled-brauer-arc-truncation}, this gives
\[
 \Psi_{r,s}^{m,n}:B_{r,s}(\delta)
     \longrightarrow
     \operatorname{End}_{K(m|n)}(E^sF^rL(\zeta))^{\mathrm{op}}.
\]
By Proposition~\ref{prop:deligne-arc-projectives}, a diagram
$d$ acts as the evaluation of its horizontal reflection
$\sigma(d)$. Thus this is the usual right walled Brauer
action on $T_{r,s}$ under our identification.
Fullness gives the centralizer theorem; bending strands
reduces its kernel to the pure tensor case.
Here and below a missing part of a partition is zero.

\begin{theorem}\label{thm:mixed-tensor-centralizer}
\cite[Theorems 7.5 and 7.8]{BrundanStroppelV}
The map $\Psi_{r,s}^{m,n}$ is surjective, and
\[
 B_{r,s}(\delta)/\ker\Psi_{r,s}^{m,n}
 \cong\operatorname{End}_{K(m|n)}(E^sF^rL(\zeta))^{\mathrm{op}}
 \cong\operatorname{End}_{\mathfrak g}(T_{r,s})^{\mathrm{op}}.
\]
Let $d=r+s$, and let $z_\rho$ be the primitive central
idempotent of $\Bbbk S_d$ indexed by $\rho\vdash d$.
Bending the last $s$ strands gives a vector-space
isomorphism $\beta_{r,s}:\Bbbk S_d\to B_{r,s}(\delta)$,
under which
\begin{equation}\label{eq:mixed-tensor-kernel}
 \ker\Psi_{r,s}^{m,n}
   =\beta_{r,s}\left(
       \bigoplus_{\substack{\rho\vdash d\\\rho_{m+1}>n}}
                   \Bbbk S_dz_\rho\right).
\end{equation}
In particular, $\Psi_{r,s}^{m,n}$ is injective exactly when
$r+s<(m+1)(n+1)$.
\end{theorem}

\begin{proof}
Since $\mathcal E_{m,n}$ is full and sends $Q_{r,s}$ to
$E^sF^rL(\zeta)$, its map on endomorphism rings is surjective.
This is precisely $\Psi_{r,s}^{m,n}$ after taking opposites,
so the first isomorphism follows by dividing out its kernel.
The second is our fixed identification with mixed tensors.

For the kernel, let $\Phi_d$ be the right super permutation
action of $\Bbbk S_d$ on $V^{\otimes d}$.
The hook theorem identifies its kernel with the sum of
$\Bbbk S_dz_\rho$ over $\rho_{m+1}>n$
\cite{sergeev1985tensor,berele1987hook}.
Tensor adjunction bends an endomorphism of $V^{\otimes d}$
into one of $T_{r,s}$: move each of the last $s$ inputs to
the output and each of the last $s$ outputs to the input.
This is an invertible linear operation, and on permutation
diagrams it is $\beta_{r,s}$. It commutes with evaluation,
so it carries $\ker\Phi_d$ onto $\ker\Psi_{r,s}^{m,n}$;
see \cite[Lemmas 7.6--7.7]{BrundanStroppelV}.
This proves \eqref{eq:mixed-tensor-kernel}; no compatibility
of bending with multiplication is needed.
The smallest partition outside the $(m,n)$-hook is
$(n+1)^{m+1}$, and such a partition exists at every larger
size. This gives the injectivity criterion.
\end{proof}

For each bipartition $\lambda$, put
$R_{m,n}(\lambda):=\mathcal E_{m,n}P(\lambda)$.

\begin{lemma}\label{lem:monoidal-evaluation-projectives}
\cite[Proposition 2.7.4 and Theorem 4.7.1]{comes2012deligne}
The nonzero images $R_{m,n}(\lambda)$ are pairwise
non-isomorphic and indecomposable. As $r,s$ vary, they give
all indecomposable summands of $E^sF^rL(\zeta)$, or
equivalently of the mixed tensor powers $T_{r,s}$.
\end{lemma}

\begin{proof}
By Theorem~\ref{thm:strongly-equivariant-evaluation},
the endomorphism ring of a nonzero
$R_{m,n}(\lambda)$ is a quotient of the local ring
$\operatorname{End}_{K(\delta)}P(\lambda)$, hence is local.
Thus $R_{m,n}(\lambda)$ is indecomposable.
If two such images are isomorphic, lift an isomorphism and
its inverse by fullness. Their composite on $P(\lambda)$
maps to the identity, so it is a unit in this local ring.
Consequently $P(\lambda)$ is a direct summand of the other
indecomposable projective, and their labels agree.
Every $P(\lambda)$ occurs in some $Q_{r,s}$ by
Corollary~\ref{cor:walled-brauer-arc-truncation}.
Conversely, decomposing $Q_{r,s}$ and applying
$\mathcal E_{m,n}$ writes $E^sF^rL(\zeta)$ as a sum of
the nonzero $R_{m,n}(\lambda)$. Krull--Schmidt proves
the assertion about all indecomposable summands.
\end{proof}

We now determine which images are nonzero and identify
their heads and socles. Recall that $\operatorname{def}(\lambda)$
is the number of cups in the cup diagram of $\lambda$.
All counts of $\circ,\times$ below refer to the
$K(\delta)$-diagram of $\lambda$; both counts are finite.

\begin{definition}\label{def:cross-bipartition}
\cite[\S8.7]{comes2012deligne}
A bipartition $\lambda=(\lambda^L,\lambda^R)$ is
\emph{$(m,n)$-cross} if
$\lambda^L_{a+1}+\lambda^R_{m-a+1}\leq n$
for some $0\leq a\leq m$.
\end{definition}

\begin{lemma}\label{lem:cross-defect-rank}
\cite[Lemma 8.16]{BrundanStroppelV}
A bipartition $\lambda$ is $(m,n)$-cross if and only if
\[
 \operatorname{def}(\lambda)+\min(\#\circ,\#\times)
       \leq\min(m,n).
\]
\end{lemma}

Construct a matching $t_\lambda$ as follows, with its source
number line at the top. Draw the cap diagram of $\lambda$
at the bottom. At the top, retain the fixed $\circ,\times$
vertices of $\eta$ and draw
$\operatorname{def}(\lambda)+\min(\#\circ,\#\times)$
nested cups across its $\wedge,\vee$ interface.
Pair the rightmost $\wedge$ with the leftmost $\vee$, then
the next nearest pair, and continue.
Join the remaining vertices in the unique order-preserving
way that is vertical outside a finite interval.

If $\lambda$ is cross, place $\zeta$ on the top number line.
Label the bottom caps anti-clockwise and carry the top labels
along the through-lines. Keeping the fixed vertices of
$\lambda$ gives a diagram in $\Xi(m|n)$.
Write $\lambda^\dagger$ for its highest weight.
This is the construction of \cite[(8.17)]{BrundanStroppelV}.

\begin{theorem}\label{thm:mixed-tensor-summands}
\cite[Theorem 8.7.6]{comes2012deligne};
\cite[(8.18), Lemma 8.18 and Theorem 8.19]{BrundanStroppelV}
The module $R_{m,n}(\lambda)$ is nonzero exactly when
$\lambda$ is $(m,n)$-cross. In that case
\begin{equation}\label{eq:mixed-summand-geometric}
 R_{m,n}(\lambda)=\mathcal E_{m,n}P(\lambda)
       \cong G^{t_\lambda}L(\zeta),
\end{equation}
where the source and target blocks contain $\zeta$ and
$\lambda^\dagger$, respectively, and the grading is forgotten.
Its head and socle are both $L(\lambda^\dagger)$.

For $\lambda\in\dot\Lambda_{r,s}$ and a primitive
$e\in B_{r,s}(\delta)$ of label $\lambda$, we have
$T_{r,s}e\cong R_{m,n}(\lambda)$, which is zero for a
non-cross label. Moreover,
\begin{equation}\label{eq:mixed-tensor-decomposition}
 T_{r,s}\cong
 \bigoplus_{\substack{\lambda\in\dot\Lambda_{r,s}\\
                      \lambda\text{ is }(m,n)\text{-cross}}}
 R_{m,n}(\lambda)^{\oplus\dim D_{r,s}(\lambda)}.
\end{equation}
As $r,s$ vary, all indecomposable mixed tensor summands are
therefore indexed by the $(m,n)$-cross bipartitions.
\end{theorem}

\begin{proof}
Choose $r,s$ such that $P(\lambda)$ occurs in
$Q_{r,s}=E^sF^rP(\eta)$, and decompose this word into
composites of the refined functors $E_i,F_i$.
By Proposition~\ref{prop:geometric-functor-projectives},
each such composite sends $P(\eta)$ to zero or to copies
of a single indecomposable projective. We can therefore
choose a composite $\mathsf G$ with
$\mathsf GP(\eta)\cong P(\lambda)^{\oplus N}$, where $N>0$.

Stack the elementary matchings of $\mathsf G$ and remove
its internal circles. Let $t$ be the remaining matching.
Since $\eta$ has no caps, the projective formula introduces
no additional circles. Thus the removed circles account for
the multiplicity $N$, and the bottom caps of $t$ are precisely
the cap diagram of $\lambda$.
Every top cup must join opposite labels of $\eta$.
All its $\wedge$ labels lie to the left of all its $\vee$
labels, so noncrossing forces these cups to be nested at
the interface, starting with the nearest pair.
Let $p=\operatorname{def}(\lambda)$ and let $q$ be their number.
Counting ends of through-lines in a sufficiently large finite
interval gives
\[
 2(q-p)=\#\circ+\#\times-|\delta|
        =2\min(\#\circ,\#\times),
\]
since $\eta$ has $|\delta|$ fixed vertices and
$\#\times-\#\circ=\delta$ for $\lambda$.
It follows that $t=t_\lambda$.

Now apply $\mathcal E_{m,n}$. Strong equivariance and
$\mathcal E_{m,n}P(\eta)\cong L(\zeta)$ give
\[
 R_{m,n}(\lambda)^{\oplus N}\cong\mathsf GL(\zeta).
\]
The elementary matchings in \eqref{eq:elementary-arc-matchings}
depend only on the fixed vertices of the blocks.
Thus the same stack computes the right-hand side, with
$\zeta$ replacing $\eta$ on its top boundary.
Relative to $\eta$, the diagram $\zeta$ changes its rightmost
$\min(m,n)$ labels $\wedge$ to $\vee$ and all its labels
$\vee$ to $\wedge$, leaving the fixed vertices unchanged.
The first $\min(m,n)$ nested cups are therefore anti-clockwise;
any further cup has two $\wedge$ ends.
If $q>\min(m,n)$, no orientation of the $q$ cups is possible
in the source block, and the composite is zero.
If $q\leq\min(m,n)$, the cups are all anti-clockwise.
The composition formula \eqref{eq:geometric-functor-composition}
and Proposition~\ref{lem:geometric-functor-simple-head} then give
\[
 \mathsf GL(\zeta)
       \cong\bigl(G^{t_\lambda}L(\zeta)\bigr)^{\oplus N}
       \neq0.
\]
Lemma~\ref{lem:cross-defect-rank} identifies this inequality
with the cross condition. Cancelling the common multiplicity
by Krull--Schmidt proves \eqref{eq:mixed-summand-geometric}.
The target label with anti-clockwise caps is
$\lambda^\dagger$, so the same proposition identifies the
head and socle.

Finally, a primitive idempotent $e$ of label $\lambda$ selects
one copy of $P(\lambda)$ in $Q_{r,s}$.
The functor $\mathcal E_{m,n}$ sends that summand to
$R_{m,n}(\lambda)=T_{r,s}e$.
The multiplicity of $P(\lambda)$ in $Q_{r,s}$ is
$\dim\operatorname{Hom}_{K(\delta)}(Q_{r,s},L(\lambda))
=\dim D_{r,s}(\lambda)$ by Yoneda and
Corollary~\ref{cor:ob-arc-equivalence}.
Applying $\mathcal E_{m,n}$ to this projective decomposition
proves \eqref{eq:mixed-tensor-decomposition}.
Every bipartition occurs in some $Q_{r,s}$, and
Lemma~\ref{lem:monoidal-evaluation-projectives} shows that
the nonzero images are distinct and exhaust the indecomposable
mixed tensor summands.
\end{proof}

Equivalently, the simple modules of
$B_{r,s}(\delta)/\ker\Psi_{r,s}^{m,n}$ are precisely the
$D_{r,s}(\lambda)$ with cross labels in $\dot\Lambda_{r,s}$;
see \cite[Theorem 8.17]{BrundanStroppelV}.
The highest weight $\lambda^\dagger$ alone does not classify
the mixed tensor summand: different cross bipartitions can
have the same $\lambda^\dagger$.

\section{Jung-Kim conjecture}
% Replacement for section5.tex in JK(8).
% Uses the same diagram conventions, affine transformations and adjunctions.

Throughout this section, $\Bbbk$ is algebraically closed of
characteristic zero, and $r,s\geq0$.
We use the walled Brauer algebra of
Definition~\ref{def:oriented-brauer-category}.

\subsection{The centre on the arc side}
\label{subsec:jk-arc-centre}
Assume $\delta\in\mathbb Z$ and use the sets
$\Lambda_{r,s}$ and $\dot\Lambda_{r,s}$ from
Corollary~\ref{cor:walled-brauer-arc-truncation}.
For each weight-diagram block $\Gamma$ meeting $\dot\Lambda_{r,s}$,
put $e_\Gamma=\sum_{\lambda\in\Gamma\cap\dot\Lambda_{r,s}}e_\lambda$.
We ignore vertices labelled $\circ$ or $\times$ when pairing cups.
In a basis diagram $\underline{\lambda}\,\alpha\,\overline{\mu}$
of $K_{r,s}(\delta)$, only the outside labels $\lambda,\mu$
must lie in $\dot\Lambda_{r,s}$; the middle label need not lie there.

For a cup $C$ of $\underline{\lambda}$, let
$x_C\in e_\lambda K(\delta)e_\lambda$ be obtained from
$e_\lambda=\underline{\lambda}\,\lambda\,\overline{\lambda}$
by making the circle through $C$ clockwise.
The circle surgery rules give
\begin{equation}\label{eq:jk-diagonal-ring}
 e_\lambda K(\delta)e_\lambda
 \cong\Bbbk[x_C\mid C\text{ is a cup of }\underline{\lambda}]/(x_C^2),
 \qquad \deg x_C=2.
\end{equation}

\begin{definition}\label{def:jk-vertex-element}
For $j\in\mathbb Z$, define $X_j\in K_{r,s}(\delta)$ by
its diagonal components
\begin{equation}\label{eq:jk-vertex-family}
 e_\lambda X_j e_\lambda=
 \begin{cases}
 x_C,&j\text{ is an endpoint of a cup }C\text{ of }\underline{\lambda},\\
 0,&\text{otherwise},
 \end{cases}
 \qquad \lambda\in\dot\Lambda_{r,s},
\end{equation}
with all off-diagonal components zero.
Before truncation, the same notation denotes the family of
diagonal components given by this rule for all bipartitions.
\end{definition}

\begin{lemma}\label{lem:jk-annihilator}
Let $\lambda\in\Lambda_{r,s}$ and let $C=(a,b)$, $a<b$,
be a cup of $\underline{\lambda}$.
Let $\mu$ be obtained from $\lambda$ by interchanging the
labels at $a,b$.
Then $\mu\in\Lambda_{r,s}$ and both partition sizes decrease
by $b-a$. If $\underline{\lambda}$ has at least two cups,
then $\mu\in\dot\Lambda_{r,s}$.
Moreover, the annihilator of $e_\lambda K(\delta)e_\mu$
in $e_\lambda K(\delta)e_\lambda$ is
\begin{equation}\label{eq:jk-cyclic-module}
 \begin{cases}
 (x_C),&C\text{ is outermost},\\
 (x_C-x_P),&P\text{ is the smallest cup strictly containing }C.
 \end{cases}
\end{equation}
\end{lemma}

\begin{proof}
In the beta-sets defining the weight diagram, the reversal
replaces $a$ by $b$ in $I_\vee(\lambda)$ and $b$ by $a$ in
$I_\wedge(\lambda)$. Hence both partition sizes decrease by
$b-a$, so $\mu\in\Lambda_{r,s}$.
If $C$ is contained in the smallest cup $P=(u,v)$,
the cup diagram of $\mu$ replaces $P,C$ by $(u,a),(b,v)$.
If $C$ is outermost, its endpoints become rays, except that
one may pair with an existing ray.
All other cups remain unchanged.
Thus $\mu$ still has a cup whenever $\lambda$ has at least
two, and cannot be the deleted empty label.

Overlay $\underline{\lambda}$ and $\overline{\mu}$.
In the outermost case, $C$ lies on a line and each other cup
lies on a separate circle.
In the other case, $C,P$ lie on the same circle and each
remaining cup lies on a separate circle.
By the surgery rules, multiplication by $x_D$ reverses the
anti-clockwise circle containing $D$, and is zero if that
component is a line or already clockwise.
Thus $x_C$ acts as zero in the first case, while $x_C$ and
$x_P$ act identically in the second.

There are no further relations.
The diagram $\underline{\lambda}\,\mu\,\overline{\mu}$
has all its circles anti-clockwise.
The cups $D\ne C$ represent these circles exactly once,
so the square-free monomials in their variables send this
diagram to all oriented basis diagrams, each once; the ray
labels are forced.
The same monomials form a basis of the diagonal algebra
modulo $(x_C)$ or $(x_C-x_P)$, respectively.
This identifies $e_\lambda K(\delta)e_\mu$ with that quotient
as a left module, proving the stated annihilator.
\end{proof}

\begin{lemma}\label{lem:jk-cup-intersection}
For $\lambda\in\dot\Lambda_{r,s}$, the cup variables satisfy
\begin{equation}\label{eq:jk-cup-intersection}
 \bigcap_{C\text{ outermost}}(x_C)
 \ \cap\!
 \bigcap_{\substack{C\text{ not outermost}\\
                   P\text{ smallest cup strictly containing }C}}(x_C-x_P)
 =\Bbbk\prod_{C\text{ cup of }\underline{\lambda}}x_C
\end{equation}
inside $e_\lambda K(\delta)e_\lambda$.
\end{lemma}

\begin{proof}
Let $f$ belong to the intersection.
It is divisible by $x_C$ for every outermost cup $C$.
Suppose it is divisible by $x_P$, and $P$ is the smallest
cup strictly containing $C$.
Using the square-free monomial basis, write
$f=x_Pg+x_Px_Ch$, where $g,h$ involve neither $x_P$ nor $x_C$.
Since $f\in(x_C-x_P)$, substituting $x_C=x_P$ gives
$x_Pg=0$ and therefore $g=0$.
Thus $f$ is also divisible by $x_C$.
Proceeding from the outermost cups inwards shows that every
monomial of $f$ contains every cup variable.
Hence $f$ is a scalar multiple of the full product.

Conversely, the full product belongs to each displayed ideal:
it is divisible by every $x_C$, and
$x_Cx_P=(x_C-x_P)x_P$ since $x_P^2=0$.
\end{proof}

\begin{theorem}\label{thm:jk-arc-centre-generation}
The $e_\Gamma$ are the block idempotents of $K_{r,s}(\delta)$.
Its centre is generated by these idempotents and the
degree-two elements $X_j$.
For $k\geq1$, its degree-$2k$ part has basis
\begin{equation}\label{eq:jk-arc-centre-basis}
 \left\{e_\Gamma
 \prod_{(a,b)\text{ cup of }\underline{\lambda}}X_b
 \ \middle|\ \lambda\in\Gamma\cap\dot\Lambda_{r,s},\quad
                  \operatorname{def}(\lambda)=k\right\}.
\end{equation}
Here $\Gamma$ ranges over the weight-diagram blocks meeting
$\dot\Lambda_{r,s}$. Its degree-zero part has basis the
$e_\Gamma$, and its odd-degree parts vanish.
\end{theorem}

\begin{proof}
First, $X_j$ acts on either side of an oriented basis diagram
by reversing the anti-clockwise circle through vertex $j$,
or by zero if that component is a line or already clockwise.
This follows from the same surgery calculation as in
Lemma~\ref{lem:jk-annihilator}.
The left and right actions therefore agree, so $X_j$ is
central. Only finitely many $X_j$ are nonzero.
The $e_\Gamma$ are central because basis diagrams never join
distinct weight-diagram blocks.

Fix $\Gamma$, and put $V=\Gamma\cap\dot\Lambda_{r,s}$ and
$A=e_\Gamma K_{r,s}(\delta)e_\Gamma$.
A central element has only diagonal components:
$z=\sum_{\lambda\in V}z_\lambda$, with
$z_\lambda=e_\lambda ze_\lambda$.
Centrality is equivalent to
\begin{equation}\label{eq:jk-central-components}
 z_\lambda u=uz_\mu
 \qquad(u\in e_\lambda A e_\mu).
\end{equation}

Let $k\geq1$.
We first show that $z\in Z(A)_{2k}$ is determined by its
components at labels with exactly $k$ cups.
Suppose those components vanish but $z\ne0$, and choose
$\lambda$ of minimal $|\lambda^L|+|\lambda^R|$ with
$z_\lambda\ne0$.
By \eqref{eq:jk-diagonal-ring}, a label with fewer than $k$
cups has no diagonal component of degree $2k$.
Thus $\operatorname{def}(\lambda)>k$.
For each cup $C$, let $\mu$ be obtained by reversing its
endpoint labels. Lemma~\ref{lem:jk-annihilator} gives
$\mu\in V$ of smaller size, so $z_\mu=0$.
Equation~\eqref{eq:jk-central-components} then places
$z_\lambda$ in every annihilator of that lemma.
By Lemma~\ref{lem:jk-cup-intersection}, it is a scalar
multiple of the product of all cup variables.
That product has degree $2\operatorname{def}(\lambda)>2k$,
a contradiction.

For each $\lambda\in V$ with $\operatorname{def}(\lambda)=k$,
let $R(\lambda)$ be the set of right endpoints of its cups,
and put $M_\lambda=e_\Gamma\prod_{b\in R(\lambda)}X_b$.
For another such label $\mu$, the definition of $X_b$ gives
\begin{equation}\label{eq:jk-basis-components}
 e_\mu M_\lambda e_\mu
 =c_{\lambda\mu}\prod_{D\text{ cup of }\underline{\mu}}x_D,
 \qquad c_{\lambda\mu}\in\{0,1\},
\end{equation}
where $c_{\lambda\mu}=1$ exactly when $R(\lambda)$ selects
one endpoint from each cup of $\underline{\mu}$.
Indeed, an unpaired vertex contributes zero, and selecting
both endpoints of one cup produces the square of its variable.
In particular, $c_{\lambda\lambda}=1$.

If $c_{\lambda\mu}=1$, each selected endpoint is at most the
right endpoint of its cup. Hence
$\sum_{b\in R(\lambda)}b\leq\sum_{b\in R(\mu)}b$,
with equality only if $R(\lambda)=R(\mu)$.
The right endpoints determine the cup diagram: process them
from left to right, pairing each with the nearest unused
vertex to its left, ignoring $\circ,\times$.
All vertices strictly inside that cup have already been paired.
The remaining ray labels are fixed by the block and the
requirement that all $\wedge$ rays precede all $\vee$ rays.
Thus equality implies $\lambda=\mu$.
Ordering these labels by increasing sum of right endpoints
makes $(c_{\lambda\mu})$ upper triangular with diagonal entries
one. It follows that the $M_\lambda$ realize a basis of all
possible components at labels with $k$ cups.
The uniqueness just proved therefore makes them a basis of
$Z(A)_{2k}$, including when there are no such labels.

There are no odd-degree central elements because the diagonal
algebras have only even degrees.
In degree zero, \eqref{eq:jk-central-components} forces the
scalar components to agree along every cup reversal, since
$\underline{\lambda}\,\mu\,\overline{\mu}\ne0$.
Repeated reversals decrease size and reach the unique cupless
label of the weight-diagram block.
They stay in $\Lambda_{r,s}$ by Lemma~\ref{lem:jk-annihilator}.
If that last label is the deleted empty bipartition, the
preceding label must be $((1),(1))$: the only cup whose
reversal gives the empty diagram is $(0,1)$.
Stopping there connects all retained labels as well.
Thus $Z(A)_0=\Bbbk e_\Gamma$.
Since the positive-degree ideal is nilpotent, $A$ has no
nontrivial central idempotents, so $e_\Gamma$ is a block
idempotent.
Combining the blocks proves the basis and generation statements.
\end{proof}

We next choose central elements and compare their left and
right actions on geometric bimodules.
For $k\geq1$, let $F_k$ be the polynomial determined by
\begin{equation}\label{eq:jk-discrete-primitive}
 F_k(z+1)-F_k(z)=kz^{k-1},\qquad F_k(0)=0.
\end{equation}
Finite difference lowers degree by one, so this polynomial
exists and is unique. It is monic of degree $k$.
We use the signs from \eqref{eq:arc-klr-maps}:
\begin{equation}\label{eq:jk-stable-sign}
 \epsilon_j(\Gamma)
 =(-1)^{j-\#\{a\leq j\mid\Gamma_a\in\{\circ,\times\}\}}.
\end{equation}
The set counted here is finite. This agrees with the earlier
definition using an even left cutoff.

\begin{definition}\label{def:jk-weighted-circle-sum}
For $k\geq1$, define
\begin{equation}\label{eq:jk-boundary-potential}
 H_k:=\sum_\Gamma\sum_{j\in\mathbb Z}
       F_k(j)\epsilon_j(\Gamma)e_\Gamma X_j
       \ \in Z(K_{r,s}(\delta)).
\end{equation}
The sum is finite because only finitely many $X_j$ are nonzero.
\end{definition}

The endpoint signs of a cup are opposite. Combining its two
contributions gives
\begin{equation}\label{eq:jk-boundary-corner}
 e_\lambda H_k e_\lambda
 =\sum_{(a,b)\text{ cup of }\underline{\lambda}}
       \epsilon_a(\Gamma)(F_k(a)-F_k(b))x_{(a,b)}
 \qquad(\lambda\in\Gamma).
\end{equation}
As for $X_j$, this formula defines compatible components before
truncation and hence left and right actions on geometric bimodules.

\begin{corollary}\label{cor:jk-degree-two-span}
The centre of $K_{r,s}(\delta)$ is generated by the $H_k$,
$k\geq1$, and the block idempotents $e_\Gamma$.
\end{corollary}

\begin{proof}
By Theorem~\ref{thm:jk-arc-centre-generation}, it suffices to
recover each nonzero $X_j$. Choose a polynomial $f$ with value
one at $j$ and zero at every other cup endpoint among the
labels in $\dot\Lambda_{r,s}$.
Since $F_k$ is monic of degree $k$, write
$f=a_0+\sum_{k\geq1}a_kF_k$ with finitely many nonzero coefficients.
The constant term contributes zero, because the two endpoint
signs of each cup are opposite. Thus the definition of $H_k$ gives
$X_j=\sum_\Gamma\epsilon_j(\Gamma)e_\Gamma\sum_{k\geq1}a_kH_k$.
\end{proof}

\begin{lemma}\label{lem:jk-vertex-action}
Decompose
\begin{equation}\label{eq:jk-residue-decomposition}
 Q_{r,s}=E^sF^rQ_\varnothing
 =\bigoplus_{\mathbf i=(i_1,\ldots,i_{r+s})\in\mathbb Z^{r+s}}
 E_{i_{r+s}}\cdots E_{i_{r+1}}
 F_{i_r}\cdots F_{i_1}Q_\varnothing.
\end{equation}
Let $m$ belong to a nonzero summand with residues $\mathbf i$
and final block $\Gamma$.
For $a\leq r$, let $\xi_a$ insert the map $\xi_{i_a}$ of
\eqref{eq:arc-klr-maps} in the factor $F_{i_a}$;
for $a>r$, insert its adjunction mate $\xi_{i_a}^*$ in
$E_{i_a}$. All other factors carry identity maps.
These maps have degree two and square zero.
For every $j\in\mathbb Z$,
\begin{equation}\label{eq:jk-vertex-action}
 X_jm=\epsilon_j(\Gamma)\left(
 \sum_{a=1}^r(\delta_{j,i_a+1}-\delta_{j,i_a})\xi_a(m)
 -\sum_{a=r+1}^{r+s}(\delta_{j,i_a+1}-\delta_{j,i_a})\xi_a(m)
 \right).
\end{equation}
Here $\delta_{j,i}$ is the Kronecker delta, and $X_j$ denotes
the compatible components of Definition~\ref{def:jk-vertex-element}.
The formula also holds after applying $e_{r,s}$.
\end{lemma}

\begin{proof}
Use the geometric bimodules of
Section~\ref{subsec:extended-Khovanov-algebras} and the circle
map \eqref{eq:arc-klr-maps}.
Let $\Delta=\Omega-\boldsymbol\alpha_i$ and let
$d=\underline{\lambda}\,\alpha\,t_i(\Omega)\,
\beta\,\overline{\mu}$ be a basis diagram for $F_i|_\Omega$,
with $\lambda\in\Delta$ and $\mu\in\Omega$.
Write $Y=y_i^\Omega$ for the unsigned circle map, and
write $X_j^L(d)=X_jd$ and $X_j^R(d)=dX_j$ for the lower
and upper boundary actions, respectively.
Multiplication by $X_j$ attaches a clockwise circle at
the corresponding boundary.
On the component meeting the nonvertical part of
$t_i(\Omega)$, the surgery rules give
\[
\begin{array}{c|cccc}
 &X_i^L&X_i^R&X_{i+1}^L&X_{i+1}^R\\ \hline
 \text{top cup}&0&Y&0&Y\\
 \text{bottom cap}&Y&0&Y&0\\
 \text{top }i\text{ to bottom }i+1&0&Y&Y&0\\
 \text{top }i+1\text{ to bottom }i&Y&0&0&Y
\end{array}
\]
Here $Y$ vanishes when this component is a line or is
already clockwise.
Using the signs of the four block configurations in
\eqref{eq:elementary-arc-matchings}, together with
$\epsilon_i(\Delta)=-\epsilon_i(\Omega)$, the signed
difference is $-\epsilon_i(\Omega)Y=-\xi_i$ at $j=i$
and $+\xi_i$ at $j=i+1$.
All other vertices are joined vertically and their
contributions cancel. Hence
\begin{equation}\label{eq:jk-vertex-boundary-difference}
 \epsilon_j(\Delta)X_jd-\epsilon_j(\Omega)dX_j
 =
 (\delta_{j,i+1}-\delta_{j,i})\xi_i(d).
\end{equation}
Taking adjunction mates interchanges the two boundary
actions. Thus an $E_i$ step gives the negative of this
formula, with $\xi_i^*$ in place of $\xi_i$.
Apply these identities successively to the geometric
bimodule factors of the summand.
The balancing relation $(dh)\otimes d'=d\otimes(hd')$
identifies the right action at one step with the left action
at the preceding step, so all intermediate actions cancel.
The initial action on $Q_\varnothing$ is zero, since $X_j$
is central and the empty label has no cups.
The remaining left action is $\epsilon_j(\Gamma)X_jm$,
which gives \eqref{eq:jk-vertex-action}.
The degree and square-zero assertions follow from
Lemma~\ref{lem:arc-klr-relations} by insertion and adjunction.
All maps are $K(\delta)$-linear, so the formula restricts to
$e_{r,s}Q_{r,s}$.
\end{proof}

\begin{lemma}\label{lem:jk-iterated-action}
For $m$ and $\xi_a$ as in Lemma~\ref{lem:jk-vertex-action},
\begin{equation}\label{eq:jk-telescoping}
 H_km=k\left(\sum_{a=1}^r i_a^{k-1}\xi_a(m)
        -\sum_{a=r+1}^{r+s}i_a^{k-1}\xi_a(m)\right)
 \qquad(k\geq1).
\end{equation}
\end{lemma}

\begin{proof}
Multiply \eqref{eq:jk-vertex-action} by
$F_k(j)\epsilon_j(\Gamma)$ and sum over $j$.
The left side is $H_km$ by \eqref{eq:jk-boundary-potential}.
Since $\epsilon_j(\Gamma)^2=1$, the coefficient of each
$\xi_a(m)$ on the right is, with the indicated sign,
$F_k(i_a+1)-F_k(i_a)=ki_a^{k-1}$.
This proves the formula.
\end{proof}

\subsection{Jucys--Murphy elements and the conjecture}
\label{subsec:jk-elements}

Number the vertices of $\uparrow^r\downarrow^s$ from $1$ to $r+s$.
Let $s_{a,b}$ transpose two vertices on the same side of the wall.
For $a\leq r<b$, let $e_{a,b}$ join $a,b$ on each boundary,
with all other strands vertical; thus $e_{a,b}^2=\delta e_{a,b}$.

\begin{definition}\label{def:jk-jucys-murphy}
The Jucys--Murphy elements are
\begin{equation}\label{eq:jk-elements}
 L_a=
 \begin{cases}
  \displaystyle\sum_{1\leq b<a}s_{b,a},&1\leq a\leq r,\\[4pt]
  \displaystyle\delta\,1-\sum_{b=1}^r e_{b,a}
      +\sum_{r<b<a}s_{b,a},&r<a\leq r+s.
 \end{cases}
\end{equation}
Empty sums are zero. We use the normalization of
\cite[Definition 2.1]{JungKim2020}, including the constant
$\delta$ on the right of the wall.
\end{definition}

The elements in \eqref{eq:jk-elements} are the elements
$J_{\mathbf a}$ of \eqref{eq:ob-jucys-murphy}, with a sign
change for downward strands. More precisely, extend each
prefix diagram by vertical strands on the right. Then
\begin{equation}\label{eq:jk-J-L-comparison}
 L_a=
 \begin{cases}
 J_{\uparrow^a}\otimes1_{\uparrow^{r-a}\downarrow^s},
     &1\leq a\leq r,\\[3pt]
 -J_{\uparrow^r\downarrow^{a-r}}\otimes1_{\downarrow^{r+s-a}},
     &r<a\leq r+s.
 \end{cases}
\end{equation}
Here $1_{\varnothing}$ means that no strand is appended.
Thus $J$ on an upward prefix gives $L_a$, whereas $J$ on a
prefix ending in $\downarrow$ gives $-L_a$.
There is no further scalar shift: the term $-\delta\,1$ in
$J$ becomes the term $+\delta\,1$ in $L_a$ on the right of
the wall. For example, in $B_{1,1}(\delta)$,
$L_2=\delta\,1-e_{1,2}=-J_{\uparrow\downarrow}$.

\begin{definition}\label{def:jk-supersymmetric}
A polynomial in $x_1,\ldots,x_r,y_1,\ldots,y_s$ is
\emph{supersymmetric} if it is symmetric in each group separately
and its specialization at $x_r=t$, $y_1=-t$ is independent of $t$.
Omit the latter condition when $rs=0$.
Write $\mathscr S_{r,s}$ for this algebra.
\end{definition}

In characteristic zero it is generated by
\begin{equation}\label{eq:jk-power-sums}
 p_k(x;y)=\sum_{a=1}^r x_a^k
                  +(-1)^{k+1}\sum_{b=1}^s y_b^k
 \qquad(k\geq1);
\end{equation}
see \cite{Stembridge1985}.
Write $p_k(L)$ for evaluation at
$(L_1,\ldots,L_r;L_{r+1},\ldots,L_{r+s})$.

\begin{proposition}\label{prop:jk-centrality}
The $L_a$ commute, and evaluation maps
$\mathscr S_{r,s}$ into $Z(B_{r,s}(\delta))$.
\end{proposition}

\begin{proof}
Each $L_a$ commutes with the algebra on the preceding strands,
so the $L_a$ commute. Diagram concatenation gives, for
$s_i=s_{i,i+1}$ and $i\ne r$,
\[
 s_iL_{i+1}-L_i s_i=1=L_{i+1}s_i-s_iL_i,
 \qquad [s_i,L_a]=0\quad(a\ne i,i+1).
\]
Thus $s_i$ commutes with $L_i+L_{i+1}$ and $L_iL_{i+1}$,
and hence with $p_k(L)$.
When $rs>0$, writing $e=e_{r,r+1}$ gives
\[
 e(L_r+L_{r+1})=(L_r+L_{r+1})e=0,
 \qquad [e,L_a]=0\quad(a\ne r,r+1).
\]
The polynomial $u^k+(-1)^{k+1}v^k$ is divisible by $u+v$,
so $e$ also commutes with $p_k(L)$.
These diagrams generate $B_{r,s}(\delta)$, proving the claim;
see \cite[Propositions 2.4, 2.6--2.7]{JungKim2020}.
\end{proof}

For $\delta\notin\mathbb Z$, Theorem~\ref{thm:jk-centre}
follows from \cite[Proposition 1.1 and Theorem 3.5]{JungKim2020}.
The result over $\mathbb C$ extends to $\Bbbk$ by base change
from $\mathbb Q(\delta)$.
Until the final theorem, assume $\delta\in\mathbb Z$.

\begin{lemma}\label{lem:jk-morita-centres}
\cite[Lemma 18.41 and Corollary 18.42]{lam1999lectures}
Let $A,B$ be unital algebras and let ${}_AP_B$ induce a Morita
equivalence $P\otimes_B-:B\text{-}\operatorname{Mod}
\xrightarrow{\sim}A\text{-}\operatorname{Mod}$.
There is a unique algebra isomorphism
$\theta:Z(B)\xrightarrow{\sim}Z(A)$ such that
$\theta(z)p=pz$ for every $z\in Z(B)$ and $p\in P$.
\end{lemma}

\begin{proof}
Morita theory identifies $A$ with $\operatorname{End}_B(P)$
by left multiplication and $B^{\mathrm{op}}$ with
$\operatorname{End}_A(P)$ by right multiplication.
In particular, $P$ is faithful on both sides.
For $z\in Z(B)$, the map $p\mapsto pz$ is right $B$-linear,
so there is a unique $a\in A$ with $ap=pz$ for all $p$.
For $a'\in A$ we have
$(aa')p=(a'p)z=a'(pz)=(a'a)p$.
Faithfulness gives $aa'=a'a$, so $a\in Z(A)$; set $\theta(z)=a$.

Conversely, for $a\in Z(A)$ the map $p\mapsto ap$ is
left $A$-linear, so it is right multiplication by a unique
$z\in B$. For $b\in B$, the equality
$(pb)z=a(pb)=(ap)b=(pz)b$ and right faithfulness give
$bz=zb$. Hence $z\in Z(B)$, proving that $\theta$ is bijective.
The identity $\theta(z)p=pz$ and faithfulness show that
$\theta$ preserves addition, multiplication and the identity,
and also prove its uniqueness.
\end{proof}

\begin{definition}\label{def:jk-centre-isomorphism}
Fix the identification
$\mathcal G(\mathcal{OB}(\delta)1_{\uparrow^r\downarrow^s})
\cong Q_{r,s}$ from Corollary~\ref{cor:ob-arc-equivalence}.
Give $Q_{r,s}$ its usual left $K(\delta)$-action and the
right $B_{r,s}(\delta)$-action obtained by applying $\mathcal G$
to right multiplication on the representable module.
These actions commute, so the right action restricts to
$e_{r,s}Q_{r,s}$.
By Corollary~\ref{cor:walled-brauer-arc-truncation}, this is
a Morita bimodule. Define
\begin{equation}\label{eq:jk-centre-isomorphism}
 \zeta:Z(B_{r,s}(\delta))\xrightarrow{\sim}Z(K_{r,s}(\delta))
\end{equation}
to be the isomorphism of Lemma~\ref{lem:jk-morita-centres}
for this bimodule, characterized by
$\zeta(z)m=mz$ for $z\in Z(B_{r,s}(\delta))$ and
$m\in e_{r,s}Q_{r,s}$.
\end{definition}

\begin{lemma}\label{lem:jk-normalized-comparison}
Let $m$ belong to the summand of $Q_{r,s}$ with residues
$\mathbf i=(i_1,\ldots,i_{r+s})$ in
\eqref{eq:jk-residue-decomposition}, and use the maps
$\xi_a$ of Lemma~\ref{lem:jk-vertex-action}. Then
\begin{equation}\label{eq:jk-normalized-comparison}
 mL_a=
 \begin{cases}
 i_a m+\xi_a(m),&1\leq a\leq r,\\
 -i_a m-\xi_a(m),&r<a\leq r+s.
 \end{cases}
\end{equation}
The same formula holds after applying $e_{r,s}$.
\end{lemma}

\begin{proof}
On the oriented Brauer representables, the bending
isomorphisms identify the affine transformation on $F$ with
right multiplication by $J_{\mathbf a\uparrow}$, and its
mate on $E$ with right multiplication by
$J_{\mathbf a\downarrow}$;
see \cite[Proposition IV.1.1 and \S IV.2]{reynolds2015oriented}
and the normalization in \eqref{eq:ob-jucys-murphy}.
Thus \eqref{eq:jk-J-L-comparison} makes right multiplication
by $L_a$ the insertion of $x$ at an upward step and the
negative of its mate at a downward step.
Strong equivariance of $\mathcal G$ preserves these maps.
The Brundan--Kleshchev comparison \eqref{eq:hecke-klr-dot}
gives $x|_{F_i}=i\operatorname{id}+\xi_i$ and, by taking
mates, $x^*|_{E_i}=i\operatorname{id}+\xi_i^*$.
Inserting these at position $a$ proves the formula.
All maps commute with $e_{r,s}$, so it also holds after truncation.
\end{proof}

\begin{proposition}\label{prop:jk-explicit-formula}
For $\lambda$ in a block $\Gamma$, the scalar
\begin{equation}\label{eq:jk-content-character}
 \chi_{\lambda,k}=
 \sum_{\square\in\lambda^L}\operatorname{ct}(\square)^k
 -\sum_{\square\in\lambda^R}(-\delta-\operatorname{ct}(\square))^k
 \qquad(k\geq1)
\end{equation}
depends only on $\Gamma$; write it as $\chi_{\Gamma,k}$.
In $Z(K_{r,s}(\delta))$, one has
\begin{equation}\label{eq:jk-explicit-formula}
 \zeta(p_k(L))=\sum_\Gamma\chi_{\Gamma,k}e_\Gamma+H_k.
\end{equation}
\end{proposition}

\begin{proof}
Put $D_\lambda(q)=
\sum_{\square\in\lambda^L}q^{\operatorname{ct}(\square)}
-\sum_{\square\in\lambda^R}q^{-\delta-\operatorname{ct}(\square)}$.
Summing row by row gives
\begin{equation}\label{eq:jk-content-occupancy}
 (q-1)D_\lambda(q)
 =\sum_{a\geq1}
 \bigl(q^{1-a+\lambda^L_a}-q^{1-a}
       +q^{a-\delta-\lambda^R_a}-q^{a-\delta}\bigr).
\end{equation}
The coefficient of $q^j$ is the number of occurrences of $j$
in $I_\wedge(\lambda),I_\vee(\lambda)$ minus that number for
the empty bipartition. The two beta sets contain $j$ zero
times at a $\circ$, once at a $\vee$ or $\wedge$, and twice
at a $\times$. Hence the coefficients are fixed by the block,
so $D_\lambda$ and
its moments $\chi_{\lambda,k}$ depend only on $\Gamma$.

An $F_i$ step moves one beta-set entry from $i$ to $i+1$,
so \eqref{eq:jk-content-occupancy} increases $D_\lambda(q)$
by $q^i$. An $E_i$ step decreases it by $q^i$.
Starting from the empty bipartition thus gives, for the
final block of a nonzero summand in
\eqref{eq:jk-residue-decomposition},
\[
 D_\lambda(q)=\sum_{a=1}^r q^{i_a}
                 -\sum_{a=r+1}^{r+s}q^{i_a},
 \qquad
 \chi_{\Gamma,k}=\sum_{a=1}^r i_a^k
                 -\sum_{a=r+1}^{r+s}i_a^k.
\]
Now let $m$ belong to this summand after applying $e_{r,s}$.
Since $\xi_a^2=0$,
\[
\begin{aligned}
 mp_k(L)
 &\overset{\ref{lem:jk-normalized-comparison}}{=}
 \left(\sum_{a=1}^r i_a^k-\sum_{a=r+1}^{r+s}i_a^k\right)m
 +k\left(\sum_{a=1}^r i_a^{k-1}\xi_a(m)
         -\sum_{a=r+1}^{r+s}i_a^{k-1}\xi_a(m)\right)\\
 &\overset{\ref{lem:jk-iterated-action}}{=}
 \chi_{\Gamma,k}m+H_km.
\end{aligned}
\]
These summands span $e_{r,s}Q_{r,s}$.
Its faithfulness and the characterization of $\zeta$ give
\eqref{eq:jk-explicit-formula}.
\end{proof}

\begin{lemma}\label{lem:jk-block-idempotents}
For every block $\Gamma$ of $K_{r,s}(\delta)$,
$\zeta^{-1}(e_\Gamma)\in\Bbbk[p_k(L)\mid k\geq1]$.
Equivalently, every block idempotent of $B_{r,s}(\delta)$
is a polynomial in the $p_k(L)$.
\end{lemma}

\begin{proof}
We first show that distinct blocks have different sequences
$(\chi_{\Gamma,k})_{k\geq1}$.
Suppose $\lambda,\mu\in\dot\Lambda_{r,s}$ have equal positive
moments. For the Laurent polynomials in the proof of
Proposition~\ref{prop:jk-explicit-formula}, their zeroth
moments also agree: $D_\lambda(1)=r-s=D_\mu(1)$.
Write $D_\lambda-D_\mu=\sum_{j\in S}u_jq^j$, where $S$
is finite. Equality of moments gives
$\sum_{j\in S}u_jj^k=0$ for $0\leq k<|S|$.
The Vandermonde matrix on the distinct integers in $S$
is invertible, so all $u_j$ vanish.
Equation~\eqref{eq:jk-content-occupancy} then gives the same
$\circ$ and $\times$ positions for $\lambda$ and $\mu$,
so they lie in the same block.

There are finitely many blocks. Choose finitely many moments
that distinguish every pair, and then scalars $a_k$ such that
the numbers $c_\Gamma=\sum_k a_k\chi_{\Gamma,k}$ are pairwise
distinct. Such a choice exists because $\Bbbk$ is infinite:
each unwanted equality excludes only a proper hyperplane.
Put $h=\sum_k a_kp_k(L)$.
Proposition~\ref{prop:jk-explicit-formula} gives
$\zeta(h)e_\Delta=c_\Delta e_\Delta+n_\Delta$, where
$n_\Delta=\sum_k a_kH_ke_\Delta$ has positive degree.
Since $K_{r,s}(\delta)$ is finite-dimensional and nonnegatively
graded, choose $N\geq1$ with $n_\Delta^N=0$ for every block.

Fix $\Gamma$. Choose a polynomial $g_\Gamma(T)\in\Bbbk[T]$
such that, for every block $\Delta$,
\[
 g_\Gamma(T)\equiv
 \begin{cases}
 1,&\Delta=\Gamma,\\
 0,&\Delta\ne\Gamma
 \end{cases}
 \pmod{(T-c_\Delta)^N}.
\]
Here the first condition means that $g_\Gamma(T)-1$ is
divisible by $(T-c_\Gamma)^N$; the other conditions mean
that $g_\Gamma(T)$ is divisible by $(T-c_\Delta)^N$ for
every $\Delta\ne\Gamma$.
The Chinese remainder theorem supplies such a polynomial
because the $c_\Delta$ are distinct.
These powers are needed to remove the nilpotent terms:
substituting $\zeta(h)e_\Delta=c_\Delta e_\Delta+n_\Delta$
and using $n_\Delta^N=0$ gives
\[
 g_\Gamma(\zeta(h))e_\Delta=
 \begin{cases}
 e_\Gamma,&\Delta=\Gamma,\\
 0,&\Delta\ne\Gamma.
 \end{cases}
\]
Summing over $\Delta$ yields
$\zeta(g_\Gamma(h))=e_\Gamma$.
Thus $g_\Gamma(h)=\zeta^{-1}(e_\Gamma)$ is the required block
idempotent, and it belongs to $\Bbbk[p_k(L)\mid k\geq1]$.
\end{proof}

\begin{theorem}\label{thm:jk-centre}
Let $\Bbbk$ be an algebraically closed field of characteristic
zero. For every $r,s\geq0$ and $\delta\in\Bbbk$, the centre
of the walled Brauer algebra is generated by supersymmetric
polynomials in its Jucys--Murphy elements. More precisely,
\begin{equation}\label{eq:jk-conjecture}
 Z(B_{r,s}(\delta))
 =\{f(L_1,\ldots,L_r;L_{r+1},\ldots,L_{r+s})
                         \mid f\in\mathscr S_{r,s}\}
 =\Bbbk[p_k(L)\mid k\geq1].
\end{equation}
\end{theorem}

\begin{proof}
Use Corollary~\ref{cor:jk-degree-two-span}, Proposition~\ref{prop:jk-explicit-formula} and Lemma~\ref{lem:jk-block-idempotents} via $\zeta$ for $\delta\in\mathbb Z$, and \cite[Theorem 3.5]{JungKim2020} otherwise.
\end{proof}

\begin{corollary}\label{cor:jk-centre-dimension}
Let $p(d)$ denote the number of partitions of $d$, with
$p(0)=1$. For every $r,s\geq0$ and $\delta\in\Bbbk$,
\begin{equation}\label{eq:jk-centre-dimension}
 \dim Z(B_{r,s}(\delta))
 =
 |\Lambda_{r,s}|
 =
 \sum_{a=0}^{\min(r,s)}p(r-a)p(s-a).
\end{equation}
For $\delta\in\mathbb Z$, give the centre the grading
transported through \eqref{eq:jk-centre-isomorphism},
and let $b_{r,s}(\delta)$ be the number of
weight-diagram blocks meeting $\dot\Lambda_{r,s}$.
Then
\begin{equation}\label{eq:jk-centre-hilbert-series}
 \operatorname{Hilb}\bigl(Z(B_{r,s}(\delta));q\bigr)
 =
 b_{r,s}(\delta)
 +
 \sum_{\substack{\lambda\in\dot\Lambda_{r,s}\\
                  \operatorname{def}(\lambda)>0}}
 q^{2\operatorname{def}(\lambda)}.
\end{equation}
\end{corollary}

\begin{proof}
The graded formula follows from
Theorem~\ref{thm:jk-arc-centre-generation} via $\zeta$.
Successive cup reversals stay in $\Lambda_{r,s}$ and
reach the unique cupless label of each block.
Thus, when $\dot\Lambda_{r,s}=\Lambda_{r,s}$,
evaluation at $q=1$ gives $|\Lambda_{r,s}|$.

When $\delta=0$ and $r=s>0$, the deleted empty
bipartition is cupless, but its block still meets
$\dot\Lambda_{r,s}$ at $((1),(1))$.
Its contribution is therefore retained in
$b_{r,s}(\delta)$, giving the same total dimension.
For $\delta\notin\mathbb Z$, the algebra is semisimple
with simples indexed by $\Lambda_{r,s}$.
Finally, the partition sum counts this indexing set.
\end{proof}

\appendix

\section{Semi-infinite highest weight categories}
\label{sec:semi-infinite-highest-weight-categories}
We use the conventions of \cite{brundanstroppelsemiinfinite},
with right modules over an algebraically closed field $\Bbbk$
of characteristic zero. Simple objects $L(a)$ are indexed by
an interval-finite poset $(\Xi,\preceq)$.
The poset is lower (respectively upper) finite if every
principal ideal $\Xi_{\preceq a}$ (respectively coideal
$\Xi_{\succeq a}$) is finite.

\begin{definition}\label{def:essentially-finite-highest-weight}
An \emph{essentially finite abelian category} is essentially
small, Hom-finite and of finite length, with enough projectives
and injectives. It is \emph{highest weight} if the largest
quotient $\Delta(a)$ of the projective cover $P(a)$ supported
on $\Xi_{\preceq a}$ satisfies $[\Delta(a):L(a)]=1$ and
\begin{equation}\label{eq:highest-weight-projective-flag}
 0\longrightarrow K(a)\longrightarrow P(a)
   \longrightarrow\Delta(a)\longrightarrow0,
\end{equation}
where $K(a)$ has a finite flag by $\Delta(b)$ with $b\succ a$.
The category is \emph{finite highest weight} if $\Xi$ is finite.
\end{definition}

\begin{definition}\label{def:lower-finite-highest-weight}
A \emph{lower finite highest weight category} is an essentially
small Hom-finite abelian category of finite-length objects,
with $\Xi$ lower finite, such that every finite ideal
$\Omega\subseteq\Xi$ defines a finite highest weight Serre
subcategory $\mathcal C_\Omega$ with labels $\Omega$.
Standards and costandards are those in these subcategories.
\end{definition}

\begin{definition}\label{def:upper-finite-highest-weight}
An \emph{upper finite highest weight category} is
$\mathcal C=\operatorname{mod}_{\mathrm{lfd}}\!\text{-}A$,
where $A=\bigoplus_{a,b\in\Xi}e_aAe_b$ is pointed locally
unital with finite-dimensional corners and $\Xi$ upper finite.
Here $\dim Me_a<\infty$, $P(a)=e_aA$, and
$\dim L(a)e_b=\delta_{a,b}$.
The largest quotient $\Delta(a)$ of $P(a)$ supported on
$\Xi_{\preceq a}$ satisfies $\dim\Delta(a)e_a=1$ and
\eqref{eq:highest-weight-projective-flag}.
\end{definition}

See \cite[Corollary 2.20 and Definitions 3.7, 3.34, 3.50]
{brundanstroppelsemiinfinite}.
In the essentially finite and upper finite cases, $\nabla(a)$
is the largest subobject of the injective hull of $L(a)$
supported on $\Xi_{\preceq a}$.
Write $\mathcal C^\Delta,\mathcal C^\nabla$ for the finite-flag
subcategories, with their inherited exact structures.
In the upper finite case the sections themselves may have
infinite length. For an ideal $J$, let $\mathcal C_J$ be the
subcategory supported on $J$; support means composition-factor
labels in the finite-length cases and $\{a\mid Me_a\ne0\}$ otherwise.

\begin{lemma}\label{lem:finite-highest-weight-subquotients}
\cite[Theorems 3.17--3.18 and 3.41--3.42]{brundanstroppelsemiinfinite}
For ideals $J'\subseteq J$ with $S=J\setminus J'$ finite,
$\mathcal C[S]=\mathcal C_J/\mathcal C_{J'}$ is finite highest
weight, with standards the images of $\Delta(a)$ for $a\in S$.
These truncations are transitive and retain $\mathcal C_\Omega$
as a Serre subcategory whenever $\Omega\subseteq S$ is an ideal
of $\Xi$.
In the essentially finite and upper finite cases, projectives
whose standard labels lie in $S$ descend to the corresponding
projective covers, and truncation preserves all Hom spaces
between them.
\end{lemma}

Tiltings have finite standard and costandard flags in the lower
finite case, and ascending standard and descending costandard
flags in the upper finite case, in the sense of
\cite[Definition 3.35 and (4.10)]{brundanstroppelsemiinfinite}.

\begin{proposition}\label{prop:semi-infinite-indecomposable-tiltings}
\cite[Theorems 4.2 and 4.18]{brundanstroppelsemiinfinite}
In either setting, there is a unique indecomposable tilting
$T(a)$ supported on $\Xi_{\preceq a}$ with $[T(a):L(a)]=1$.
Its standard flag starts with $\Delta(a)$ and its costandard
flag has top section $\nabla(a)$.
In the upper finite case, $[M:L(a)]$ means $\dim Me_a$.
\end{proposition}

For an essentially finite category, use the tiltings in its
locally finite-dimensional envelope
\cite[\S4.3]{brundanstroppelsemiinfinite}.

\begin{definition}\label{def:tilting-bounded-highest-weight}
The category is \emph{tilting-bounded} if
$(\dim\operatorname{Hom}(T(a),T(b)))_{a,b\in\Xi}$
has finite support in every row and column.
\end{definition}

In the essentially finite case this ensures that $T(a)$ lies
in $\mathcal C$ and has finite standard and costandard flags
\cite[Definition 4.20 and Lemma 4.21]{brundanstroppelsemiinfinite}.

\begin{definition}\label{def:semi-infinite-ringel-dual}
\cite[Definitions 4.24, 4.26 and 4.33]{brundanstroppelsemiinfinite}
Put
\begin{equation}\label{eq:finite-entry-tilting-algebra}
 B=\bigoplus_{a,b}\operatorname{Hom}(T(a),T(b)),
\end{equation}
with matrix composition. The \emph{Ringel dual} is
$\mathcal C^\vee=\operatorname{mod}_{\mathrm{lfd}}\!\text{-}B$
when $\mathcal C$ is lower finite, and
$\mathcal C^\vee=\operatorname{mod}_{\mathrm{fd}}\!\text{-}B$
when it is tilting-bounded essentially finite.
When $\mathcal C$ is upper finite, put $T=\bigoplus_aT(a)$
and $\widehat B=\operatorname{End}(T)$, with the topology of
restriction to finitely many components $Te_a$.
Then $\mathcal C^\vee$ consists of finite-dimensional continuous
right $\widehat B$-modules, given the discrete topology.
\end{definition}

\begin{theorem}\label{thm:semi-infinite-ringel-duality}
\cite[Theorems 4.25, 4.27, 4.34 and Corollaries 4.29--4.30, 4.35]
{brundanstroppelsemiinfinite}
The Ringel dual has weight poset $\Xi^{\mathrm{op}}$ and is,
respectively, upper finite, lower finite, or tilting-bounded
essentially finite. There are exact equivalences
\begin{equation}\label{eq:ringel-finite-flag-equivalences}
 \mathcal R:\mathcal C^\nabla\xrightarrow{\sim}(\mathcal C^\vee)^\Delta,
 \qquad
 \mathcal S:\mathcal C^\Delta\xrightarrow{\sim}(\mathcal C^\vee)^\nabla,
\end{equation}
sending $\nabla(a)$ to $\Delta^\vee(a)$ and $\Delta(a)$ to
$\nabla^\vee(a)$, respectively.
The dual is independent of tilting representatives, up to
label-preserving highest weight equivalence, and
$(\mathcal C^\vee)^\vee\simeq\mathcal C$ with these ambient conventions.
\end{theorem}

\section{Defect one examples}\label{sec:defect-one-examples}

We use Alperin diagrams to describe the structure of finite-length
modules; see \cite{Alperin1980}.
Vertices represent occurrences of simple composition factors, and
a downward edge from $S$ to $T$ records a nonsplit extension of
$S$ by $T$ occurring as a subquotient.
In the interpretation that describes the entire submodule lattice,
submodules correspond to arrow-closed subsets of vertices, with
unions and intersections corresponding to sums and intersections
of submodules. Here arrow-closed means closed under following
downward arrows.

Such a diagram exists precisely when the submodule lattice is
finite and distributive: one takes the Hasse diagram of its
nonzero join-irreducible elements, labelled by their simple heads;
see \cite{Musson2017}.
This applies, in particular, to multiplicity-free modules and to
uniserial modules, including those with repeated composition factors.
It does not apply to every finite-length module: for example,
$S\oplus S$, with $S$ simple, has a nondistributive submodule lattice.
A diagram displaying only the Loewy layers should therefore be
distinguished from one describing all submodules.

\subsection{$K^m_1$}\label{subsec:defect-one-finite}
Assume $m\geq1$. Number the weights of $\Lambda^m_1$ by
$0,\ldots,m$, where $r$ denotes the weight with its unique
$\vee$ at position $r+1$. The Bruhat order is then
$0<1<\cdots<m$. We use the left-module and grading
conventions from the main text.

\paragraph{Quiver and modules.}
Let $Q^m_1$ be the doubled line
\[
 0
 \mathrel{\mathop{\rightleftarrows}^{x_0}_{y_0}}
 1
 \mathrel{\mathop{\rightleftarrows}^{x_1}_{y_1}}
 2
 \mathrel{\mathop{\rightleftarrows}^{x_2}_{y_2}}
 \cdots
 \mathrel{\mathop{\rightleftarrows}^{x_{m-1}}_{y_{m-1}}}
 m,
\]
with arrows
$x_r:r\to r+1$ and $y_r:r+1\to r$ for $0\leq r<m$.
The algebra $K^m_1$ is the quotient of $\Bbbk Q^m_1$ by
\begin{equation}\label{eq:km1-relations}
\begin{aligned}
 x_{r+1}x_r=y_ry_{r+1}&=0 &&(0\leq r\leq m-2),\\
 y_rx_r&=x_{r-1}y_{r-1} &&(1\leq r\leq m-1),\\
 x_{m-1}y_{m-1}&=0.
\end{aligned}
\end{equation}
Paths are composed from right to left, and every arrow has
degree one. These relations follow directly from the surgery
rules. A basis consists of the $m+1$ vertex idempotents,
the $2m$ arrows, and the $m$ cycles $y_rx_r$ for $0\leq r<m$;
all paths of length three vanish.
With the order above, $K^m_1$ is positively graded and
quasi-hereditary.

Write $L_r,P_r,I_r,\Delta_r,\nabla_r,T_r$ for the corresponding
simple, projective, injective, standard, costandard and tilting
modules. Thus $L_r$ lies in degree zero and $P_r=K^m_1e_r$.
The anti-involution $e_r^*=e_r$, $x_r^*=y_r$, $y_r^*=x_r$
induces the duality $\circledast$ from the main text, with
$I_r=P_r^{\circledast}$ and
$\nabla_r=\Delta_r^{\circledast}$.
We normalize $T_r$ by a degree-zero inclusion
$\Delta_r\hookrightarrow T_r$.

The following diagrams show composition factors and extensions;
their gradings are omitted. We have
\[
 \Delta_0=\nabla_0=T_0=L_0,\qquad
 \Delta_r=
 \begin{tikzpicture}[baseline=(current bounding box.center),scale=0.65]
 \node (a) at (0,1) {$L_r$};
 \node (b) at (0,0) {$L_{r-1}$};
 \draw (a)--(b);
 \end{tikzpicture},\qquad
 \nabla_r=
 \begin{tikzpicture}[baseline=(current bounding box.center),scale=0.65]
 \node (a) at (0,1) {$L_{r-1}$};
 \node (b) at (0,0) {$L_r$};
 \draw (a)--(b);
 \end{tikzpicture}
 \quad(1\leq r\leq m).
\]
The projective modules are
\[
 P_0=
 \begin{tikzpicture}[baseline=(current bounding box.center),scale=0.65]
 \node (a) at (0,2) {$L_0$};
 \node (b) at (0,1) {$L_1$};
 \node (c) at (0,0) {$L_0$};
 \draw (a)--(b)--(c);
 \end{tikzpicture},\qquad
 P_r=
 \begin{tikzpicture}[baseline=(current bounding box.center),scale=0.65]
 \node (a) at (0,2) {$L_r$};
 \node (b) at (-1,1) {$L_{r-1}$};
 \node (c) at (1,1) {$L_{r+1}$};
 \node (d) at (0,0) {$L_r$};
 \draw (a)--(b)--(d);
 \draw (a)--(c)--(d);
 \end{tikzpicture}
 \quad(1\leq r<m),\qquad
 P_m=\Delta_m.
\]
The layers of $P_r$ have degrees $0,1,2$ for $r<m$.
The precise graded identifications are
\begin{equation}\label{eq:km1-injectives-tiltings}
\begin{aligned}
 I_r&\cong P_r\langle-2\rangle &&(0\leq r<m),\\
 T_r&\cong P_{r-1}\langle-1\rangle &&(1\leq r\leq m),
\end{aligned}
\end{equation}
and $I_m=\nabla_m$.
In particular, after forgetting the grading, $P_r\cong I_r$
for $r<m$, and every $T_r$ is self-dual.

\paragraph{Resolutions and Ringel duality.}
The projective and tilting filtrations give short exact sequences
\begin{equation}\label{eq:km1-standard-sequences}
\begin{aligned}
 0&\longrightarrow\Delta_{r+1}\langle1\rangle
 \longrightarrow P_r\longrightarrow\Delta_r\longrightarrow0
 &&(0\leq r<m),\\
 0&\longrightarrow\Delta_r\longrightarrow T_r
 \longrightarrow\Delta_{r-1}\langle-1\rangle\longrightarrow0
 &&(1\leq r\leq m).
\end{aligned}
\end{equation}
Splicing the first family gives the minimal projective resolution
\begin{equation}\label{eq:km1-standard-resolution}
 0\longrightarrow P_m\langle m-r\rangle
 \longrightarrow\cdots\longrightarrow P_r
 \longrightarrow\Delta_r\longrightarrow0.
\end{equation}
Its term in homological degree $j$ is $P_{r+j}\langle j\rangle$
for $0\leq j\leq m-r$; for $r=m$ it is just $P_m=\Delta_m$.
The resolution is linear. Applying $\circledast$ gives a linear
injective coresolution of every costandard, so $K^m_1$ is
standard Koszul.

Splicing the second family gives the linear tilting coresolution
\begin{equation}\label{eq:km1-tilting-coresolution}
 0\longrightarrow\Delta_r\longrightarrow T_r
 \longrightarrow T_{r-1}\langle-1\rangle
 \longrightarrow\cdots\longrightarrow T_0\langle-r\rangle
 \longrightarrow0
 \qquad(1\leq r\leq m),
\end{equation}
while $\Delta_0=T_0$.
The endomorphisms between these tiltings also show that $K^m_1$
is Ringel self-dual. More precisely, the graded algebra
$\operatorname{End}_{K^m_1}(\bigoplus_{r=0}^mT_r)^{\operatorname{op}}$,
using all homogeneous maps, has the same doubled-line quiver
and interior relations as \eqref{eq:km1-relations}, but the
zero boundary cycle is at $T_0$. Relabelling $T_r$ by $m-r$
gives a graded isomorphism with $K^m_1$.

Every simple has a linear BGG resolution:
\begin{equation}\label{eq:km1-bgg-resolution}
 0\longrightarrow\Delta_0\langle r\rangle
 \longrightarrow\Delta_1\langle r-1\rangle
 \longrightarrow\cdots\longrightarrow\Delta_r
 \longrightarrow L_r\longrightarrow0
 \qquad(1\leq r\leq m),
\end{equation}
and $L_0=\Delta_0$. This follows by splicing the sequences
$0\to L_{r-1}\langle1\rangle\to\Delta_r\to L_r\to0$.

For comparison, the projective resolutions of the costandards
are generally not linear. The diagrams give
\begin{equation}\label{eq:km1-costandard-sequence}
 0\longrightarrow\nabla_{r-1}\langle1\rangle
 \longrightarrow P_{r-1}\langle-1\rangle
 \longrightarrow\nabla_r\longrightarrow0
 \qquad(1\leq r\leq m).
\end{equation}
Together with the resolution of $\nabla_0=\Delta_0$, these
sequences determine the minimal projective resolution of
$\nabla_r$. For $r\geq1$, its term in homological degree $j$ is
\begin{equation}\label{eq:km1-costandard-resolution}
 Q_j^{(r)}=
 \begin{cases}
 P_{r-1-j}\langle j-1\rangle,&0\leq j<r,\\
 P_{j-r}\langle j\rangle,&r\leq j\leq m+r.
 \end{cases}
\end{equation}
Thus the resolution is
$0\to Q_{m+r}^{(r)}\to\cdots\to Q_0^{(r)}\to\nabla_r\to0$.
The map between the two copies of $P_0$ is right multiplication
by $y_0x_0$; it accounts for the jump of two in the grading shifts.
All differentials have image in the radical, so the resolutions
above are minimal. In particular,
\[
 \operatorname{pd}\Delta_r=m-r,\qquad
 \operatorname{pd}\nabla_r=m+r
 \qquad(0\leq r\leq m).
\]

\paragraph{The corner algebra.}
Set $e=\sum_{r=0}^{m-1}e_r$. The corresponding Khovanov arc
algebra is $H^m_1=eK^m_1e$.
For $m\geq2$, its quiver is the doubled line
\[
 0
 \mathrel{\mathop{\rightleftarrows}^{x_0}_{y_0}}
 1
 \mathrel{\mathop{\rightleftarrows}^{x_1}_{y_1}}
 \cdots
 \mathrel{\mathop{\rightleftarrows}^{x_{m-2}}_{y_{m-2}}}
 m-1,
\]
with relations
\begin{equation}\label{eq:hm1-relations}
\begin{aligned}
 x_{r+1}x_r=y_ry_{r+1}&=0 &&(0\leq r\leq m-3),\\
 y_rx_r&=x_{r-1}y_{r-1} &&(1\leq r\leq m-2),
\end{aligned}
\end{equation}
and all paths of length three are zero. The last condition
follows from the quadratic relations when $m\geq3$, but
must be imposed separately when $m=2$.
There is a nonzero length-two cycle at every vertex.
Thus $H^m_1$ is the zigzag algebra of type $A_m$.

Equivalently, as an ungraded algebra,
\[
 H^m_1\cong
 T\!\left(\Bbbk Q/\operatorname{rad}^2(\Bbbk Q)\right),
\]
where $Q$ is any orientation of $A_m$ and
$T(B)=B\oplus\operatorname{Hom}_{\Bbbk}(B,\Bbbk)$ is the
trivial extension. In particular, $H^m_1$ is symmetric.
For $m=1$, one has $H^1_1\cong\Bbbk[c]/(c^2)$ with
$\deg c=2$.

\paragraph{The Koszul dual.}
The quadratic dual has the doubled-line quiver
\[
 0
 \mathrel{\mathop{\rightleftarrows}^{a_0}_{b_0}}
 1
 \mathrel{\mathop{\rightleftarrows}^{a_1}_{b_1}}
 \cdots
 \mathrel{\mathop{\rightleftarrows}^{a_{m-1}}_{b_{m-1}}}
 m,
\]
with arrows
$a_r:r\to r+1$ and $b_r:r+1\to r$, and relations
\begin{equation}\label{eq:km1-koszul-dual}
 b_0a_0=0,\qquad
 b_ra_r+a_{r-1}b_{r-1}=0\quad(1\leq r<m).
\end{equation}
There is no relation killing the cycle at vertex $m$.
Replacing $b_r$ by $(-1)^rb_r$ changes the interior relations
to commutativity relations. These are the standard relations
for the Auslander algebra of $R=\Bbbk[t]/(t^{m+1})$;
see \cite[\S3]{IyamaZhang2020}. Hence
\[
 (K^m_1)^!\cong
 \operatorname{End}_{R}
 \left(\bigoplus_{j=1}^{m+1}R/(t^j)\right)^{\operatorname{op}}.
\]
Here vertex $r$ corresponds to $R/(t^{r+1})$; the arrows are
represented by the natural projections and multiplication by $t$.
For $m=1$, reversing the two vertices gives
$(K^1_1)^!\cong K^1_1$.
Finally, imposing the additional relation $a_{m-1}b_{m-1}=0$
gives, after the same rescaling, the preprojective algebra of
type $A_{m+1}$; see also \cite[\S5]{IyamaZhang2020}.

\subsection{$K^\infty_1$}\label{subsec:defect-one-infinite}
% Replacement for the body of the K_1^infty subsection.
% Uses the left-module and grading conventions of the main text.
% All citation keys below are already present in refs.bib.

We use locally unital left modules and the grading convention
$(M\langle d\rangle)_j=M_{j-d}$.
The algebra $K_1^\infty$ is the path algebra of the doubly infinite
quiver
\[
\cdots
\mathrel{\mathop{\rightleftarrows}^{x_{-2}}_{y_{-2}}}
-1
\mathrel{\mathop{\rightleftarrows}^{x_{-1}}_{y_{-1}}}
0
\mathrel{\mathop{\rightleftarrows}^{x_0}_{y_0}}
1
\mathrel{\mathop{\rightleftarrows}^{x_1}_{y_1}}
2
\mathrel{\mathop{\rightleftarrows}^{x_2}_{y_2}}
\cdots,
\]
where $x_r:r\to r+1$ and $y_r:r+1\to r$ for $r\in\mathbb Z$,
modulo the relations
\begin{equation}\label{eq:kinf1-relations}
 x_{r+1}x_r=0,\qquad
 y_ry_{r+1}=0,\qquad
 y_rx_r=x_{r-1}y_{r-1}
 \qquad(r\in\mathbb Z).
\end{equation}
Paths are composed from right to left, and every arrow has
degree one. All paths of length three vanish. A basis consists
of the vertex idempotents, the arrows and the cycles $y_rx_r$.
In particular, both $K_1^\infty e_r$ and $e_rK_1^\infty$ have
dimension four, so $K_1^\infty$ is an essentially
finite-dimensional locally unital algebra.

The category $\operatorname{rep}(K_1^\infty)$ of finite-dimensional
ungraded left modules is an essentially finite highest weight
category with weight poset $(\mathbb Z,\leq)$.
Write $L_r$, $P_r$, $I_r$, $\Delta_r$, $\nabla_r$ and $T_r$
for its simple,
indecomposable projective, indecomposable injective, standard,
costandard and indecomposable tilting modules, respectively,
and use the same symbols for their graded lifts.
We place $L_r$ in degree zero, take $P_r=K_1^\infty e_r$,
and give $\Delta_r$ its quotient grading.

The degree-preserving anti-involution defined by
$e_r^*=e_r$, $x_r^*=y_r$ and $y_r^*=x_r$ induces the duality
\[
 (M^\circledast)_j=\operatorname{Hom}_{\Bbbk}(M_{-j},\Bbbk),
 \qquad (a\cdot f)(v)=f(a^*v).
\]
We normalize $I_r=P_r^\circledast$ and
$\nabla_r=\Delta_r^\circledast$, and normalize $T_r$ by a
degree-zero inclusion $\Delta_r\hookrightarrow T_r$.
Thus the heads of the standards and the socles of the
costandards lie in degree zero.

The following Alperin diagrams omit the gradings:
\[
\Delta_r=
\begin{tikzpicture}[baseline=(current bounding box.center),scale=0.9]
\node (a) at (0,1) {$L_r$};
\node (b) at (0,0) {$L_{r-1}$};
\draw (a)--(b);
\end{tikzpicture},
\qquad
\nabla_r=
\begin{tikzpicture}[baseline=(current bounding box.center),scale=0.9]
\node (a) at (0,1) {$L_{r-1}$};
\node (b) at (0,0) {$L_r$};
\draw (a)--(b);
\end{tikzpicture},
\qquad
P_r=
\begin{tikzpicture}[baseline=(current bounding box.center),scale=0.75]
\node (a) at (0,2) {$L_r$};
\node (b) at (-1,1) {$L_{r-1}$};
\node (c) at (1,1) {$L_{r+1}$};
\node (d) at (0,0) {$L_r$};
\draw (a)--(b)--(d);
\draw (a)--(c)--(d);
\end{tikzpicture}
\qquad(r\in\mathbb Z).
\]
The three layers of $P_r$ have degrees $0,1,2$.
The graded identifications are
\begin{equation}\label{eq:kinf1-duality-and-tiltings}
 L_r^\circledast\cong L_r,\qquad
 P_r^\circledast=I_r\cong P_r\langle-2\rangle,\qquad
 T_{r+1}\cong P_r\langle-1\rangle.
\end{equation}
In particular, every $T_r$ is self-dual. After forgetting the
grading, $T_{r+1}\cong I_r\cong P_r$; hence projective and
injective objects in $\operatorname{rep}(K_1^\infty)$ coincide.

The projective and tilting filtrations give short exact sequences
\begin{equation}\label{eq:kinf1-standard-sequences}
\begin{aligned}
 0&\longrightarrow\Delta_{r+1}\langle1\rangle
 \longrightarrow P_r\longrightarrow\Delta_r\longrightarrow0,\\
 0&\longrightarrow\Delta_r\longrightarrow T_r
 \longrightarrow\Delta_{r-1}\langle-1\rangle\longrightarrow0.
\end{aligned}
\end{equation}
The first sequence verifies the required standard filtration
of each projective. Splicing the first family gives the minimal
linear projective resolution
\begin{equation}\label{eq:kinf1-standard-resolution}
 \cdots\longrightarrow P_{r+3}\langle3\rangle
 \longrightarrow P_{r+2}\langle2\rangle
 \longrightarrow P_{r+1}\langle1\rangle
 \longrightarrow P_r\longrightarrow\Delta_r\longrightarrow0.
\end{equation}
For $j\geq1$, the differential from $P_{r+j}\langle j\rangle$
is right multiplication by $x_{r+j-1}$.
Applying $\circledast$ gives a linear injective coresolution
of $\nabla_r$. Thus $K_1^\infty$ is standard Koszul.
The displayed minimal resolution never terminates, so
$\operatorname{pd}\Delta_r=\infty$ and $K_1^\infty$ has
infinite global dimension.

Splicing the sequences
$0\to L_{r-1}\langle1\rangle\to\Delta_r\to L_r\to0$
gives a linear BGG resolution
\begin{equation}\label{eq:kinf1-bgg-resolution}
 \cdots\longrightarrow\Delta_{r-2}\langle2\rangle
 \longrightarrow\Delta_{r-1}\langle1\rangle
 \longrightarrow\Delta_r\longrightarrow L_r\longrightarrow0.
\end{equation}
Its dual is a linear coresolution by costandard modules.
There is also an exact coresolution by standard modules:
\begin{equation}\label{eq:kinf1-standard-coresolution}
 0\longrightarrow L_r
 \longrightarrow\Delta_{r+1}\langle-1\rangle
 \longrightarrow\Delta_{r+2}\langle-2\rangle
 \longrightarrow\Delta_{r+3}\langle-3\rangle
 \longrightarrow\cdots.
\end{equation}
This is obtained by splicing
$0\to L_r\to\Delta_{r+1}\langle-1\rangle
\to L_{r+1}\langle-1\rangle\to0$.

Coulembier proved uniqueness, up to equivalence, of upper finite
and lower finite highest weight structures in the presence of a
simple-preserving duality
\cite{coulembier2020classification,coulembier2020indcompletions}.
The present example illustrates the failure of this uniqueness
in the essentially finite setting; see also
\cite[Remark~3.68 and \S6.7]{brundanstroppelsemiinfinite}.
Indeed, the orders $(\mathbb Z,\leq)$ and $(\mathbb Z,\geq)$
give two distinct highest weight structures on the fixed category
$\operatorname{rep}(K_1^\infty)$, despite its
simple-preserving duality. For each $r$, the corresponding
standard modules have head $L_r$ and socles $L_{r-1}$ and
$L_{r+1}$, respectively. Neither order is upper finite or
lower finite. Here the simple labels are held fixed: reflection
$r\mapsto-r$ interchanges the two structures by an
autoequivalence of the underlying category.

These two structures are Ringel dual to each other.
Indeed, using all homogeneous maps, the locally unital
Ringel dual algebra is
\begin{equation}\label{eq:kinf1-ringel-dual}
 \left(\bigoplus_{r,s\in\mathbb Z}
 \operatorname{Hom}_{K_1^\infty}(T_r,T_s)\right)^{\operatorname{op}}
 \cong K_1^\infty.
\end{equation}
This follows from $T_r\cong P_{r-1}\langle-1\rangle$:
the idempotent corresponding to $T_r$ is identified with
$e_{r-1}$, and the reversed order on the tilting labels
becomes the opposite order on the vertices.

The algebra $K_1^\infty$ is Koszul by
\cite[Corollary~5.13]{brundanstroppelII}.
Its Koszul dual $(K_1^\infty)^!$ is the path algebra of the
doubly infinite quiver
\[
\cdots
\mathrel{\mathop{\rightleftarrows}^{a_{-2}}_{b_{-2}}}
-1
\mathrel{\mathop{\rightleftarrows}^{a_{-1}}_{b_{-1}}}
0
\mathrel{\mathop{\rightleftarrows}^{a_0}_{b_0}}
1
\mathrel{\mathop{\rightleftarrows}^{a_1}_{b_1}}
2
\mathrel{\mathop{\rightleftarrows}^{a_2}_{b_2}}
\cdots,
\]
where $a_r:r\to r+1$ and $b_r:r+1\to r$ are dual to
$y_r$ and $x_r$, respectively, modulo the relations
\begin{equation}\label{eq:kinf1-koszul-dual-relations}
 b_ra_r+a_{r-1}b_{r-1}=0
 \qquad(r\in\mathbb Z).
\end{equation}
There are no quadratic relations on $a_{r+1}a_r$ or
$b_rb_{r+1}$. Indeed, the orthogonal complement of the
relation $y_rx_r-x_{r-1}y_{r-1}$ is spanned by the sum in
\eqref{eq:kinf1-koszul-dual-relations}, whereas each straight
length-two path lies in the original relation space.
Rescaling $b_r$ by $(-1)^r$ gives the usual preprojective
relations for the doubly infinite line.

\subsection{$K^{-\infty}_1$}\label{subsec:defect-one-minus-infinite}
We use locally unital left modules,
compose paths from right to left, and give every arrow degree one.
The quiver of $K_1^{-\infty}$ is
\[
 \cdots
 \mathrel{\mathop{\leftrightarrows}^{x_3}_{y_3}}3
 \mathrel{\mathop{\leftrightarrows}^{x_2}_{y_2}}2
 \mathrel{\mathop{\leftrightarrows}^{x_1}_{y_1}}1
 \mathrel{\mathop{\leftrightarrows}^{x_0}_{y_0}}0,
\]
where $x_r:r\to r+1$ and $y_r:r+1\to r$ for $r\geq0$.
The relations are
\begin{equation}\label{eq:kminf1-relations}
\begin{aligned}
 x_{r+1}x_r=y_ry_{r+1}&=0 &&(r\geq0),\\
 y_rx_r&=x_{r-1}y_{r-1} &&(r\geq1),\\
 y_0x_0&=0.
\end{aligned}
\end{equation}
All paths of length three vanish.
Thus $K_1^{-\infty}e_0$ has dimension two
and $K_1^{-\infty}e_r$ has dimension four for $r\geq1$.

The category $\operatorname{rep}(K_1^{-\infty})$ of
finite-dimensional modules is essentially finite highest weight
for the order $\cdots<3<2<1<0$.
This order is upper finite; the category
$K_1^{-\infty}\text{-mod}_{\mathrm{lfd}}$ of locally
finite-dimensional modules is an upper finite highest weight
category.
Use the notation $L_r$, $P_r$, $I_r$, $\Delta_r$, $\nabla_r$
and $T_r$ as in the preceding examples.
We take $P_r=K_1^{-\infty}e_r$, put $L_r$ in degree zero,
and normalize $\Delta_r$ by its head and $\nabla_r$ by its
socle in degree zero.
The anti-involution $e_r^*=e_r$, $x_r^*=y_r$, $y_r^*=x_r$
defines the graded duality $\circledast$, with
$I_r=P_r^\circledast$ and $\nabla_r=\Delta_r^\circledast$.

The Alperin diagrams, with gradings omitted, are
\[
 \Delta_r=
 \begin{tikzpicture}[baseline=(current bounding box.center),scale=.7,every node/.style={inner sep=1pt}]
 \node (a) at (0,1) {$L_r$};
 \node (b) at (0,0) {$L_{r+1}$};
 \draw (a)--(b);
 \end{tikzpicture},\qquad
 \nabla_r=
 \begin{tikzpicture}[baseline=(current bounding box.center),scale=.7,every node/.style={inner sep=1pt}]
 \node (a) at (0,1) {$L_{r+1}$};
 \node (b) at (0,0) {$L_r$};
 \draw (a)--(b);
 \end{tikzpicture}
 \quad(r\geq0),\qquad
 P_r=
 \begin{tikzpicture}[baseline=(current bounding box.center),scale=.7,every node/.style={inner sep=1pt}]
 \node (a) at (0,2) {$L_r$};
 \node (b) at (-1,1) {$L_{r-1}$};
 \node (c) at (1,1) {$L_{r+1}$};
 \node (d) at (0,0) {$L_r$};
 \draw (a)--(b)--(d);
 \draw (a)--(c)--(d);
 \end{tikzpicture}
 \quad(r\geq1).
\]
At the boundary, $P_0=\Delta_0$ and $I_0=\nabla_0$.
The layers of $P_r$ have degrees $0,1,2$ for $r\geq1$.
With the normalization $\Delta_r\hookrightarrow T_r$ of degree
zero, the graded identifications are
\begin{equation}\label{eq:kminf1-injectives-tiltings}
 I_r\cong P_r\langle-2\rangle\quad(r\geq1),
 \qquad T_r\cong P_{r+1}\langle-1\rangle\quad(r\geq0).
\end{equation}
In particular, the tiltings are self-dual.

For $r\geq1$, the standard module has minimal linear projective
resolution
\begin{equation}\label{eq:kminf1-standard-resolution}
 0\longrightarrow P_0\langle r\rangle
 \longrightarrow P_1\langle r-1\rangle
 \longrightarrow\cdots\longrightarrow P_r
 \longrightarrow\Delta_r\longrightarrow0.
\end{equation}
Hence $\operatorname{pd}\Delta_r=r$ for all $r\geq0$.
Duality gives linear injective coresolutions of the costandards,
so $K_1^{-\infty}$ is standard Koszul; its Koszulity also follows
from \cite[Corollary~5.13]{brundanstroppelII}.
The costandards have minimal projective resolutions
\begin{equation}\label{eq:kminf1-costandard-resolution}
 \cdots\longrightarrow P_{r+3}\langle1\rangle
 \longrightarrow P_{r+2}
 \longrightarrow P_{r+1}\langle-1\rangle
 \longrightarrow\nabla_r\longrightarrow0.
\end{equation}
Thus $\nabla_r\langle1\rangle$ has an infinite linear
projective resolution and $\operatorname{pd}\nabla_r=\infty$.
Every simple has a linear BGG resolution
\begin{equation}\label{eq:kminf1-bgg-resolution}
 \cdots\longrightarrow\Delta_{r+2}\langle2\rangle
 \longrightarrow\Delta_{r+1}\langle1\rangle
 \longrightarrow\Delta_r\longrightarrow L_r\longrightarrow0.
\end{equation}

The locally unital corner
$H_1^{-\infty}:=\bigoplus_{r,s\geq1}e_rK_1^{-\infty}e_s$
has quiver
\[
 \cdots
 \mathrel{\mathop{\leftrightarrows}^{x_3}_{y_3}}3
 \mathrel{\mathop{\leftrightarrows}^{x_2}_{y_2}}2
 \mathrel{\mathop{\leftrightarrows}^{x_1}_{y_1}}1,
\]
with the straight-path relations from \eqref{eq:kminf1-relations}
and the cycle relations for $r\geq2$, but no zero boundary cycle.
Since $T_r\cong P_{r+1}\langle-1\rangle$, the algebra
$\bigl(\bigoplus_{r,s\geq0}
\operatorname{Hom}_{K_1^{-\infty}}(T_r,T_s)\bigr)^{\operatorname{op}}$,
using all homogeneous maps, is $H_1^{-\infty}$.

\paragraph{The Koszul dual.}
The Koszul dual $(K_1^{-\infty})^!$ has quiver
\[
 \cdots
 \mathrel{\mathop{\leftrightarrows}^{a_3}_{b_3}}3
 \mathrel{\mathop{\leftrightarrows}^{a_2}_{b_2}}2
 \mathrel{\mathop{\leftrightarrows}^{a_1}_{b_1}}1
 \mathrel{\mathop{\leftrightarrows}^{a_0}_{b_0}}0,
\]
where $a_r:r\to r+1$ and $b_r:r+1\to r$ are dual to $y_r$
and $x_r$, respectively.
Its only defining relations are
\begin{equation}\label{eq:kminf1-dual-relations}
 b_ra_r+a_{r-1}b_{r-1}=0\qquad(r\geq1).
\end{equation}
In particular, there is no relation $b_0a_0=0$, nor any
straight-path relation.
The corner $e_0(K_1^{-\infty})^!e_0$ is the polynomial algebra
generated by $b_0a_0$.

The appropriate highest weight category is
$\operatorname{rep}\bigl((K_1^{-\infty})^!\bigr)$,
the category of finite-dimensional locally unital left
$(K_1^{-\infty})^!$-modules, with order $0<1<2<\cdots$.
For $N\geq0$, its Serre subcategory supported on $\{0,\ldots,N\}$
is the module category of
\begin{equation}\label{eq:kminf1-dual-truncation}
 (K_1^{-\infty})^!/
 \langle e_{N+1},e_{N+2},\ldots\rangle.
\end{equation}
For $N\geq1$ the additional boundary relation is
$a_{N-1}b_{N-1}=0$.
After rescaling arrows, this quotient is the Auslander algebra
of $\Bbbk[t]/(t^{N+1})$, with vertex $r$ corresponding to
$\Bbbk[t]/(t^{N+1-r})$; see \cite[\S3]{IyamaZhang2020}.
It is quasi-hereditary for $0<\cdots<N$.
Consequently,
$\operatorname{rep}\bigl((K_1^{-\infty})^!\bigr)$
is lower finite highest weight in the sense of
\cite[Definition~3.50]{brundanstroppelsemiinfinite}.

Write $S_r$, $\Delta_r^!$, $\nabla_r^!$ and $T_r^!$
for its simple, standard, costandard and indecomposable
tilting objects.
The anti-involution interchanging $a_r$ and $b_r$ gives the same
graded duality convention as above.
We have $\Delta_0^!=\nabla_0^!=T_0^!=S_0$, while
\[
 \Delta_r^!=
 \begin{tikzpicture}[baseline=(current bounding box.center),scale=.7,every node/.style={inner sep=1pt}]
 \node (a) at (0,2) {$S_r$};
 \node (b) at (0,1) {$\vdots$};
 \node (c) at (0,0) {$S_0$};
 \draw (a)--(b)--(c);
 \end{tikzpicture},\qquad
 \nabla_r^!=
 \begin{tikzpicture}[baseline=(current bounding box.center),scale=.7,every node/.style={inner sep=1pt}]
 \node (a) at (0,2) {$S_0$};
 \node (b) at (0,1) {$\vdots$};
 \node (c) at (0,0) {$S_r$};
 \draw (a)--(b)--(c);
 \end{tikzpicture}
 \qquad(r\geq1).
\]
These are uniserial chains through all the indicated consecutive
labels.
The factors of $\Delta_r^!$ from head to socle are
$S_r,S_{r-1},\ldots,S_0$, in degrees $0,1,\ldots,r$;
the costandard has the dual grading.

The category $\operatorname{rep}\bigl((K_1^{-\infty})^!\bigr)$
has no nonzero projective or injective objects.

\subsection{$K^{+\infty}_1$}\label{subsec:defect-one-plus-infinite}
We use locally unital left modules, compose paths from right
to left, and give every arrow degree one.
The algebra $K_1^{+\infty}$ has quiver
\[
 0
 \mathrel{\mathop{\rightleftarrows}^{x_0}_{y_0}}1
 \mathrel{\mathop{\rightleftarrows}^{x_1}_{y_1}}2
 \mathrel{\mathop{\rightleftarrows}^{x_2}_{y_2}}3
 \mathrel{\mathop{\rightleftarrows}^{x_3}_{y_3}}\cdots,
\]
where $x_r:r\to r+1$ and $y_r:r+1\to r$ for $r\geq0$,
with relations
\begin{equation}\label{eq:kplusinf1-relations}
\begin{aligned}
 x_{r+1}x_r=y_ry_{r+1}&=0 &&(r\geq0),\\
 y_rx_r&=x_{r-1}y_{r-1} &&(r\geq1).
\end{aligned}
\end{equation}
There is no relation $y_0x_0=0$.
All paths of length three vanish, so $K_1^{+\infty}e_0$
has dimension three and $K_1^{+\infty}e_r$ has dimension
four for $r\geq1$.

The category $\operatorname{rep}(K_1^{+\infty})$ of
finite-dimensional modules is both lower finite and
essentially finite highest weight for the order
$0<1<2<\cdots$.
Use the notation $L_r$, $P_r$, $I_r$, $\Delta_r$, $\nabla_r$
and $T_r$ as in the preceding examples.
We take $P_r=K_1^{+\infty}e_r$, put $L_r$ in degree zero,
and normalize $\Delta_r$ by its head and $\nabla_r$ by its
socle in degree zero.
The anti-involution $e_r^*=e_r$, $x_r^*=y_r$, $y_r^*=x_r$
defines the graded duality $\circledast$, with
$I_r=P_r^\circledast$ and $\nabla_r=\Delta_r^\circledast$.

The Alperin diagrams, with gradings omitted, are
\[
 \Delta_0=\nabla_0=T_0=L_0,\qquad
 \Delta_r=
 \begin{tikzpicture}[baseline=(current bounding box.center),scale=.7,every node/.style={inner sep=1pt}]
 \node (a) at (0,1) {$L_r$};
 \node (b) at (0,0) {$L_{r-1}$};
 \draw (a)--(b);
 \end{tikzpicture},\qquad
 \nabla_r=
 \begin{tikzpicture}[baseline=(current bounding box.center),scale=.7,every node/.style={inner sep=1pt}]
 \node (a) at (0,1) {$L_{r-1}$};
 \node (b) at (0,0) {$L_r$};
 \draw (a)--(b);
 \end{tikzpicture}
 \quad(r\geq1),
\]
and
\[
 P_0=
 \begin{tikzpicture}[baseline=(current bounding box.center),scale=.7,every node/.style={inner sep=1pt}]
 \node (a) at (0,2) {$L_0$};
 \node (b) at (0,1) {$L_1$};
 \node (c) at (0,0) {$L_0$};
 \draw (a)--(b)--(c);
 \end{tikzpicture},\qquad
 P_r=
 \begin{tikzpicture}[baseline=(current bounding box.center),scale=.7,every node/.style={inner sep=1pt}]
 \node (a) at (0,2) {$L_r$};
 \node (b) at (-1,1) {$L_{r-1}$};
 \node (c) at (1,1) {$L_{r+1}$};
 \node (d) at (0,0) {$L_r$};
 \draw (a)--(b)--(d);
 \draw (a)--(c)--(d);
 \end{tikzpicture}
 \quad(r\geq1).
\]
The layers of each $P_r$ have degrees $0,1,2$.
With the normalization $\Delta_r\hookrightarrow T_r$ of
degree zero, the graded identifications are
\begin{equation}\label{eq:kplusinf1-injectives-tiltings}
 I_r\cong P_r\langle-2\rangle,\qquad
 T_{r+1}\cong P_r\langle-1\rangle
 \qquad(r\geq0).
\end{equation}
Thus $K_1^{+\infty}$ is self-injective and all the tiltings
are self-dual.

With these weight orders, $K_1^{+\infty}$ and
$K_1^{-\infty}$ are mutually Ringel dual.
Relabelling vertex $r\geq1$ by $r-1$ gives a graded
isomorphism
$H_1^{-\infty}=\bigoplus_{r,s\geq1}
e_rK_1^{-\infty}e_s\cong K_1^{+\infty}$.

For every $r\geq0$, the standard module has minimal linear
projective resolution
\begin{equation}\label{eq:kplusinf1-standard-resolution}
 \cdots\longrightarrow P_{r+2}\langle2\rangle
 \longrightarrow P_{r+1}\langle1\rangle
 \longrightarrow P_r
 \longrightarrow\Delta_r\longrightarrow0.
\end{equation}
Hence $\operatorname{pd}\Delta_r=\infty$ for all $r\geq0$.
Duality gives linear injective coresolutions of the
costandards, so $K_1^{+\infty}$ is standard Koszul.
For $r\geq1$, the simple module has a finite linear BGG
resolution
\begin{equation}\label{eq:kplusinf1-bgg-resolution}
 0\longrightarrow\Delta_0\langle r\rangle
 \longrightarrow\Delta_1\langle r-1\rangle
 \longrightarrow\cdots\longrightarrow\Delta_r
 \longrightarrow L_r\longrightarrow0,
\end{equation}
while $L_0=\Delta_0$.

\paragraph{The Koszul dual.}
The Koszul dual $(K_1^{+\infty})^!$ has quiver
\[
 0
 \mathrel{\mathop{\rightleftarrows}^{a_0}_{b_0}}1
 \mathrel{\mathop{\rightleftarrows}^{a_1}_{b_1}}2
 \mathrel{\mathop{\rightleftarrows}^{a_2}_{b_2}}3
 \mathrel{\mathop{\rightleftarrows}^{a_3}_{b_3}}\cdots,
\]
where $a_r:r\to r+1$ and $b_r:r+1\to r$ are dual to
$y_r$ and $x_r$, respectively.
Its only defining relations are
\begin{equation}\label{eq:kplusinf1-dual-relations}
 b_0a_0=0,\qquad
 b_ra_r+a_{r-1}b_{r-1}=0\quad(r\geq1).
\end{equation}
There are no straight-path relations.
The algebra is locally finite-dimensional, with
$\dim e_s(K_1^{+\infty})^!e_r=\min(r,s)+1$.

The appropriate highest weight category is
$(K_1^{+\infty})^!\text{-mod}_{\mathrm{lfd}}$,
the category of locally finite-dimensional locally unital
left modules, with order $\cdots<2<1<0$.
It is upper finite highest weight in the sense of
\cite[Definition~3.34]{brundanstroppelsemiinfinite}.
Write $S_r$, $P_r^!$, $I_r^!$, $\Delta_r^!$, $\nabla_r^!$
and $T_r^!$ for its distinguished objects, with
$P_r^!=(K_1^{+\infty})^!e_r$.
The anti-involution interchanging $a_r$ and $b_r$ defines
vertexwise graded duality, so
$I_r^!=(P_r^!)^\circledast$ and
$\nabla_r^!=(\Delta_r^!)^\circledast$.

The standard and costandard modules are represented
schematically by the infinite chains
\[
 \Delta_r^!=
 \begin{tikzpicture}[baseline=(current bounding box.center),scale=.7,every node/.style={inner sep=1pt}]
 \node (a) at (0,3) {$S_r$};
 \node (b) at (0,2) {$S_{r+1}$};
 \node (c) at (0,1) {$S_{r+2}$};
 \node (d) at (0,0) {$\vdots$};
 \draw (a)--(b)--(c)--(d);
 \end{tikzpicture},\qquad
 \nabla_r^!=
 \begin{tikzpicture}[baseline=(current bounding box.center),scale=.7,every node/.style={inner sep=1pt}]
 \node (a) at (0,3) {$\vdots$};
 \node (b) at (0,2) {$S_{r+2}$};
 \node (c) at (0,1) {$S_{r+1}$};
 \node (d) at (0,0) {$S_r$};
 \draw (a)--(b)--(c)--(d);
 \end{tikzpicture}
 \qquad(r\geq0).
\]
The successive radical layers of $\Delta_r^!$ are
$S_r,S_{r+1},\ldots$, in degrees $0,1,\ldots$;
the costandard has the dual socle filtration and grading.
In particular, these objects are infinite-dimensional,
which necessitates working with locally finite-dimensional
modules.

We have $P_0^!=\Delta_0^!$, and for $r\geq1$ there is a
minimal linear projective resolution
\begin{equation}\label{eq:kplusinf1-dual-standard-resolution}
 0\longrightarrow P_{r-1}^!\langle1\rangle
 \longrightarrow P_r^!
 \longrightarrow\Delta_r^!\longrightarrow0.
\end{equation}
Thus $P_r^!$ has a finite standard filtration with sections
$\Delta_r^!,\Delta_{r-1}^!\langle1\rangle,\ldots,
\Delta_0^!\langle r\rangle$.
Duality gives the corresponding costandard filtrations
of the injectives; these filtrations verify the upper
finite highest weight structure.
Moreover, $\operatorname{pd}\Delta_0^!=0$ and
$\operatorname{pd}\Delta_r^!=1$ for $r\geq1$.

Normalize $T_r^!$ by requiring
$\Delta_r^!\hookrightarrow T_r^!$ to have degree zero.
It is self-dual and has an ascending standard filtration
with sections $\Delta_s^!\langle r-s\rangle$ for
$s\geq r$, each occurring once, together with the dual
descending costandard filtration.
In particular,
$\dim e_sT_r^!=\max\{s-r+1,0\}$,
so all these tiltings are infinite-dimensional.

The lower finite category
$\operatorname{rep}\bigl((K_1^{-\infty})^!\bigr)$ and the
upper finite category
$(K_1^{+\infty})^!\text{-mod}_{\mathrm{lfd}}$
are mutually Ringel dual, with opposite weight orders,
in the sense of
\cite[Theorems~4.25 and~4.27]{brundanstroppelsemiinfinite}.
Indeed, the direct sum of all homogeneous Hom-spaces
between the normalized indecomposable tiltings on the
lower finite side, with opposite composition, is
$(K_1^{+\infty})^!$.

\bibliographystyle{alpha}
\bibliography{refs}

\bigskip

\noindent
\textsc{Shunsuke Hirota} \\
\textsc{Department of Mathematics, Kyoto University} \\
Kitashirakawa Oiwake-cho, Sakyo-ku, 606-8502, Kyoto \\
\textit{E-mail address}:
{shun299509732@gmail.com}

\end{document}